\documentclass{article}
\usepackage[utf8]{inputenc}
\usepackage[T1]{fontenc}
\usepackage[latin9]{luainputenc}
\usepackage{geometry}
\usepackage{babel}
\usepackage{amsmath}
\usepackage{amsthm}
\usepackage{amssymb}
\usepackage[scr=esstix]
       {mathalpha}
\usepackage[unicode=true,pdfusetitle,
 bookmarks=true,bookmarksnumbered=false,bookmarksopen=false,
 breaklinks=false,pdfborder={0 0 1},backref=false,colorlinks=false]
 {hyperref}
\usepackage{dsfont,graphicx}
\usepackage[dvipsnames]{xcolor}
\usepackage{esint}

\makeatletter
\numberwithin{equation}{section}
\numberwithin{figure}{section}
\theoremstyle{plain}
\newtheorem{thm}{\protect\theoremname}[section]
\theoremstyle{plain}
\newtheorem{prop}[thm]{\protect\propositionname}
\theoremstyle{plain}
\newtheorem{cor}[thm]{\protect\corollaryname}
\theoremstyle{plain}
\newtheorem{lem}[thm]{\protect\lemmaname}
\theoremstyle{definition}
\newtheorem{defn}[thm]{\protect\definitionname}
\theoremstyle{remark}
\newtheorem{rem}[thm]{\protect\remarkname}

\@ifundefined{date}{}{\date{}}

\def\vep{\varepsilon}
\def\vfi{\varphi}

\def\be{\begin{equation}}
\def\ee{\end{equation}}

\def\intr{\int_{\R}}

\def\cP{{\mathcal P}}
\def\cE{{\mathcal E}}

\def\cE{{\mathcal E}}
\def\cEw{{\mathcal E}_{w,\mu}}

\usepackage[nottoc,notlot,notlof]{tocbibind}

\makeatother

\providecommand{\corollaryname}{Corollary}
\providecommand{\definitionname}{Definition}
\providecommand{\lemmaname}{Lemma}
\providecommand{\propositionname}{Proposition}
\providecommand{\remarkname}{Remark}
\providecommand{\theoremname}{Theorem}

\global\long\def\R{\mathbb{R}}%

\global\long\def\al{\alpha}%

\global\long\def\et{\eta}%

\global\long\def\ta{\tau}%

\global\long\def\util{\tilde{u}}%
\global\long\def\mtil{\tilde{m}}%
\title{Self-similar intermediate asymptotics for first-order mean field games}
\author{Sebastian Munoz}

\begin{document}
\maketitle
\begin{abstract}
    We study the intermediate asymptotic behavior of solutions to the first-order mean field games system with a local coupling, when the initial density is a compactly supported function on the real line, and the coupling is of power type. Addressing a question that was left open in \cite{CMP}, we prove that the solutions converge to the self-similar profile. We proceed by analyzing a continuous rescaling of the solution, and identifying an appropriate Lyapunov functional. We identify a critical value for the parameter of the coupling, which determines the qualitative behavior of the functional, and the well-posedness of the infinite horizon system. Accordingly, we also establish, in the subcritical and critical cases, a second convergence result which characterizes the behavior of the full solution as the time horizon approaches infinity. We also prove the corresponding results for the mean field planning problem. A large part of our analysis and methodology apply just as well to arbitrary dimensions. As such, this work is a major step towards settling these questions in the higher-dimensional setting.
\end{abstract}
\tableofcontents{}
\section{Introduction}
The goal of this paper is to determine the asymptotic behavior of the solutions to the first-order mean field games system (MFG), expressed as:
\begin{equation} \tag{MFG}
\begin{cases}
-u_{t}+\frac{1}{2}u_{x}^{2}=m^{\theta} & (x,t)\in\R\times(0,T),\\
m_{t}-(mu_{x})_{x}=0 & (x,t)\in\R\times(0,T).
\\
m(x,0)=m_{0}(x) & x\in\R,
\end{cases}\label{mfg}
\end{equation}
where $\theta>0$, and $m_0: \R \to [0,\infty)$ is a continuous, compactly supported function.

 From a modeling perspective, the solution $(u,m)$ can be interpreted as the Nash equilibrium of a differential game with a large number of agents. The function $m:\R \times [0,T] \to [0,\infty) $ may be interpreted as the density of players at time $t$ and position $x$, whereas $u(x,t)$ is the optimal cost for a generic player at time $t$, at the position $x$, that is, 
\be \label{urepform}
u(x,t)=\inf_{\beta \in H^1((t,T)), \,\beta(t)=x} \quad\int_t^T \left( \frac{1}{2} |\Dot \beta(s)|^2+m(\beta(s),s)^{\theta}\right) ds + u(\beta(T),T).
\ee
In addition to prescribing the initial density of agents, the system \eqref{mfg} will be supplemented by a terminal condition, which may either be
\be \label{tc} \tag{TC} u(x,T)=c_T m^{\theta}(x,T), \ee
for some $c_T>0$, or
\be \label{pp} \tag{P} m(x,T)=m_T(x), \ee
where $m_T: \R \to [0,\infty)$ satisfies the compatibility condition
\be \int_{\R}m_T = \int_{\R} m_0.\ee
The problem with a prescribed terminal cost for the agents such as \eqref{tc} is typical in the theory of MFG. On the other hand, the problem with a prescribed terminal density \eqref{pp} is known as the planning problem, and is related to the theory of optimal transport (see \cite[Sec. 1]{CMP}).

The system \eqref{mfg} has been previously studied in a more general setting, and in arbitrary dimensions, when the right hand side of the Hamilton--Jacobi equation (HJ) in \eqref{mfg} is allowed to be an arbitrary strictly increasing function $f(m)$. This study was started by P.-L. Lions in his lectures at Coll\`ege de France \cite{L-college}, where he observed that the monotonicity of $f$ creates a regularization effect in the region $\{m>0\}$ which allows one to reformulate the problem as a single elliptic equation in $u$, in the space-time variables. This point of view was developed and extended in \cite{MimikosMunoz, Munoz, Porretta}. In arbitrary dimensions, the solution is known to be smooth under the blow-up assumption \begin{equation} \label{eq:blowup} \lim_{m\rightarrow 0^+}f(m)=-\infty, \end{equation} 
provided that the marginals $m_0, m_T$ are (strictly) positive, say for positive measures on a compact domain (e.g. on the flat torus) or for Gaussian-like measures on the whole space. It was also established in \cite{MimikosMunoz} that, for the one-dimensional case, assumption \eqref{eq:blowup} may be removed, thus requiring only the positivity of the marginals. 

The case of compactly supported marginals on the whole space is more challenging, largely due to the loss of ellipticity at the interface where $m$ vanishes, and was recently studied for the first time by P. Cardaliaguet, A. Porretta, and the author in \cite{CMP}.

Indeed, it was shown in \cite{CMP} that the power-type nonlinearity $f(m)=m^{\theta}$ gives rise to finite speed of propagation, leading to the formation of a free boundary. The regularity of the solution and the geometry of its free boundary were then thoroughly examined. Furthermore, the authors exhibited a family of self-similar solutions $(\mathcal{U}_a,\mathcal{M}_a)$, where $a>0$, and the density $\mathcal{M}_a$ is given, for $(x,t)\in \R \times (0,\infty)$, by
\be \label{Mself} \mathcal{M}_a(x,t)=t^{-\al}M_a(x/t^{\al}) , \,\,\,M_a(\eta)=\left(R_a- \frac{\al(1-\al)}{2}\eta^2\right)_+^{1/\theta},\,\,\,\,\, \al=\frac{2}{2+\theta}, \ee
with $R_a>0$ being the unique constant such that
$\int_{\R} \left(R_a- \frac{\al(1-\al)}{2}\eta^2\right)_+^{1/\theta}d\eta = a. $
The function $\mathcal{M}_a$ evolves from a Dirac mass into a compactly supported measure, and is analogous to the well-known source-type solutions for the porous medium equation \cite{vazquez2007porous}. The accompanying value function $\mathcal{U}_a$ is given, in the support of $\mathcal{M}_a$, by the formula
\be \label{Uself} \mathcal{U}_a(x,t) = t^{2\al-1}U_a(x/t^{\al})+z(t),  \quad U_a(\et)=-\frac{\al}{2}\et^2 , \quad z'(t)=-R_at^{2\al-2}. \ee

 There is ample experimental, computational, and analytical evidence showing that self-similar solutions determine the asymptotic behavior of a  large number of evolution problems. The purpose of this work is to determine whether this is the case for the solutions to \eqref{mfg}, a question that was left open in \cite{CMP}. Our main result answers this question in the affirmative, for arbitrary $\theta>0$. We refer to Section \ref{sec: assumptions} for the assumptions \eqref{m0 assum}, \eqref{m0 C1,1}, \eqref{mT assum}, \eqref{mT bump}, and \eqref{kap size}.
\begin{thm} \label{thm: fin hor} Assume that $m_0$ satisfies \eqref{m0 assum} and \eqref{m0 C1,1}, and that \eqref{kap size} holds. In case of problem \eqref{mfg}--\eqref{pp}, assume also that $m_T$ satisfies \eqref{mT assum} and $\eqref{mT bump}$.  Let $(u^T,m^T)$ be the solution to \eqref{mfg}--\eqref{tc} or \eqref{mfg}--\eqref{pp}, and let $\int_{\R}m_0=a>0$. Then, as $T \to \infty$ and $t\to \infty$, $m^T(\cdot,t) \to \mathcal{M}_a(\cdot,t)$ in the $L^1$ sense, that is,
\be \label{conv result L1} \lim_{t \to \infty} \limsup_{T\to \infty} \| m^T(\cdot,t)- \mathcal{M}_a(\cdot,t) \|_{L^{1}(\R)} =0,
\ee 
and in the $L^{\infty}$ sense in the proper scale, that is,
\be \label{conv result Linfty} \lim_{t \to \infty} \limsup_{T\to \infty} t^{\al}\| m^T(\cdot,t)- \mathcal{M}_a(\cdot,t) \|_{L^{\infty}(\R)} =0.
\ee 
More generally, for every $p\in [1,\infty]$, one has $L^p$ convergence of $m^T$ and $Du^T$ in the proper scale, that is,
\be \label{conv result 1} \lim_{t \to \infty} \limsup_{T\to \infty} t^{\al(1-\frac1p)}\| m^T(\cdot,t)- \mathcal{M}_a(\cdot,t) \|_{L^{p}(\R)} =0,
\ee
\be \label{conv result 2} \lim_{t \to \infty} \limsup_{T\to \infty} t^{2-\al(1+\frac1p)}\left\| m^T(\cdot,t)\left|u^T_x(\cdot,t)- (\mathcal{U}_a)_x(\cdot,t)\right|^2 \right\|_{L^{p}(\R)} =0, \ee
\be \label{conv result 3} \lim_{t \to \infty} \limsup_{T\to \infty} t^{2-\al(1+\frac1p)}\left\| m^T(\cdot,t)\left|u^T_t(\cdot,t)- (\mathcal{U}_a)_t(\cdot,t)\right| \right\|_{L^{p}(\R)} =0. \ee
\end{thm}
In relation to the main text, this result appears as Theorem \ref{thm: main} (for the terminal cost problem \eqref{tc}) and Theorem \ref{thm: main plan} (for the planning problem \eqref{pp}).
\begin{rem} We now explain what is meant in  Theorem \ref{thm: fin hor} by the {\it proper scale}, a term that we borrow from \cite{vazquez2007porous}. Note that, from \eqref{Mself}, $\mathcal{M}_a$ decays with algebraic rate $t^{-\al}$ as $t\to \infty$, and its support $S_a(t)$ propagates with speed $t^{\al}$. It can be similarly seen, from \eqref{Uself}, and from the speed of propagation of $S_a(t)$, that the space and time derivatives of $\mathcal{U}_a$ decay, respectively, with algebraic rates $t^{-(1-\al)}$ and $t^{-2(1-\al)}$ inside $S_a(t)$. These rates determine the only possible (and non-trivial) scaling of the convergence results \eqref{conv result 1}--\eqref{conv result 3}. Similar observations can be made about higher derivatives of these functions, and their moduli of continuity. Part of our work in this paper will be to show that these exact algebraic rates of decay are a general phenomenon exhibited by arbitrary solutions to $\eqref{mfg}$. An estimate in this direction was already obtained in our previous work \cite[Thm. 1.4]{CMP}, which will be our starting point.
\end{rem}
\begin{rem} Both $m^T$ and $\mathcal{M}_a$ tend to $0$ uniformly as $t\to \infty$. However, due to conservation of mass, $a=\int_{\R}m_0$ is the unique positive number satisfying \eqref{conv result L1}. In particular, Theorem \ref{thm: fin hor} says that the intermediate asymptotic behavior of \eqref{mfg} is fully determined by the initial mass $a$ (c.f. \cite[Thm. 18.1]{vazquez2007porous}).  
\end{rem}
A follow-up question to Theorem \ref{thm: fin hor}, which is natural in MFG (see \cite{CaPo2019,CiPo}), is whether the limit as $T \to \infty$ of the entire solution $u^T$ can be characterized as a specific solution $(u,m)$ of the infinite horizon system:
\begin{equation}
\begin{cases} \tag{MFGI}
-u_{t}+\frac{1}{2}u_{x}^{2}=m^{\theta} & (x,t)\in\R\times(0,\infty),\\
m_{t}-(mu_{x})_{x}=0 & (x,t)\in\R\times(0,\infty),\\
m(x,0)=m_0(x).
\end{cases}\label{mfginfhor}
\end{equation}
In the simpler case of solutions in a compact domain with strictly positive density, where there is no free boundary and the solutions are smooth, this analysis was carried out in \cite[Thm. 1.3]{MimikosMunoz}. In the present setting, however, this convergence question turns out to be a more delicate issue that is closely related to the well-posedness of \eqref{mfginfhor}. To illustrate this point, we note that, in view of \eqref{Uself}, the function $\mathcal{U}_a$ is unbounded in $S_a(t)$ as $t \to \infty$ precisely when $2\al-1\geq 0$, that is, when $\theta \leq2$. As we will see, this reflects the general behavior of the solutions to \eqref{mfginfhor}. Namely, the range $\theta \in (2,\infty)$ is the maximal interval in which global well-posedness of \eqref{mfginfhor} holds and the long time behavior of the solution $(u^T,m^T)$ can be characterized as the unique solution $(u,m)$  to \eqref{mfginfhor} that vanishes as $t \to \infty$. Furthermore, the intermediate asymptotic behavior of $(u,m)$ is also given by the self-similar solution. That is the content of the following result.
\begin{thm} \label{thm: INTRO theta >2} Assume that $m_0$ satisfies \eqref{m0 assum} and \eqref{m0 C1,1}, and that \eqref{kap size} holds. Let $a=\int_{\R}m_0>0$, and assume that $\theta > 2$. The following holds:
\begin{enumerate} 
\item There exists a unique solution $(u,m)\in W^{1,\infty}(\R \times (0,\infty))\times C(\R \times [0,\infty))$ to \eqref{mfginfhor} satisfying 
\be \label{INTRO u(infty)=0 th>2} \lim_{t \to \infty} \|u(\cdot,t)\|_{L^{\infty}(\R)}=0. \ee
Furthermore, for every $p\in [1,\infty]$, 
\be \label{INTRO conv result 1 theta>2} \lim_{t \to \infty}  t^{\al(1-\frac1p)}\left\| m(\cdot,t)- \mathcal{M}_a(\cdot,t) \right\|_{L^{p}(\R)} =0,
\ee
\be \label{INTRO conv result 2 theta>2} \lim_{t \to \infty}  t^{2-\al(1+\frac1p)}\left\| m(\cdot,t)\left|u_x(\cdot,t)- (\mathcal{U}_a)_x(\cdot,t)\right|^2 \right\|_{L^{p}(\R)} =0,  \ee
\be \label{INTRO conv result 3 theta>2} \lim_{t \to \infty}  t^{2-\al(1+\frac1p)}\left\| m(\cdot,t)\left|u_t(\cdot,t)- (\mathcal{U}_a)_t(\cdot,t)\right| \right\|_{L^{p}(\R)} =0. \ee
\item Let $(u^T,m^T)$ be the unique solution to \eqref{mfg}--\eqref{tc}. Then $(u,m)$ is characterized uniquely as
\be \label{INTRO limit char th>2} (u,m)=\lim_{T \to \infty} (u^T,m^T) \,\,\,\,\text{ in } \,\,\, C^1_{\emph{loc}}(\R \times (0,\infty))\times C_{\emph{loc}}(\R \times \left[0,\infty\right)).\ee
\item Assume further that \eqref{mT assum} and \eqref{mT bump} hold. Let $(u^T,m^T)$ be the solution to \eqref{mfg}--\eqref{pp} that satisfies $\intr u(\cdot,T/2)=0.$ Then, up to an adequate choice of $u^T$ on its region of non-uniqueness $\{m^T=0\}$, \eqref{INTRO limit char th>2} holds.
\end{enumerate}
\end{thm}
In the main text, this is the combination of Theorem \ref{thm: fin theta>2} (for the terminal cost problem) and Theorem \ref{thm: fin theta>2 plan} (for the planning problem).
\begin{rem} In view of Theorem \ref{thm: INTRO theta >2} and the preceding comments, we henceforth refer to the range $\theta \in (2,\infty)$ as the subcritical range, to $\theta \in (0,2)$ as the supercritical range, and to $\theta=2$ as the critical case. 
\end{rem}
\begin{rem} We explain here the non-uniqueness statement regarding the planning problem. It is a general fact (see \cite[Thm. 4.3]{CMP}, or Theorem \ref{thm: prelim wellposedness} below) that, for both \eqref{mfg}--\eqref{tc} and \eqref{mfg}--\eqref{pp}, the density function $m=m^T$ is unique. However, while the value function $u=u^T$ is unique for \eqref{mfg}--\eqref{tc}, in the case of \eqref{mfg}--\eqref{pp} it is only unique up to a constant, and only in the region $\{m>0\}$. Hence, in the statement of Theorem \ref{thm: INTRO theta >2}, we specify its integral at $t=T/2$ to uniquely determine $u^T$ in $\{m>0\}$. Furthermore, to avoid pathological behavior in the region of non-uniqueness, we also specify that the global convergence only holds up to an adequate choice of the solution $u^T$ in that region. Naturally, the convergence inside $\{m>0\}$ holds independently of the choice of $u^T$.    
\end{rem}
Additionally, for the critical case $\theta=2$, we are still able to establish a partial well-posedness result for \eqref{mfginfhor} which is sufficient to characterize the limit of $(Du^T,m^T)$ (see Theorems \ref{thm: fin theta=2} and \ref{thm: th=2 plan} below). These theorems for the subcritical and critical values of $\theta$ are not only of interest on their own, but they are also instrumental in proving our main convergence result, Theorem \ref{thm: fin hor}.

In fact, while the statement of Theorem \ref{thm: fin hor} holds for arbitrary $\theta>0$, the proofs diverge substantially across the different ranges of the parameter. In short, while the proofs for $\theta \in [2,\infty)$ rely primarily on well-posedness (or partial well-posedness) of \eqref{mfginfhor}, the proof for the supercritical range $\theta \in (0,2)$ relies on a precise estimate of the convergence rate.

We prove that, when $\theta \in (0,2)$, all sufficiently regular solutions $(u,m)$ of \eqref{mfginfhor} converge, as $t \to \infty$, to the self-similar solution, at a quantifiable algebraic rate that is uniform across all such solutions (see Proposition \ref{prop: exponential theta<2}). This quantitative result, which is a feature of the supercritical range, allows us to bypass the need for a well-posedness theorem and complete the proof of Theorem \ref{thm: fin hor} for this range.

We further elaborate on our proof strategy. We use  the method of continuous rescaling to analyze a zoomed version of $(u,m)$, in which the scaling factors are indexed by a continuous parameter $\ta$, representing the `new time'. While this approach is new in the context of MFG, it is a classical idea that been applied successfully to study the asymptotics of many evolution problems (see, for instance, \cite{GaVaBook,Hu}). To determine the correct zooming factors, we observe that for any $\lambda >0$, the pair of functions $(x,t) \mapsto \lambda^{\al} m(\lambda^{\al}x ,t \lambda )$ and $(x,t) \mapsto \lambda^{1-2\al}u(\lambda^{\al}x,\lambda t)$ still satisfy \eqref{mfg}, with the total mass $a=\intr m(\cdot,t)$ preserved. Crucially, in the case of the self-similar solution $\mathcal{M}_a$, the function itself remains invariant under this rescaling. Accordingly, we define the continuous rescaling $(v,\mu)$, of \eqref{mfginfhor} by
\be \label{INTRO v mu defi}
\mu(\et,\ta)=t^{\al}m(x,t), \, v(\et,\ta)=t^{1-2\al}u(x,t), \quad (\et,\ta)\in \R \times \left(-\infty,\log(T) \right],
\ee
where
\be \label{INTRO tau eta defi} t=e^{\ta},\,\, x= t^{\al} \et.\ee
Under this transformation, the pair $((\mathcal{U}_a)_x(x,t),\mathcal{M}_a(x,t))$ becomes the stationary profile $(U_a'(\eta),M_a(\eta))$. Accordingly, \eqref{INTRO conv result 1 theta>2}--\eqref{INTRO conv result 2 theta>2} are reduced to showing that, as $\ta \to \infty$, $(v_{\eta},\mu)$ converges to $(U_a'(\eta),M_a(\eta))$. To show this, we analyze the system of PDE satisfied by $(v,\mu)$, together with an appropriately chosen Lyapunov functional. We will see that, in particular, the qualitative behavior of this functional is determined by whether $\theta$ is in the subcritical, critical, or supercritical range.

We present the formal computations behind this analysis in Section \ref{sec: convergence IH}, where we study the asymptotic behavior of the solutions to \eqref{mfginfhor}, showing that they always satisfy \eqref{INTRO conv result 1 theta>2}--\eqref{INTRO conv result 3 theta>2}, under the assumption that the continuous rescaling is sufficiently well-behaved. A major technical difficulty, however, is obtaining adequate estimates for the continuous rescaling $(v,\mu)$. This amounts to obtaining bounds for $(u,m)$ whose algebraic rate of growth as $t \to \infty$ is the sharp one, in the sense that it is attained by the self-similar solution. We carry this out in Section \ref{sec: estimates}, initially narrowing down our focus to the terminal cost problem \eqref{mfg}--\eqref{tc}, to streamline the exposition. To establish these estimates, we use a variety of techniques, including optimal control, Lagrangian coordinates, maximum principle arguments, displacement convexity, and the regularity method of intrinsic scaling. In Section \ref{sec: main}, we connect the asymptotic findings for \eqref{mfginfhor} with the finite horizon problem, and take advantage of the sharp growth estimates to prove the main results of the paper for \eqref{mfg}--\eqref{tc}. Finally, Section \ref{sec:plan} outlines how to adapt these arguments for the planning problem \eqref{mfg}--\eqref{pp}. 

\subsection*{The case of dimensions $d>1$}
We now discuss the extent to which our arguments may be used to study the asymptotic behavior of \eqref{mfg} in dimensions larger than $1$. This is a delicate matter. On one hand, the strongest regularity results of \cite{CMP}, of which this paper is a sequel, are currently quite far from being extended to $d>1$, with major obstacles such as the lack of a continuity result for $m$, and of a gradient estimate for $u$ that does not require $m_0$ to be positive. We use these regularity results here in full force, in order to obtain the uniform convergence result \eqref{conv result Linfty}. Therefore, the case $p=\infty$ of Theorem \ref{thm: fin hor} currently appears to be out of reach for higher dimensions.

However, the range $p \in [1,\infty)$ is a different matter. Indeed, a major portion of the analysis of the infinite horizon system and the continuous rescaling, developed in Section \ref{sec: convergence IH}, involves energy computations which remain valid in any dimension. While there are still significant regularity issues, we believe that this paper paves the way to prove an analogue of Theorem \ref{thm: fin hor} in arbitrary dimensions, and this will be explored in upcoming work.

\section{Assumptions and preliminary results} \label{sec: assumptions}
We will now list our assumptions, which correspond to those of \cite{CMP}, as well as the main results from \cite{CMP} that will be needed later about the solutions to \eqref{mfg} and their free boundary. In what follows, $\theta$, $\overline{\al}$, and $C_0$ are given positive constants, with $\overline{\al} \in (0,1)$.

We assume that the function $m_0:\R \to [0,\infty)$ satisfies, for some given interval $(a_0,b_0),$
\be \label{m0 assum} \{m_0>0\}=(a_0,b_0),\,\,\,  m_0\in C(\R) \cap C^{1,\overline{\al}}((a_0,b_0)), \ee
and
\be \label{m0 C1,1} \quad -C_0 \leq (m_0^{\theta})_{xx} \leq -\frac{1}{C_0} \text{ in } (a_0,b_0).
\ee
As was explained in \cite{CMP}, \eqref{m0 C1,1} can be understood as a second-order non-degeneracy assumption which, in particular, ensures that the free boundary starts moving at $t=0$ (note also that \eqref{m0 C1,1} is satisfied at each time $t>0$ by the self similar solution $\mathcal{M}_a(\cdot,t)$). Along these lines, the following non-degeneracy property is not an assumption, but a useful consequence of \eqref{m0 assum} and \eqref{m0 C1,1}, which holds up to increasing the value of $C_0$:
\be \label{m0 bump} \frac{1}{C_{0}}\text{dist}(x,\{a_{0},b_{0}\})\leq m_{0}(x)^{\theta}\leq C_{0}\text{dist}(x,\{a_{0},b_{0}\}),\quad x\in [a_0,b_0].
\ee 
For the case of \eqref{mfg}--\eqref{pp}, we assume that $m_T: \R \to [0,\infty)$ satisfies, in accordance with \eqref{m0 assum} and \eqref{m0 bump}, for some given interval $(a_1,b_1)$,
\be 
\int_{\R} \label{mT assum} m_T = \int_{\R} m_0, \,\,\,\, \{m_{T}>0\}=(a_{1},b_{1}), \,\,\,\, m_T\in C(\R) \cap C^{1,\overline{\al}}(a_1,b_1), \ee
and
\be \label{mT bump} \frac{1}{C_{1}}\text{dist}(x,\{a_{1},b_{1}\})\leq m_{T}(x)^{\theta}\leq C_{1}\text{dist}(x,\{a_{1},b_{1}\}),\quad x\in [a_1,b_1]. \ee 
In order to state the results in a way that is consistent with the proper scaling of \eqref{mfg}, we will also write
\be \label{kappaT defi} c_T=\kappa_T T, \ee
where $\kappa_T>0$, and assume that
\be \label{kap size} \frac{1}{C_0}\leq \kappa_T \leq C_0, \quad T \in [0,\infty]. \ee

Next, since \eqref{mfg} does not have, in general, classical solutions, we define a notion of generalized solution, and cite a well-posedness result (\cite[Thm. 4.3, Thm. 1.1]{CMP}).
\begin{defn} \label{def: sol} We say that $(u,m)\in W^{1,\infty}(\R \times (0,T)) \times C_c(\R \times [0,T])$ is a solution to \eqref{mfg} if $u$ satisfies the first equation in the viscosity sense, and $m$ satisfies the second equation in the distributional sense. Furthermore, we say that $(u,m)$ is a solution to \eqref{mfg}--\eqref{tc} (respectively, \eqref{mfg}--\eqref{pp}) if \eqref{tc} also holds (respectively, if \eqref{pp} also holds), in the pointwise sense.
\end{defn}
\begin{thm}\label{thm: prelim wellposedness}Assume that $m_0$ satisfies \eqref{m0 assum} and \eqref{m0 C1,1}. Then the following holds:
\begin{enumerate}
    \item There exists a unique solution $(u,m)$ to \eqref{mfg}--\eqref{tc}.
    \item If $m_T$ satisfies \eqref{mT assum} and $\eqref{mT bump}$, then there exists a solution $(u,m)$ to \eqref{mfg}--\eqref{pp}. The function $m$ is unique, and $u$ is unique up to an additive constant on the set $\{m>0\}$. 
\end{enumerate}
    
\end{thm}

The next result concerns the regularity of the solution and the analysis and characterization of its free boundary (\cite[Thm. 1.1, Thm. 1.5]{CMP}).
\begin{thm}\label{thm.intro1} Assume that $m_0$ satisfies \eqref{m0 assum} and \eqref{m0 C1,1}, and that \eqref{kap size} holds.
In case of problem \eqref{mfg}--\eqref{pp}, assume also that $m_T$ satisfies \eqref{mT assum} and $\eqref{mT bump}$.
Let $(u,m)$ be a solution to \eqref{mfg}--\eqref{tc}, or to \eqref{mfg}--\eqref{pp}. Then the following holds:
\begin{enumerate}
\item $(u, m)\in C_{\emph{loc}}^{2,\overline{\al}}((\mathbb{\mathbb{R}\times}[0,T])\cap\{m>0\})\times C_{\emph{loc}}^{1,\overline{\al}}((\mathbb{\mathbb{R}\times}[0,T])\cap\{m>0\})$, and there exists $\beta\in (0,1)$ such that, for every $\delta>0$, $(u,m)\in C^{1,\beta}(\R \times [\delta,T-\delta]) \times C^{0,\beta}(\R \times [\delta,T-\delta])$.
\item There exist two functions $\gamma_{L},\gamma_{R} \in W^{2,\infty}(0,T)$
such that $\gamma_L<\gamma_R$, and, for every $t\in (0,T),$
\begin{equation} \label{freebd ch}
\{m(\cdot,t)>0\}=(\gamma_{L}(t),\gamma_{R}(t)).
\end{equation}
The support of $m$ is convex, with the curves $\gamma_L$ and $\gamma_R$ satisfying $\Ddot\gamma_L>0,\Ddot\gamma_R<0$. Furthermore, if $u$ solves \eqref{mfg}--\eqref{tc}, then the support is expanding in time, with $\dot\gamma_L<0,\dot\gamma_R>0.$
\end{enumerate}
\end{thm}

A fundamental tool in the analysis of \eqref{mfg} is the Lagrangian flow $\gamma:(a_0,b_0) \times [0,T] \to \R $. This is simply the flow of optimal trajectories associated to the HJ equation, and can be defined, for each $x\in (a_0,b_0),$ as the solution to the differential equation
\be \label{flow defi}
\begin{cases}
\dot \gamma ( x,\cdot)=-u_{x}(\gamma( x,\cdot),\cdot),\\
\gamma( x,0)=x. 
\end{cases}
\ee
The next result we cite here contains the main properties of the Lagrangian flow, its relation to the free boundary curves, and, crucially, its exact rate of growth  (\cite[Thm. 1.1, Thm. 1.4]{CMP}). To specify the rate of growth, we define
\be \label{scrd defi}
\mathscr{d}(t) = 
\begin{cases} 
t & \text{if } u \text{ solves \eqref{mfg}--\eqref{tc}}, \\
\text{dist}(t, \{0, T\}) & \text{if } u \text{ solves \eqref{mfg}--\eqref{pp}.}
\end{cases}
\ee

\begin{thm}\label{thm.intro2} Under the assumptions of Theorem \ref{thm.intro1}, let $(u,m)$ be a solution to \eqref{mfg}--\eqref{tc}, or to \eqref{mfg}--\eqref{pp}, and let $\gamma_L, \gamma_R$ be its free boundary curves. The Lagrangian flow $\gamma$ is well defined
on $(a_{0},b_{0})\times[0,T]$, and
\be \label{gamma regu}
\gamma\in W^{1,\infty}((a_{0},b_{0})\times(0,T))\cap C_{\emph{\text{loc}}}^{2,\overline{\al}}((a_{0},b_{0})\times[0,T]),\,\,\,\gamma_{L}=\gamma(a_{0},\cdot),\,\gamma_{R}=\gamma(b_{0},\cdot).
\ee 
Furthermore, we have, for $(x,t)\in (a_0,b_0)\times (0,T)$,
\be \label{mass cons} \gamma_x(x,t)=\frac{m_0(x)}{m(\gamma(x,t),t)},\ee
and there exists $C>0$ such that
\be \label{gamma bd} |\gamma(x,t)|\leq C(1+\mathscr{d}(t)^{\al}), \ee
    \be \label{gamma x bd}  \frac{1}{C} (1+\mathscr{d}(t)^{\al})\leq \gamma_x(x,t) \leq  C(1+\mathscr{d}(t)^{\al}),\ee
where $C=C(C_0,|a_0|,|b_0|)$ in the case of problem \eqref{mfg}--\eqref{tc}, and $C=C(C_0,C_1,|a_0|,|b_0|,|a_1|,|b_1|)$ in the case of problem \eqref{mfg}--\eqref{pp}.
\end{thm}
Finally, we state a version of the so-called displacement convexity formula, which is a versatile tool for obtaining energy estimates. In the context of MFG, displacement convexity was first studied in \cite{LaSa,GoSe}, but it is a classical topic in the theory of optimal transport (see \cite{McCann}).
\begin{lem} \label{lem:displ} Under the assumptions of Theorem \ref{thm.intro1}, let $(u,m)$ be a solution to \eqref{mfg}--\eqref{tc} or \eqref{mfg}--\eqref{pp}, and let $p\in (0,1)\cup(1,\infty)$. 
Then, for every $\zeta \in C^{\infty}_c(0,T)$, with $\zeta \geq 0$, we have
\be \label{displ subsol} \int_{0}^T\intr (m^p u_{xx}^{2}+4\theta(\theta+p)^{-2} (m^{\frac{\theta+p}{2}})_{x}^{2})\zeta(t) dxdt \leq \frac{1}{p(p-1)}\int_{0}^T \intr m^p \Ddot \zeta(t)dxdt.  \ee
In particular, the map $t \mapsto (p(p-1))^{-1}\int_{\R} m^p(x,t)dx$ is convex.

\end{lem}
\begin{proof}
For smooth solutions, the fact that \eqref{displ subsol} holds with equality is precisely the  displacement convexity formula (\cite[Prop. 3.6]{CMP}). Therefore, the task at hand is to explain how to handle the limited regularity. Consider, for $R>0$, the MFG system with Neumann boundary conditions
\begin{equation}
\begin{cases}
-u_{t}+\frac{1}{2}u_{x}^{2}=m^{\theta} & (x,t)\in[-R,R] \times(0,T),\\
m_{t}-(mu_{x})_{x}=0 & (x,t)\in [-R,R] \times(0,T), \\
u_x(-R,t)=u_x(R,t)=0, & t\in [0,T].
\end{cases}\label{mfgneum}
\end{equation}
    It was shown in \cite[Thm. 4.3]{CMP} that one may uniformly approximate $(u,m)$ by a sequence of smooth solutions $(u_n,m_n)$ to \eqref{mfgneum}, for a fixed value of $R$, as long as $R$ is sufficiently large for the rectangle $[-R,R]\times [0,T]$ to contain the support of $m$. Furthermore, $u_n \to u$ in $C^2_{\text{loc}}(\{m>0\})$, so that, in particular,
    \be \label{liminf muxxocme} m^p u_{xx}^2 \leq \liminf_{n \to \infty}  m_n^{p} ((u_n)_{xx})^2. \ee     
    The displacement convexity formula, which is valid in this smooth setting,  reads:
    \be \label{disp-dksdaok}  \frac{d^{2}\vfi_n}{dt^{2}} = \intr (m_n^p (u_n)_{xx}^{2}+4\theta(\theta+p)^{-2} (m_n^{\frac{\theta+p}{2}})_{x}^{2}))dx, \quad \vfi_n(t)= \frac{1}{p(p-1)}\int_{-R}^{R}m_n^p(x,t)dx.\ee 
 In view of \eqref{liminf muxxocme}, the result follows from Fatou's lemma by integrating \eqref{disp-dksdaok} against $\zeta$ and letting $n \to \infty$.
\end{proof}

\section{Intermediate asymptotics of the infinite horizon system} \label{sec: convergence IH}

In this section, we will study the asymptotic behavior of compactly supported solutions to the infinite horizon system:
\begin{equation}
\begin{cases}
-u_{t}+\frac{1}{2}u_{x}^{2}=m^{\theta} & (x,t)\in\R\times (t_0,\infty),\\
m_{t}-(mu_{x})_{x}=0 & (x,t)\in\R\times (t_0,\infty),
\end{cases}\label{mfgi}
\end{equation}
where $t_0>0$. In order to present some of the main ideas in a simple setting, we will postpone the discussion of well-posedness and regularity of solutions to this system. Instead, we will prove a convergence result under the assumption that a solution $(u,m)$ exists, and that its continuous rescaling $(v,\mu)$ has sufficient regularity, uniformly in time. Specifically, our main goal in this section is to prove the following statement.
\begin{prop} \label{prop: convergence result}  Let $(u,m)\in C^1(\R \times (t_0,\infty)) \times C(\R \times [t_0,\infty))$ satisfy the first equation of \eqref{mfgi} in the classical sense, and the second equation in the distributional sense, with $m \geq 0$. Assume that there exist functions $\gamma_L, \gamma_R \in W^{1,\infty}_{\emph{loc}}((t_0,\infty))$ such that \eqref{freebd ch} holds for every $t\in (t_0,\infty)$, and let $a=\int_{\R} m(\cdot,t_0)dx>0$.  Let $(v,\mu)$ be the continuous rescaling \eqref{INTRO v mu defi} of $(u,m)$, and let $\tau_0=\log(t_0)$. Assume further that, for some constant $\rho>0$, \be \label{mu cont assumption} \mu \in \emph{BUC}(\R \times (\tau_0,\infty)), \quad \sup_{\tau_1 \in (\tau_0+\rho,\infty)}\|\mu^{\theta}\|_{H^1(\R \times (\tau_1-\rho,\tau_1+\rho)))}<\infty,\ee
and that there exists a constant $R\geq R_a$ such that 
\be \label{Dv bd assumption}  \emph{supp}(\mu) \subset [-R,R] \times \left[\tau_0, \infty\right), \quad Dv\in  C^1(\{\mu>0\})\cap L^{\infty}((-R,R)\times (\tau_0,\infty) ).\ee 
Then, for every $p\in [1,\infty]$,
\be \label{IHH conv result 1} \lim_{t \to \infty}  t^{\al(1-\frac1p)}\left\| m(\cdot,t)- \mathcal{M}_a(\cdot,t) \right\|_{L^{p}(\R)} =0,
\ee
\be \label{IHH conv result 2} \lim_{t \to \infty}  t^{2-\al(1+\frac1p)}\left\| m(\cdot,t)\left|u_x(\cdot,t)- (\mathcal{U}_a)_x(\cdot,t)\right|^2 \right\|_{L^{p}(\R)} =0,  \ee
\be \label{IHH conv result 3} \lim_{t \to \infty}  t^{2-\al(1+\frac1p)}\left\| m(\cdot,t)\left|u_t(\cdot,t)- (\mathcal{U}_a)_t(\cdot,t)\right|\right\|_{L^{p}(\R)} =0. \ee
\end{prop}

The significance of this result will become clear in Section \ref{sec: estimates}, where we show that its regularity assumptions are indeed satisfied by the continuous rescaling $(v^T,\mu^T)$ of the solutions $(u^T,m^T)$ to \eqref{mfg}--\eqref{tc} or \eqref{mfg}--\eqref{pp}, uniformly in $T$. As a result, any subsequential limit as $T \to \infty$ of $(u^T,m^T)$ must be a solution $(u,m)$ to \eqref{mfgi} which satisfies the hypotheses, and hence the conclusion, of Proposition \ref{prop: convergence result} (see Proposition \ref{prop: existence subseq}).
\subsection{Basic properties of the continuous rescaling}
For clarity, we restate the precise definition of the continuous rescaling.
\begin{defn} \label{def: cont resc inf} Given a solution $(u,m)$ to \eqref{mfgi}, we define the continuous rescaling $(v,\mu)$ as follows. For $(\et,\ta)\in \R \times \left[\tau_0, \infty\right):=\R \times\left[ \log(t_0),\infty\right),$ we set
\be \label{tau eta defi IH} t=e^{\ta},\,\, x=t^{\al} \et,\ee
and we define
\be \label{v mu defi IH}
\mu(\et,\ta)=t^{\al}m(x,t), \, v(\et,\ta)=t^{1-2\al}u(x,t). 
\ee
\end{defn}
We begin by reducing the proof of Proposition \ref{prop: convergence result} to the equivalent problem of showing that $(v_\eta,\mu)$ converges to the stationary profile $(U'_a,M_a)$.
\begin{lem} \label{lem:reduc to cont resc} Let $(u,m)$ and $(v,\mu)$ be as in Proposition \ref{prop: convergence result}. Then \eqref{IHH conv result 1} and \eqref{IHH conv result 2} are equivalent, respectively, to
\be \label{conv v mu} \lim_{\tau \to \infty}  \left\| \mu (\cdot,\tau)- M_a(\cdot) \right\|_{L^{p}(\R)} =0,\,\,\, \text{ and } \,\,\, \lim_{\tau \to \infty}  \left\| \mu (\cdot,\tau)\left|v_\eta - U'_a(\cdot)\right|^2 \right\|_{L^{p}(\R)} =0.
\ee
\end{lem}
\begin{proof} The key point is that the continuous rescaling of the self-similar solution $\mathcal{M}_a$ is precisely the stationary profile $M_a$. Indeed, changing variables according to \eqref{tau eta defi IH}, we have
\begin{multline}  
t^{\al(p-1)}\left\| m(\cdot,t)- \mathcal{M}_a(\cdot,t) \right\|_{L^{p}(\R)}^p =t^{\al(p-1)}\intr |m(x,t)-t^{-\al}M_a(x/t^{-\al})|^pdx\\ = \intr |t^{\al}m(t^{\al}\eta ,t)-M_a(\eta)|^pd\eta= \intr |\mu(\eta,\tau)-M_a(\eta)|^pd\eta =  \left\| \mu (\cdot,\tau)- M_a(\cdot) \right\|_{L^{p}(\R)}^p,
\end{multline}
which proves the first equivalence of \eqref{conv v mu}. The other equivalence is obtained in the same way. 
\end{proof}

The next step will be to obtain the system of equations satisfied by the continuous rescaling $(v,\mu)$. In fact, recalling \eqref{Uself}, we want to show that, in $\{\mu>0\}$, $v_{\eta} \to U'_{a} \equiv -\al \eta$. As such, it will be convenient to define
\be \label{w defi} w(\eta,\tau)= v(\eta,\tau)+ \frac{\al}{2} \eta^2,\ee
and to rewrite the problem as a system in $(w,\mu)$.
\begin{lem}  Let $(u,m)$ and $(v,\mu)$ be as in Proposition \ref{prop: convergence result}, and define $w$ by \eqref{w defi}. Then the pair $(w,\mu)$ satisfies
    \be  \label{w mu sys}\begin{cases}-w_{\ta}+\frac12w_{\et}^2 =\mu^{\theta}+\frac{\al(1-\al)}{2}\et^2+(2\al-1)w & (\eta,\tau)\in \R \times (\tau_0, \infty), \\
    \mu_{\ta}-(\mu w_{\et})_{\et}=0 & (\et, \ta) \in \R \times (\ta_0, \infty),
    \end{cases}\ee
    where the first equation is taken in the classical sense, and the second equation is taken in the sense of distributions. Furthermore, in the set $\{ \mu >0 \}$,  the second equation is also satisfied in the classical sense.
\end{lem}
\begin{proof} We write $u(x,t)=t^{2\al-1}v(\eta,\tau),$ and, using the fact that $u\in C^1(\R \times (\tau_0,\infty))$ we compute
    
\be u_t=t^{2\al-2}((2\al-1)v+v_\ta-\al v_{\et}),\ee
\be \label{ux comp owdkxkp}u_x=t^{\al-1}v_\et.\ee
Substituting these findings in the first equation of \eqref{mfgi}, we obtain
\be \label{v hj eq} -v_{\tau}+\frac12 v_{\et}^2+\al \et v_{\et}=\mu^{\theta}+(2\al-1)v.\ee
The first equation of \eqref{w mu sys} then follows by substituting \eqref{w defi} in \eqref{v hj eq}.

We now show that the second equation of \eqref{w mu sys} holds in the distributional sense. Let $\zeta \in C^{\infty}_c(\R \times (\tau_0,\infty))$, and define $\psi\in C^{\infty}_c(\R \times (t_0,\infty))$ by
$\psi(x,t)=\zeta(x/t^{\al},\log(t))$.
Note that
\be \label{psi der tau} \psi_t(x,t)=(-\al x t^{-\al-1})\zeta_{\eta}+\zeta_{\tau}t^{-1},\ee
\be \label{psi der eta} \psi_x(x,t)=\zeta_{\eta}t^{-\al}.\ee
Using \eqref{ux comp owdkxkp} and \eqref{w defi}, we also have
\be \label{318saican} w_{\eta}(\eta,\tau)=t^{1-\al}u_x(x,t)+ \alpha \eta,\ee
and, by \eqref{tau eta defi IH},
\be \label{319dpodke} d\eta d\tau = t^{-1-\al}dxdt.\ee
Hence, we obtain from \eqref{psi der tau}, \eqref{psi der eta}, \eqref{318saican} and \eqref{319dpodke} that
\begin{multline} \int_{\tau_0}^{\infty} \int_{\R} (\mu \zeta_{\tau} - \mu w_{\eta}\zeta_{\eta}) d\eta d\tau \\ =\int_{t_0}^{\infty} \int_{\R} t^{\al}m(x,t)\big(\zeta_{\tau}(x/t^{\al},\log(\tau)) - (t^{1-\al}u_x(x,t)+\alpha x/t^{\al})\zeta_{\eta}(x/t^{\al},\log(t)\big)t^{-1-\al}dxdt \\
=\int_{t_0}^{\infty} \intr (m\psi_t - m u_x \psi_x) dxdt.
\end{multline}
The right hand side vanishes, since $m$ satisfies the second equation of \eqref{mfgi} in the distributional sense. Finally, we observe from \eqref{v hj eq} that, since $u \in C^2(\{\mu>0\}),$ we have $\mu \in C^1(\{\mu>0\})$, and so the second equation of \eqref{w mu sys} must also be satisfied classically in $\{\mu>0\}$.
\end{proof}
We now take advantage of \eqref{w mu sys} to get an interior regularity estimate for $w$ in $\{\mu >0 \}$, by computing the elliptic equation satisfied by $w$.
\begin{lem}\label{lem:regu Dw in mu>0}  Let $(u,m)$ and $(v,\mu)$ be as in Proposition \ref{prop: convergence result}, and define $w$ by \eqref{w defi}. Let $\delta>0$, and assume that $(\eta_1 ,\tau_1) \in \R\times (\tau_0,\infty)$ is such that
\be \mu(\eta, \tau) \geq \delta, \quad (\et,\ta)\in (\eta_1-4\delta,\eta_1+4\delta)\times (\ta_1-4\delta,\ta_1+4\delta):= \mathcal{R}. \ee
Then there exists $C=C(R,\|\mu\|_{L^{\infty}(\R\times (\tau_0,\infty))},\|Dv\|_{L^{\infty}(\{\mu >0\})},\theta,\theta^{-1},\delta^{-1})$
such that 
\be \label{Dw C^1 bd} \|Dw\|_{C^{1,1/2}([\eta_1-\delta,\eta_1+\delta]\times \left[\tau_1-\delta,\tau_1+\delta]\right)} \leq C.\ee
\end{lem}
\begin{proof} Noting that \eqref{w mu sys} is a first order MFG system (albeit with $w$ dependence in the HJ equation), and that the equations are satisfied classically when $\mu>0,$ we follow the standard computation of the elliptic equation of $w$ (see \cite{L-college,Munoz}). Namely, solving for $\mu$ in the first equation of \eqref{w mu sys}, and substituting in the second equation, one obtains the quasilinear, degenerate elliptic equation
\be\label{w elliptic eq}  -\text{tr}(A(w_{\eta},w_{\tau})D^2w)+\al(1-\al)\eta w_{\eta}-(1-2\al)w_{\eta}^2 +(1-2\al)w_{\tau}=0, \ee
where
\be A(p,s) = \begin{pmatrix}
w_{\eta}^2+\theta \mu^{\theta} & -w_{\eta} \\
-w_{\eta}. & 1
\end{pmatrix}\ee
This matrix is positive, with
\be \text{tr}(A)=1+w_{\eta}^2+\theta \mu^{\theta} , \quad \det(A)=\theta \mu^{\theta}. \ee
Thus, $\lambda I \leq A \leq \Lambda I$ in $\mathcal{R}$, where $\lambda$, $\Lambda$ depend only on $\|\mu\|_{L^{\infty}(\R\times (\tau_0,\infty))},\delta^{-1}, \theta, \theta^{-1}, \|Dw\|_{L^{\infty}(\{\mu >0\})}$. Note also that, in view of \eqref{w defi} and \eqref{Dv bd assumption}, $\|Dw\|_{L^{\infty}(\{\mu >0\})}$ is bounded in terms of   $\|Dv\|_{L^{\infty}(\{\mu >0\})}$ and $R$.
Additionally, any vertical translation of $w$ satisfies \eqref{w elliptic eq}. Thus, since the conclusion \eqref{Dw C^1 bd} is also invariant under vertical translation, we may assume that
\be \|w \|_{W^{1,\infty}(\{\mu>0\})} \leq K_1=K_1(R,\|Dv\|_{L^{\infty}(\{\mu>0\})}). \ee
Writing, $S_{\delta}=[\eta_1-\delta,\eta_1+\delta]\times [\tau_1-\delta,\tau_1+\delta]$, the interior $C^{1,\beta}$ estimates for quasilinear elliptic equations (\cite[Thm. 13.6]{GilbargTrudinger}) imply that there exist $\beta\in (0,1)$ and $K_2>0$, depending only on  $K_1$, $\delta^{-1}$, $R$, and $\Lambda/\lambda$ such that
\be \label{schaud asccf} \|w\|_{C^{1,\beta}(S_{2\delta})}\leq K_2.\ee
Next, the interior Schauder estimates \cite[Thm. 6.2]{GilbargTrudinger} imply that there exists $K_3 >0$, depending only on $K_1, K_2, R, \Lambda, \lambda^{-1}$, and $\beta$ such that
\be \label{schaud asccf2} \| w \|_{C^{2,\beta}(S_{\delta})} \leq K_3.\ee
Thus, \eqref{schaud asccf} holds with $\beta=1/2$ and constant $K_3$, which, a posteriori, means \eqref{schaud asccf2} must also hold with $\beta=1/2$ and some constant $K_4$.
\begin{rem} \label{rem: Dv regu finite hor} The estimate of Lemma 3.5 is entirely local, so it only requires $(w,\mu)$ to solve \eqref{w mu sys} on $(\eta_1-4\delta,\eta_1+4\delta) \times (\tau_1-4\delta, \tau_1+4\delta).$ In particular, the same conclusion holds if $(u,m)$ is merely a solution to the finite time horizon problem \eqref{mfg}.     
\end{rem}
\end{proof}
\subsection{Analysis of a Lyapunov functional}
We define \be F(s)=\frac{s^{\theta+1}}{\theta+1}, \quad s\in [0,\infty),\ee and consider, given a solution $(w,\mu)$ to \eqref{w mu sys}, the functional
\be \label{Lyap} \cEw(\tau)=\intr  \left( \frac12 \mu |w_{\et}|^2 -\left(F(\mu)-F(M_a)-\left(R_a-\frac{\al(1-\al)}{2} \eta^2\right)(\mu-M_a) \right)\right)d\eta, \quad \tau \in \left[\tau_0,\infty\right). \ee
We note that, by construction, $\cEw \equiv 0$ when $(v_{\et},\mu)$ is the self-similar profile $(U'_a,M_a)$. 

In order to understand the asymptotics of the continuous rescaling $(w,\mu)$, we will first analyze $\cEw$. As it turns out, the qualitative behavior of $\cEw$ is dependent on the parameter $\theta$, with $\theta=2$ being the critical value (see \eqref{Lyap derivative}, and Propositions \ref{prop: E behavior} and \ref{prop: exponential theta<2} below). 
\begin{lem}\label{lem: E basic} Under the assumptions of Proposition \ref{prop: convergence result}, let $(w,\mu)$ be given according to Definition \ref{def: cont resc inf} and \eqref{w defi}, and let $\cEw:\left[\tau_0,\infty\right) \to \R$ be defined by \eqref{Lyap}. Let $K>0$ be a constant such that, for every $\tau_1 \in (\tau_0+\rho,\infty)$, 
\be \label{energy bds kqaxos} \|\mu^{\theta}\|_{H^1(\R \times (\tau_1-\rho,\tau_1+\rho))} \leq K. \ee
Then $\cEw \in C^{1,1/2}(\left[\tau_0,\infty)\right)$, with
\be \label{Lyap derivative}\frac{d}{d\tau}\cEw(\tau)= \frac{\theta-2}{\theta+2} \int_{\R}\mu w_{\eta}^2 d\et, \ee
and
\be \label{Lyap energy bd} \| \cEw \|_{H^2((\tau_1-\rho,\tau_1+\rho))} \leq C,\ee
where $C=C(K,R,\|\mu\|_{L^{\infty}(\R \times (\tau_0,\infty))}, \|Dw\|_{L^{\infty}(\{\mu>0\})})$.
    
\end{lem}
\begin{proof}  Let $\tau_0<b<c<\infty$. For each $\vep>0,$ let $h \in C^{\infty}_c(\R)$ satisfy $0\leq h \leq 1$, $h \equiv 1$ in $[2\vep,1-2\vep],$ $h \equiv 0$ outside of $[\vep,1-\vep]$, and $|h'|\leq K_1/\vep$. Let $\gamma_L < \gamma_R$ be the (Lipschitz) free boundary curves of $m$, so that
\[\beta_L(\tau)=t^{-\al}\gamma_L(e^{\tau}), \quad \beta_R(\tau)=t^{-\al}\gamma_R(e^{\tau}), \quad \tau \in \left[\tau_0,\infty\right) \]
are the (Lipschitz) free boundary curves for $\mu$, that is,
\[\{\mu(\cdot,\tau)>0\}=(\beta_L(\tau),\beta_R(\tau)), \quad \tau \in \left[\tau_0,\infty\right).\]
Set
\[\vfi(\eta,\tau)=  h\left((x-\beta_L(\tau))/(\beta_R(\tau)-\beta_L(\tau))\right), \quad (\eta,\tau)\in \R \times [b,c].\]
Note that $\vfi$ is compactly supported  in $\{\mu>0\}$, and Lipschitz, with 
\[|D\vfi|\leq K_{b,c}/\vep,\]
for some constant $K_{b,c}>0$ depending on $b,c, \beta_L, \beta_R,$ and $K_1$.

We recall that \eqref{w mu sys} holds classically in $\{\mu>0\}$. Multiplying the first and second equations in \eqref{w mu sys} by $\mu_{\ta}$ and $v_{\ta}$, respectively, and adding the two resulting equations, we obtain
\[\mu_{\ta}\frac{1}{2}w_{\et}^2-(\mu w_{\et})_{\et}w_{\ta}= \mu^{\theta}\mu_{\ta} + \frac{\al(1-\al)}{2}\et^2 \mu_{\ta} +(2\al-1)w\mu_{\ta}.\]
Integrating against $\vfi,$ using the fact that $w\in C^2(\{\mu>0\})$, and integrating by parts, we obtain
\be \label{lyap pf 1asdf} \int_{b}^{c} \intr  \frac{\partial }{\partial \ta}\left( \frac12 \mu w_{\et}^2 
-\left(\frac{\mu^{\theta+1}}{\theta+1}+\frac{\al(1-\al)}{2} \et^2 \mu\right)\right) \vfi d\eta d\tau = \int_{b}^{c} \intr (2\al-1)w \mu_{\ta}\vfi d\et d\ta +E_1, \ee
where
\[E_1=\int_{b}^{c} \intr -\mu w_{\et}w_{\ta}\vfi_{\et} d\et d\ta.\]
On the other hand,
\be \label{lyap pf 2asdf}  \int_{b}^{c} \intr w\mu_{\ta}\vfi d\et d\ta  =  \int_{b}^{c} \intr w(\mu w_{\et})_{\et}\vfi d\et d\ta=
\int_{b}^{c} \intr -\mu w_{\et}^2\vfi d\et d\ta +E_2,\ee
where 
\[|E_2|=\left|\int_{b}^{c} \intr -w\mu w_{\et}\vfi_{\et} d\et d\ta\right|\leq C_{b,c,w} \int_{b}^{c}\vep^{-1}\int_{\text{dist}(\et,\{\beta_L(\ta),\beta_R(\ta)\})\leq \vep} \mu=o(1) \]
as $\vep \to 0,$ by the continuity of $\mu$, which vanishes at $\eta= \beta_L(\ta)$ and $\eta=\beta_R(\ta)$. Similarly,
\[  E_1=\int_{a}^b\intr -\mu w_{\et}w_{\ta}\vfi_{\et} = o(1),\]
and
\[ \int_{b}^{c} \intr  \left( \frac12 \mu w_{\et}^2 - \left(\frac{\mu^{\theta+1}}{\theta+1}+\frac{\al(1-\al)}{2} \et^2 \mu\right)\right) \vfi_{\tau} d\eta d\tau =  o(1). \]
Thus, \eqref{lyap pf 1asdf} and \eqref{lyap pf 2asdf} yield
\be  \intr  \left( \frac12 \mu w_{\et}^2 - \left(\frac{\mu^{\theta+1}}{\theta+1}+\frac{\al(1-\al)}{2} \et^2 \mu\right)\right)  \vfi d\eta \Bigg|_{b}^c =(1-2\al) \int_{b}^{c} \intr \mu w_{\eta}^2 \vfi d\et d\ta +o(1). \ee
Noting that $\vfi \to \chi_{\{\mu>0\}}$ pointwise as $\vep \to 0$, we conclude by dominated convergence that
\be \label{lyap pf 3asdf}\frac{d}{d\tau}  \intr  \left( \frac12 \mu w_{\et}^2 - \left(F(\mu)+\frac{\al(1-\al)}{2} \et^2 \mu\right)\right) d\eta  = (1-2\al) \intr \mu w_{\eta}^2 d\et.  \ee 
 On the other hand, we clearly have
\be \label{lyap pf 4asdf}\frac{d}{d\tau}  \intr \left( F(M_a(\et)) + \frac{\al(1-\al)}{2} \eta^2 M_a(\et)\right) d\et =0,\ee
and, since $\mu$ and $M_a$ have the same mass, 
\be \label{lyap pf 5asdf} \frac{d}{d\tau} \intr R_a (\mu-M_a) d\et= \frac{d}{d\tau} 0 =0. \ee
Adding \eqref{lyap pf 3asdf}, \eqref{lyap pf 4asdf}, and \eqref{lyap pf 5asdf}, we obtain \eqref{Lyap derivative}.  A posteriori, it follows from the the continuity of $\mu$ and the interior $C^1$ regularity of $w$ that the right hand side of \eqref{Lyap derivative} is continuous, which means that $\cEw\in C^{1}(\left[ \tau_0,\infty \right))$.
Through similar test function arguments, and using \eqref{energy bds kqaxos}, the following formal computations also hold:
\be \small \label{lyap pf 6asdf} \frac{d}{d\tau} \intr F(\mu)d\et =\intr F'(\mu) \mu_{\tau}d\et  = \intr \mu^{\theta} (\mu w_{\eta})_{\et}d\et =\intr - \theta \mu ^{\theta}\mu_{\et}w_{\et}d\et =-\intr \theta (F(\mu))_{\et}w_{\et}d\et \in L^2_{\text{loc}}((\tau_0,\infty)) , \ee
and 
\be \small \label{lyap pf 7asdf} \frac{d}{d\tau} \intr \frac12 \eta^2 \mu d\et  = \intr \frac12 \eta^2  (\mu w_{\et})_{\et}d\et = \intr -\et w_{\et}\mu d\et \in  L^{\infty}((\tau_0,\infty)).   \ee
Thus, \eqref{lyap pf 3asdf}, \eqref{lyap pf 6asdf}, and \eqref{lyap pf 7asdf} yield
\be \label{lyap pf 8asdf}\frac{d}{d\tau} \intr \frac12 \mu w_{\et}^2 = (1-2\al) \int_{\R}\big(\mu w_{\et}^2 - \theta (F(\mu))_{\et}w_{\et}- \al(1-\al) \et w_{\et} \mu \big)d\et. \ee
 The first and third terms on the right hand side of \eqref{lyap pf 8asdf} are bounded, and the second term is in $L^2_{\text{loc}}((\tau_0,\infty))$ by \eqref{energy bds kqaxos}. Thus, \eqref{Lyap derivative} combined with \eqref{lyap pf 8asdf} yields the $H^2 \hookrightarrow C^{1,1/2}$ regularity of $\cEw$, and the quantitative claim \eqref{Lyap energy bd} follows from \eqref{energy bds kqaxos}.
\end{proof}
We emphasize that the form of \eqref{Lyap derivative} illustrates the criticality of the value $\theta=2$, which will require a separate treatment. That is the purpose of the following lemma.
\begin{lem}  \label{lem:lasrylionstheta=2}  Under the assumptions of Proposition \ref{prop: convergence result}, let $(w,\mu)$ be given according to Definition \ref{def: cont resc inf} and \eqref{w defi}. Assume that $\theta=2$, and let
\be \label{f bd} f(\ta)= \intr w(\et,\ta)\big(\mu(\et,\ta)-M_a(\et)\big) d\et, \quad \ta \in \left[\ta_0,\infty\right). \ee
Then $f\in L^{\infty}((\ta_0,\infty))$, and
\be \label{f deriv} f'(\ta)= - \intr \left( \frac{\mu+M}{2}|w_\et|^2 + \left(\mu^2-\left(R_a - \frac{\al(1-\al)}{2}\et^2\right)(\mu-M_a) \right) \right )d\et .\ee
\end{lem}
\begin{proof} The boundedness of $f$ follows because, since $\mu$ and $M_a$ have the same mass $a$, \eqref{Dv bd assumption} implies
\be |f(\ta)|=\left|\intr \left(w(\et,\ta)- (2R)^{-1}\int_{-R}^R w(y,\ta)dy\right)(\mu-M_a)d\et\right|\leq 2a\|w_{\et}\|_{L^{\infty}(-R,R)}.\ee
The point now is that, since the $2\al-1$ factor in the equation of $w$ \eqref{w mu sys} vanishes when $\theta=2$, the standard Lasry-Lions computation becomes especially useful. Indeed, we rewrite the first equation of \eqref{w mu sys} as
\be -w_{\ta} +\frac{1}{2}w_{\et}^2= \mu^{2} - \left(R_a - \frac{\al(1-\al)}{2}\et^2\right)+ R_a,\ee
so that, using once more that $\mu$ and $M_a$ have equal mass, we have, for $\tau_0<b<c<\infty,$
\be \label{LL theta=2 asdmfs} \int_{b}^{c}\int_{\R}\left(-w_{\ta} +\frac{1}{2}w_{\et}^2\right)(\mu-M_a ) d\et= \int_{b}^{c} \intr \left( \mu^2- \left(R_a - \frac{\al(1-\al)}{2}\et^2\right)\right)(\mu-M_a) d\et. \ee 
Keeping in mind that $\partial_{\ta}M_a\equiv 0$, equation \eqref{f deriv} now follows by testing the second equation of \eqref{w mu sys} against $w \chi_{\{\ta \in [b,c]\}}$, and subtracting it with \eqref{LL theta=2 asdmfs}. 
\end{proof}
We can now show the main convergence result of this section.
\begin{proof}[Proof of Proposition \ref{prop: convergence result}]
First, consider the function $\vfi(\tau)=\cE_{w_{\et},\mu}(\ta)$, which, is monotone and bounded over the interval $(\ta_0,\infty)$, due to Lemma \ref{lem: E basic}. Consequently, it approaches a finite limit as $\ta \to \infty$. Moreover, the uniform $C^{1,1/2}$ regularity estimate \eqref{Lyap energy bd} leads us to conclude that $\vfi'(\ta) \to 0$. That is,
 \be \label{lim E'=0} \lim_{\ta \to \infty}\frac{\theta-2}{\theta+2} \int_{\R} \mu w_{\et}^2d\et = 0. \  \ee 
 Let $s_n\in (\tau_0,\infty)$ be an arbitrary sequence of time delays such that $s_n \to \infty$, and define the solutions \[v_n(\eta,\tau)=v(\eta,\ta+s_n), \quad \mu_n(\eta,\tau)=\mu(\eta,\ta+s_n),\quad (\et,\ta) \in  \R \times \left[\ta_0-s_n,\infty \right).\]
 We also define the corresponding $w_n$ according to \eqref{w defi}. 
Then, in view of \eqref{mu cont assumption}, up to extracting a subsequence, the Arzel\`a-Ascoli theorem implies that $\mu_n$ converges locally uniformly in $\R \times \R$ to a function $\overline{\mu} \in \text{BUC}(\R \times \R)$. Note that $(v_n,\mu_n)$ satisfies both equations in \eqref{w mu sys}, so we may apply Lemma \ref{lem:regu Dw in mu>0}. Therefore, by \eqref{Dv bd assumption} and \eqref{Dw C^1 bd}, up to a subsequence, $Dw_n=(\partial_{\et}w_n,\partial_{\ta}w_n)$ also converges locally uniformly in $C^1(\{\overline{\mu}>0\})$, to a pair of $C^1(\{\overline{\mu}>0\})$ functions $(q,r)$, which are bounded in $L^{\infty}(\{\overline{\mu}>0\})$. In particular, extending the bounded function $q$ to be $0$ outside of $\{ \overline{\mu}>0 \}$, we have
\be  \label{mu w^2 conv wkoadoad}\mu_n |w_n|^2 \to \overline{\mu} |q|^2\,\,\,\, \text{ locally uniformly in }\,\R \times \R. \ee 
Assume first that $\theta \neq 2,$ so that $2\al-1 \neq 0$. Then \eqref{mu cont assumption}, \eqref{Dv bd assumption}, and the first equation in \eqref{w mu sys} imply that $w$ is uniformly bounded in $\{\mu>0\}$. Hence, $w_n$ converges locally uniformly in $\{ \overline{\mu}>0 \}$ to some bounded continuous function $\overline{w}: \{ \overline{\mu}>0 \} \to R$, which must then be $C^2$ and satisfy $D\overline{w}=(q,r)$. Solving for $\mu_n$ in the first equation of \eqref{w mu sys} also shows that $\overline{\mu}$ is $C^1$ in $\{\overline{\mu}>0\}$, and the pair $(\overline{w},\overline{\mu})$ is a classical solution to \eqref{w mu sys} in $\{\overline{\mu}>0\}$. 

 Furthermore, \eqref{Dv bd assumption}, \eqref{lim E'=0}, and \eqref{mu w^2 conv wkoadoad} imply that $\overline{\mu}\,\overline{w_{\et}}^2\equiv 0$, that is, $\overline{w}_{\et}=0$ in $\{\overline{\mu}>0\}$. Fixing $\tau_1 \in \R$, we infer that the function $w$ and its derivatives are independent of $\eta$ on every connected component of $\{\overline{\mu}(\cdot,\ta_1)>0\}$. It then follows from the first equation in \eqref{w mu sys} that the function $\overline{\mu}^{\theta}+ \et^2 \frac{\al(1-\al)}{2}$ is independent of $\eta$ on each connected component of $\{\overline{\mu}(\cdot,\ta_1)>0\}$.
 
 We claim that the set $\{\overline{\mu}(\cdot,\ta_1)>0\}$ is an interval, independent of $\ta_1$. Indeed, if $\eta_1\neq 0$ is any number such that $\overline{\mu}(\eta_1,\ta_1)>0$, then there exists a constant $R_{\eta_1,\ta_1}>0$ such that
\be \label{conv pf asxsao}\overline{\mu}^{\theta}(\et,\ta_1)=R_{\eta_1,\ta_1}- \et^2 \frac{\al(1-\al)}{2}>0 \ee
for every $\eta$ in the connected component of $\{\overline{\mu}(\cdot,\ta_1)>0\}$ containing $\eta_1$. However, by continuity of $\overline{\mu}$, \eqref{conv pf asxsao} must then hold for $\et\in [-|\eta_1|,|\eta_1|]$. We conclude that \eqref{conv pf asxsao} must hold for all $\et \in \{\overline{\mu}(\cdot,\ta_1)>0\}$ for a constant $R_{\ta_1}=R_{\eta_1,\ta_1}$, independent of $\eta_1$. Since $\overline{\mu}$ has mass $a$, the only possible choice is $R_{\ta_1}=R_a$. Thus, since $\ta_1$ was arbitrary, it follows that, for all $(\eta,\ta)\in \R \times \R$, \[ \overline{\mu}(\eta,\ta)=\left( R_a- \frac{\al(1-\al)}{2}\et^2\right)_+^{\frac{1}{\theta}}=M_a(\eta).\] 

In summary, since the sequence $s_n$ was arbitrary, we have shown that, as $\ta \to \infty,$ $\mu(\cdot,\ta) \to M_a(\cdot)$ and $\mu w_{\et}^2(\cdot, \ta)\to 0$ uniformly, in $\R$. In view of \eqref{Dv bd assumption}, we infer that
\eqref{conv v mu} holds for all $p\in [1,\infty]$. By Lemma \ref{lem:reduc to cont resc}, this shows that \eqref{IHH conv result 1} and \eqref{IHH conv result 2} hold. As for \eqref{IHH conv result 3}, it readily follows from the fact that $(u,m)$ and $(\mathcal{U}_a,\mathcal{M}_a)$ both solve the first equation of \eqref{mfgi}.

For the case $\theta=2$, \eqref{lim E'=0} now provides no information, so we must proceed differently. Letting $f:\left[\ta_0,\infty\right) \to \R$ be as in Lemma \ref{lem:lasrylionstheta=2}, we have that $f$ is bounded, and
\begin{multline} \label{log lasry vmu}  f'(\ta) = - \intr \left( \frac{\mu+M_a}{2}|w_\et|^2 + \left(\mu^2-\left(R_a - \frac{\al(1-\al)}{2}\et^2\right) \right)(\mu-M_a)\right) d\et\\
\leq - \intr \left( \frac12 \mu|w_\et|^2 + \left(\mu^2-\left(R_a - \frac{\al(1-\al)}{2}\et^2\right) \right)(\mu-M)\right) d\et. \end{multline}
We note that the right hand side is always non-positive, so that $f$ is non-increasing. Indeed, recalling that $M_a^2=M_a^{\theta}=\left(R_a-\frac{\al(1-\al)}{2}\et^2\right)_+$, we have
\be \label{quant pos dcoaia} \left(\mu^2-\left(R_a - \frac{\al(1-\al)}{2}\et^2\right) \right)(\mu-M_a)=\begin{cases}\left(\mu^2-M_a^2\right)(\mu-M) & M_a>0\\
\left(\mu^2 - \left(R_a - \frac{\al(1-\al)}{2}\et^2\right) \right)\mu & M_a=0
\end{cases} \geq 0.
\ee
Furthermore, the right hand side of \eqref{log lasry vmu} is also uniformly $\frac12$--H\"older continuous. Indeed, recalling \eqref{lyap pf 8asdf} and \eqref{mu cont assumption}, we have uniform $H^1 \hookrightarrow C^{\frac12}$ bounds on the first term. Similarly, the second term is readily seen to be bounded, by \eqref{lyap pf 7asdf} and \eqref{mu cont assumption}. 

We infer then that $\lim_{\ta \to \infty}f(\ta)$ exists, and that $\lim_{\tau \to \infty} f'(\ta)=0$. That is, the right hand side of \eqref{log lasry vmu} must tend to $0$ as $\ta \to \infty$. Thus, it directly follows from \eqref{quant pos dcoaia} that $q\equiv 0$ and $\overline{\mu}=M_a$, which concludes the proof.

\end{proof}
Next, we explore the extent to which $\cEw$ is a Lyapunov functional for the solutions to \eqref{w mu sys}, and the extent to which this depends on the value of $\theta$.
\begin{prop} \label{prop: E behavior} Under the assumptions of Proposition \ref{prop: convergence result}, let $(w,\mu)$ be given according to Definition \ref{def: cont resc inf} and \eqref{w defi}, and let $\cEw:\left[\tau_0,\infty\right) \to \R$ be defined by \eqref{Lyap}. Then the following holds:
\begin{enumerate}
    \item If $\theta \leq 2$ (resp. $\theta \geq 2$), then $\cEw(\ta) \geq 0$ (resp. $\cEw(\ta) \leq 0$) for all $\ta \in \left[ \ta_0, \infty \right)$. 
    \item If $\theta \neq 2$, then $\cEw(\ta_1)=0$ for some $\ta_1$ if and only if $(w_{\et}(\cdot,\ta),\mu(\cdot,\ta))=(0,M_a)$ for all $\ta \geq \ta_1$. 
   
\end{enumerate}    
\end{prop}
\begin{proof}  For concreteness, we assume first that $\theta \leq 2$. By Proposition \ref{prop: convergence result} and Lemma \ref{lem:reduc to cont resc}, we have $(\mu w_{\et}^2,\mu) \to (0,M_a)$, and so 
\be \lim_{\tau\to\infty} \cEw (\tau)= \intr   \left( \frac12 0 -\left(F(M_a)-F(M_a)-\left(R_a-\frac{\al(1-\al)}{2} \eta^2\right)(M_a-M_a) \right)\right)d\eta =0.  \ee    
Since $\theta \leq 2,$ \eqref{Lyap derivative} implies that $\cEw$ is non-increasing. This proves that $\cEw \geq 0$, and, if $\cEw(\ta_1)=0$ for some $\tau_1$, then $\cEw(\ta)=0$ for $\ta \geq \ta_1$. But in the latter case, if $\theta=2$, \eqref{Lyap derivative} yields that $w_{\eta}(\cdot,\tau)\equiv 0$ in $\{\mu(\cdot,\ta)>0\}$ for $\tau \geq \ta_1$. As in the proof of Proposition \ref{prop: convergence result}, this readily implies that $\mu(\cdot,\tau)=M_a(\cdot,\tau)$ for $\ta \geq \ta_1$. The proof of parts 1 and 2 for $\theta > 2 $ is directly analogous.
\end{proof}
 Finally, we obtain exponential rate of convergence for the range $\theta \in (0,2)$, which will be important in handling the convergence of the finite horizon problem for these values of $\theta$. 
 \begin{prop}\label{prop: exponential theta<2}Under the assumptions of Proposition \ref{prop: convergence result}, let $(w,\mu)$ be given according to Definition \ref{def: cont resc inf} and \eqref{w defi}, let $\cEw:\left[\tau_0,\infty\right) \to \R$ be defined by \eqref{Lyap}, and assume that $\theta<2$. Let $K>0$ be a constant such that \eqref{energy bds kqaxos} holds. Then $\cEw(\ta)$ converges exponentially to $0$. Indeed, setting $k=\frac{2-\theta}{2+\theta}=2\alpha-1$,
    \be \label{exp conv E} 0\leq \cEw (\ta) \leq e^{-2k\ta}\cEw(\tau_0), \quad \ta \in \left[\ta_0,\infty \right).\ee
Moreover, $\mu w_{\et}^2 \to 0$ and $\mu \to M_a$ exponentially in the following sense:
\be \label{exp conv w} \intr \mu w_{\et}^2 d\et \leq Ce^{-k\ta}, \quad \ta\in (\ta_0+\rho,\infty),\ee
and
\be \label{exp conv mu} \intr F(\mu)-F(M_a) - \left(R_a - \frac{\al(1-\al)}{2} \et^2\right)(\mu-M_a) d\et \leq  Ce^{-k\ta}, \quad \ta\in (\ta_0+\rho,\infty),\ee 
where $C=C(K,R,\|\mu\|_{L^{\infty}(\R \times (\ta_0,\infty))},\|Dw\|_{L^{\infty}(\{\mu>0\})},\rho, \rho^{-1}).$
 \end{prop}
 
\begin{proof}
    Recalling the definition of $M_a$ (see \eqref{Mself}), we observe that the quantity 
\be \label{thequantity}F(\mu)-F(M_a)-\left(R_a-\frac{\al(1-\al)}{2} \eta^2\right)(\mu-M_a)\ee
is always non-negative. Indeed, if $M_a>0$, then it equals
\be F(\mu)-F(M_a)-F'(M_a)(\mu-M_a),\ee
which is non-negative because $F$ is convex. On the other hand, if $M_a=0$, then by definition of $M_a$ we must have $R- \frac{\al(1-\al)}{2}\eta^2 \leq 0$, and \eqref{thequantity} equals
\be F(\mu)-\left(R-\frac{\al(1-\al)}{2}\eta^2\right)\mu \geq 0.\ee
Therefore, in view of \eqref{Lyap derivative} and Proposition \ref{prop: E behavior}, setting $\vfi(\ta)=\cEw(\ta)$, we have
\be \vfi'(\ta)=\frac{2(\theta-2)}{\theta+2}\int_{\R}\frac{1}{2}\mu w_{\eta}^2=-2k\int_{\R}\frac{1}{2}\mu w_{\eta}^2\leq -2k \vfi(\ta). \ee
Consequently,
\be 0 \leq \vfi(\ta) \leq e^{-2k\ta}\vfi(\tau_0), \quad \ta \in (\ta_0,\infty),\ee
which shows \eqref{exp conv E}. Now, in view of \eqref{lyap pf 8asdf}, we have, for some constant $C>0$ that may increase at each step,
\be |\vfi ''(\ta)| \leq C(|\vfi'(\ta)|+ \|(\mu^{\theta})_{\et}(\ta)\|_{L^2(\R)}\sqrt{|\vfi'(\ta)|}+ \|\et^2 \mu(\ta) \|_{L^2(\R)}\sqrt{|\vfi'(\ta)|} )\leq C h(\ta) \sqrt{|\vfi'(\ta)|},\ee
where $\|h\|_{L^2((\ta-\rho,\ta+\rho))}\leq C.$ In view of \eqref{Lyap energy bd}, we may then estimate $\vfi'$ using interpolation. Indeed, for $\ta \geq \ta_0+\rho,$ and some constant $K_1=K_1(\rho^{-1}),$ we have
\begin{align} \|\vfi'\|_{L^{\infty}((\ta-\rho,\ta+\rho))} \leq & K_1 \|\vfi''\|_{L^2((\ta-\rho,\ta+\rho))}^{2/3}\|\vfi\|_{L^{\infty}((\ta-\rho,\ta+\rho))}^{1/3}\\
\leq &\left( K_1 \|h\|_{L^2((\ta-\rho,\ta+\rho))}^{2/3}\|\vfi'\|_{L^{\infty}(\ta-\rho,\ta+\rho)}^{1/3}\vfi(\ta_0)^{1/3}e^{2k\rho/3}\right)e^{-2k\ta}\\
\leq & C \|\vfi'\|_{L^{\infty}(\ta-\rho,\ta+\rho)}^{1/3} e^{-2k\ta/3}, \end{align}
which proves \eqref{exp conv w}. Finally, \eqref{exp conv mu} follows from \eqref{Lyap}, \eqref{exp conv E}, and \eqref{exp conv w}.
\end{proof}
\begin{rem}Note that, due to \eqref{tau eta defi IH}, the exponential rates of Proposition \ref{prop: exponential theta<2} correspond to algebraic rates in the original $(x,t)$ variables.    
\end{rem}
\section{Estimates for the continuous rescaling}\label{sec: estimates}
Starting with this section, we will focus on proving the main results of the paper exclusively for the terminal cost problem \eqref{mfg}--\eqref{tc}. At the end, in Section \ref{sec:plan}, we will explain how to adapt the arguments to the planning problem \eqref{mfg}--\eqref{pp}. To fix ideas, we restate the definition of the continuous rescaling, for the finite horizon problem \eqref{mfg}.
\begin{defn} Given a solution $(u,m)$ to \eqref{mfg}, we define the continuous rescaling $(v,\mu)$ as follows. For $(\eta,\tau)\in \R \times \left(-\infty, \log(T)\right],$ we set
\be \label{tau eta defi} t=e^{\ta},\,\, x=t^{\al} \et,\ee
and we define
\be \label{v mu defi}
\mu(\et,\ta)=t^{\al}m(x,t), \, v(\et,\ta)=t^{1-2\al}u(x,t). 
\ee
\end{defn}

The main goal of this section will be to prove all the necessary a priori estimates for the pair $(v,\mu)$, namely those that appear in the hypotheses of Proposition \ref{prop: convergence result}
(see Propositions \ref{prop: mu bd}, \ref{prop: v Dv bd}, \ref{prop: v mu energy}, and \ref{prop: mu holder cont} below). This will amount to obtaining estimates for $(u,m)$ whose algebraic rate of growth as $t \to \infty$ is the optimal one.

\begin{rem}In actual fact, some of the estimates obtained here will be stronger than what is required in Proposition \ref{prop: convergence result}. For instance, in contrast with \eqref{mu cont assumption}, $\mu$ will be shown to be uniformly H\"older continuous (see Proposition \ref{prop: mu holder cont}), and the energy estimates will be stronger as well (see Proposition \ref{prop: v mu energy}). Of course, it is also already known from Theorem \ref{thm.intro1} that the actual regularity of the free boundary curves is also higher than what is required by Proposition \ref{prop: convergence result}.
    
\end{rem}
\subsection{Exact rates for the density and the free boundary curves}

The results of this subsection will be the simplest to derive, because they are quick consequences of Theorem \ref{thm.intro2}. We recall once more, for the reader's convenience, that $\alpha\in (0,1)$ is defined by \be \label{alpha defi}\alpha=\frac{2}{2+\theta}.\ee
It will also be frequently useful to rewrite this as,
\be \label{alpha theta} \alpha \theta = 2-2\al.\ee
We begin by noting the fact that $m(\cdot,t)$ decays like $t^{-\al}$, which is the same rate as that of $\mathcal{M}_a(x,t)=t^{-\al}M_a(xt^{-\al})$. 
\begin{prop}[Optimal $L^{\infty}$ rate for the density] \label{prop: smoothing effect}
    Under the assumptions of Theorem \ref{thm.intro1}, let $(u,m)$ be the solution to \eqref{mfg}--\eqref{tc}. There exists a constant $K>0$, $K=K(C_0,|a_0|,|b_0|)$, such that
    \begin{equation} \label{smoothing effect}
        \|m(\cdot,t)\|_{L^{\infty}(\R)}\leq \frac{K \|m_0 \|_{\infty}}{1+t^{\al}}, \quad t\in (0,T).
    \end{equation}    
\end{prop}
\begin{proof} By rewriting $\gamma_x$ in terms of $m$ and $m_0$ using \eqref{mass cons}, the first inequality in \eqref{gamma x bd} becomes
\[m(\gamma(x,t),t) \leq C \frac{m_0(x)}{1+t^{\al}},\]
which readily implies \eqref{smoothing effect}.

\end{proof}
\begin{rem}The statement of Proposition \ref{prop: smoothing effect} can be strengthened into a so-called $L^{1}$--$L^{\infty}$ smoothing effect, to yield
\be \|m(\cdot,t)\|_{L^{\infty}(\R)} \leq K\|m_0\|^{\al}_{L^1(\mathbb{R})}t^{-\al},\ee
where $K$ depends exclusively on $\theta$ and $\theta^{-1}$. The proof is completely independent from Theorem \ref{thm.intro2}; it relies instead on the displacement convexity formula \cite[Prop. 3.6]{CMP}, in a similar fashion to  \cite[Prop. 5.2]{Porretta}. This estimate is new, and could potentially be used to treat more general data, but we omit the proof because \eqref{smoothing effect} is sufficient for our present purposes.     
\end{rem}
We now explain why the free boundary curves, up to their second derivative, have the same rate of growth as the corresponding free boundary curves of $\mathcal{M}_a$, which have the form $x=\pm c_a t^{\al}$, where $c_a>0$. 

\begin{prop}[Optimal rates for the free boundary curves] \label{prop: free bd rates} Under the assumptions of Theorem \ref{thm.intro1}, let $(u,m)$ be the solution to \eqref{mfg}--\eqref{tc}. Assume also that $T\geq 2$. There exists a constant \[K=K(C_0,|a_0|,|b_0|)\]  such that the free boundary curves $\gamma_L, \gamma_R$ satisfy, for $t\in [0,T],$

\be |\gamma_L(t)| + |\gamma_R(t)| \leq K(1+t^{\al}), \label{gamL}
\ee 
\be \label{gamL'} |\dot \gamma_L(t)|+ |\dot \gamma_R(t)| \leq \frac{K}{1+t^{1-\al}}, 
\ee
    \begin{equation} \label{gamL''}
\frac{1}{K(1+t^{2-\al})}\leq \Ddot\gamma_{L}(t)\leq \frac{K}{(1+t^{2-\al})},\,\text{ and }\,-\frac{K}{(1+t^{2-\al})}\leq \Ddot\gamma_{R}(t)\leq-\frac{1}{K(1+t^{2-\al})}.
\end{equation}
\end{prop}
\begin{proof}
Since the statement is symmetrical, we only show the estimates for the left curve $\gamma_L$. The estimate \eqref{gamL} follows by recalling \eqref{scrd defi} and \eqref{gamma regu}, and setting $x=a_0$ in \eqref{gamma bd}. Now, from \cite[Thm 4.14]{CMP}, one has the explicit second derivative estimate
\[ 0< f(m_0)_x (a_0) \|\gamma_x(\cdot,t)\|_{\infty}^{-(1+\theta)}\leq \Ddot \gamma_L(t) \leq f(m_0)_x (a_0) \| \gamma_x(\cdot,t)^{-1}\|_{\infty}^{1+\theta}.  \]
It then follows from \eqref{m0 bump} and \eqref{gamma x bd} that \[ \frac{1}{C}\frac{1}{(1+t^{\al})^{1+\theta}} \leq \Ddot \gamma_L(t) \leq C \frac{1}{(1+t^{\al})^{1+\theta}}.\]
Thus, recalling \eqref{alpha theta}, up to increasing the constant $C$, we obtain \eqref{gamL''}. 

Finally, \eqref{gamL'} follows by basic interpolation. For instance, if $t \leq  (T-1)/2,$ one may write $\gamma_L(2t+1)-\gamma_L(t)=(t+1)\dot \gamma_L(t+h),$ for some $h\in (0,t+1)$, so that
\[ |\dot \gamma_L(t)|\leq |\dot\gamma_L(t)-\dot \gamma_L(t+h)|+|\dot \gamma_L(t+h)|\leq (t+1) \|\Ddot \gamma_L\|_{L^{\infty}(t,2t+1)}+ \frac{2}{t+1} \|\gamma_L\|_{L^{\infty}(t,2t+1)}, \]
and \eqref{gamL'} follows from \eqref{gamL} and \eqref{gamL''}. For $t>(T-1)/2\geq 1/2$, we instead write $\gamma_L(t)-\gamma_L(t/2)=(t/2)\dot \gamma_L(t-h)$, for some $h\in(0,t/2)$, and then proceed in the same way.
\end{proof}
We are now in a position to show the first set of estimates for the continuous rescaling, namely the bounds for $\mu$ and its support.
\begin{prop} [Uniform bound on the continuous rescaling $\mu$] \label{prop: mu bd} Under the assumptions of Theorem \ref{thm.intro1}, let $(u,m)$ be the solution to \eqref{mfg}--\eqref{tc}, and let $\mu$ be the continuous rescaling \eqref{v mu defi}. There exists a constant $K=K(C_0,|a_0|,|b_0|)$ such that
\be \label{mu linfty}\| \mu \|_{L^{\infty}\left(\R \times (-\infty, \log(T))\right)}\leq K \|m_0\|_{L^{\infty}(\R)}, \ee
Furthermore, for every $M>0$, there exists $K_1=K_1(C_0,|a_0|,|b_0|,M)$ such that, for every $\tau \in [-M,\log(T)]$,
\be \label{mu suppt bd}\emph{supp}(\mu(\cdot,\tau)) \subset [-K_1,K_1].\ee 
\end{prop}
\begin{proof}
For arbitrary $\tau\in (-\infty,\log(T)),$ we write $t=e^{\tau}$. We infer from \eqref{v mu defi} and \eqref{smoothing effect} that
\[ \|\mu(\cdot,\tau)\|_{L^{\infty}(\R)} =t^{\al}\|m(\cdot,t)\|_{L^{\infty}(\R)} \leq \frac{t^{\al} K \|m_0 \|_{L^{\infty}(\R)} }{(1+t^{\al})}\leq K\|m_0 \|_{L^{\infty}(\R)}. \]
On the other hand, \eqref{freebd ch}, \eqref{tau eta defi}, and \eqref{v mu defi} yield
\[ \{ \mu(\cdot,\tau) >0\}=(\gamma_L(t)/t^{\al},\gamma_R(t)/t^{\al}).
\]
Thus, if $\tau \in [-M,\log(T)]$, we have $t =e^{\tau} \geq e^{-M}:=\delta >0$, and \eqref{gamL} implies, for some constant $K_1$ depending on $\delta$, that
\[\text{supp}(\mu(\cdot,\tau)) =[\gamma_L(t)/t^{\al},\gamma_R(t)/t^{\al}] \subset [-C(1+t^{\al})/t^{\al},C(1+t^{\al})/t^{\al}]  \subset [-K_1,K_1].\]

\end{proof}
\subsection{Exact rates for the value function and its gradient} \label{subsec: grad bds}
We focus now on obtaining the optimal rates for the value function $u$ and its first derivatives. It will be useful to remark that, since $u\in W^{1,\infty}(\R \times (0,T))\cap C^1(\R \times (0,T))$ is a classical solution to the HJ equation in \eqref{mfg}, it satisfies the dynamic programming principle. That is, for $(x,t)\in \R \times (0,T)$ and $0<h<T-t,$
\be \label{dpp}
u(x,t)=\inf_{\beta \in H^1((t,t+h)), \,\beta(t)=x} \quad\int_t^{t+h} \left( \frac{1}{2} |\Dot \beta(s)|^2+m(\beta(s),s)^{\theta} \right) ds + u(\beta(t+h),t+h).
\ee
We define, for $s \leq t$, 
\[ S(m,s,t)=\text{supp}(m) \cap (\R \times [s,t]),\] and begin by estimating the oscillation of $u$ inside the support of $m$.
\begin{prop}[Optimal rate for the oscillation of $u$] \label{prop:osc u} Under the assumptions of Theorem \ref{thm.intro1}, let $(u,m)$ be the solution to \eqref{mfg}--\eqref{tc}.  For every $\delta>0$, there exists a constant $K=K(C_0,|a_0|,|b_0|,\delta^{-1})$ such that, for every $t_0\in [\delta,T/2] $,
   \be \label{oscu local bound} \underset{S\left(m,\frac{t_0}{2},2t_0\right)}{\emph{osc}} (u)\leq  
  Kt_0^{2\al-1}.
   \ee
\end{prop}
\begin{proof} In view of Theorem \ref{thm.intro1}, we know that $u\in C^1(\R \times (0,T))$, and
\be \label{supdefxzav}S:=S\left(m,\frac{t_0}{2},2t_0\right)=\left\{(x,t) : \frac{t_0}{2}\leq t \leq2t_0,\gamma_L(t) \leq  x \leq \gamma_R(t)\right\}.\ee
We claim that 
\be \label{eq: osc u asagsaf}\max_{S} u = \max_{S\cap \{t=t_0/2\}}u,\,\, \text{ and } \,\,\, \min_{S} u = \min_{S\cap \{t=2t_0\}}u. \ee
 To see this, we switch to Lagrangian coordinates, setting $\overline{u}(x,t)=u(\gamma(x,t),t)$. In these coordinates, the set $S$ becomes the rectangle $[a_0,b_0] \times [t_0/2,2t_0].$ Therefore, we infer from \eqref{flow defi}, \eqref{gamma regu}, and \eqref{mfg} that
 \[ \overline{u}_t(x,t)=u_t(\gamma(x,t),t)+\gamma_t(x,t) u_x(\gamma(x,t),t) = u_t-u_x^2=-m^{\theta}-\frac{1}{2}u_x^2\leq 0,\]
 which implies \eqref{eq: osc u asagsaf}. Now, given $(x,t_0/2)\in S\cap\{t=t_0/2\}$ and $(y,2t_0)\in S\cap\{t=2t_0\} $, we let $\beta:[t_0/2,2t_0] \to \mathbb{R}$ be the linear function joining $(x,t_0/2)$ and $(y,2t_0)$. As a result of \eqref{gamL} and \eqref{supdefxzav}, and recalling that $t_0\geq \delta>0,$ we have, for some constant $C>0$ depending on $\delta$,
 \[ |\dot\beta|= \frac{|y-x|}{t_0}\leq  C t_0^{\al-1}. \]
From Proposition \ref{prop: smoothing effect}, recalling \eqref{alpha theta}, and using $\beta$ as a competitor in \eqref{dpp}, we infer that, up to increasing the value of $C$,
 \[u(x,t_0/2) \leq \int_{\frac{t_0}{2}}^{2t_0} \left(\frac12|\dot\beta(s)|^2 + m(\beta(s),s)^{\theta}\right)ds + u(y,2t_0) \leq Ct_0^{2\al-1} + u(y,2t_0).\]
 Thus, \eqref{oscu local bound} follows from \eqref{eq: osc u asagsaf}. 
\end{proof}
We turn our attention to obtaining the optimal rate for the gradient. We will first reduce the problem to estimating $u_x$ on the region $\{m>0\}$, where $u$ is known to be smooth and its oscillation has now been optimally quantified. This reduction is done through optimal control arguments that exploit the convex shape of the support. 

To make these geometric arguments more transparent, we first characterize the optimal trajectories of \eqref{urepform}. Of course, in the region where the density is positive, these trajectories are uniquely given by the (smooth) Lagrangian flow $\gamma$, according to Theorem \ref{thm.intro2}. The following proposition describes the remaining trajectories (up to interchanging $\gamma_L$ and $\gamma_R$).
\begin{prop}[Identification of the optimal trajectories in the vanishing set] \label{prop trajec} Under the assumptions of Theorem \ref{thm.intro1}, let $(u,m)$ be the solution to \eqref{mfg}--\eqref{tc}, and let $\gamma_L$ be its left free boundary curve. Let
\[ l(s)=\gamma_L(T)+\dot\gamma_L(T)(s-T), \quad s\in [0,T]\]
be the supporting line for the decreasing, convex function $\gamma_L$ at $t=T$. Assume that $(x,t)\in \R \times (0,T),$ with $x\leq \gamma_L(t)$, and let $\beta:[t,T] \to \R$ be an optimal trajectory for \eqref{urepform}. Then $\beta \in C^1([t,T])$ is uniquely determined as follows:
\begin{enumerate}
    \item If $x \leq \gamma_L(T)$, then $\beta$ is the straight line $\dot \beta \equiv  0$.
    \item If $\gamma_L(T) < x\leq l(t)$, then $\beta$ is the straight line joining the points $(x,t)$ and $(\gamma_L(T),T)$.
    \item If $l(t) < x \leq \gamma_L(t)$, then there exists $t^*\in (t,T)$ such that
   \[ \beta(s)= \begin{cases}x -u_x(x,t)(s-t) & \,\, s\in [t,t^*], \\
    \gamma_L(s) & \,\, s\in [t^*,T].        
    \end{cases} 
    \]
\end{enumerate}
\end{prop}
\begin{proof}  We will tacitly use the fact that, from Theorem \ref{thm.intro1}, the function $\gamma_L$ is strictly convex and strictly decreasing on $(0,T)$. The $C^{1}([t,T))$ regularity of $\beta$ follows from the fact that $u\in C^1(\R \times (0,T))$, and from the optimality condition for \eqref{dpp}, namely that $\beta$ solves the differential equation
\be \label{optimc} \begin{cases} \dot \beta(s)=-u_x(\beta(s),s) & \,\,s \in (t,T),\\
 \beta(t)=x.   
\end{cases} 
\ee
The $C^1$ regularity up to $t=T$ will follow later from the explicit form of $\beta$. Assume that $(x_1,t_1)\in \R \times (0,T)$ is such that $x_1<\gamma_L(t_1)$. Then, since $u\in C^1(\R \times (0,T)),$ is a classical solution to the HJ equation \[-u_t+\frac12 u_x^2=0\] near $(x_1,t_1)$, we must have $u \in C^{1,1}_{\text{loc}}$ at such points, with bounds that may degenerate near the interface (see \cite[Thm 15.1]{L-HJ book}). As a result, any characteristic curve passing through $(x_1,t_1)$ is locally a straight line. In particular, \eqref{optimc} has a unique solution at such points, and the minimizer $\beta(\cdot)$ is a uniquely determined line segment until (and if) it hits the interface $\gamma_L$. More precisely, $\beta$ is determined on the interval $(t,t^*)$ as the linear function with slope
\[ \dot\beta(s)=-u_x(x,t), \quad s\in [t,t^*),\]
where 
\[ t^*=\sup \{q \in (t,T) : m(\beta(s),s)=0 \text{ for every } s\in (t,q)\}.\] 
We also reiterate that, by Theorem \ref{thm.intro2}, the optimal trajectories in $\{m>0\} \cap (\R \times [0,T])$ are uniquely determined by the (smooth) Lagrangian flow $\gamma$. In particular,  such trajectories never touch the free boundary. For this reason, 
\be \label{beta zaxmsa} \beta(s)\leq \gamma_L(s),\,\,\, s\in [t,T]. \ee

We claim that if $t^*<T$, then $\beta\equiv \gamma_L$ on $[t^*,T]$. Indeed, otherwise there would be a maximal time $t^{**}\in [t^*,T)$ such that $\beta \equiv \gamma_L$ on $[t^{*},t^{**}]$. From \eqref{beta zaxmsa} and the $C^1$ regularity of $\beta$, we see that $\dot \beta(t^{**})=\dot \gamma_L(t^{**})$. By the definition of $t^{**}$ and the convexity of $\gamma_L$, $\beta$ must then be a straight line on $[t^{**},T],$ which yields
\begin{multline*}\int_{t^{**}}^{T} \left(\frac{1}{2}|\dot \beta|^2 +m(\beta(s),s)^{\theta} \right)ds + c_T m^{\theta}(\beta(T),T)=\int_{t^{**}}^{T} \frac{1}{2}|\dot \gamma_L(t^{**})|^2 ds >\int_{t^{**}}^{T} \frac{1}{2}|\dot \gamma_L(s)|^2 ds \\
= \int_{t^{**}}^{T} \left(\frac{1}{2}|\dot \gamma_L(s)|^2+ m(\gamma_L(s),s)^{\theta}\right) ds + c_T m^{\theta}(\gamma_L(T),T),  \end{multline*}
contradicting the optimality of $\beta$. 

Summing up, we have so far shown the following characterization: the minimizer $\beta\in C^{1}([t,T])$ satisfies $\beta\leq \gamma_L$, it is a straight line on the interval $[t,t^*]$, and coincides with $\gamma_L$ on the (possibly empty) interval $(t^*,T]$. We now proceed to the proof of the lemma, in each of the three cases.

First, assume that $x\leq \gamma_L(T)$. Then, by the monotonicity of $\gamma_L$, we have $m(x,s)=0$ for $s\in [t,T]$. Thus, the minimal cost $0$ in \eqref{urepform} must be $0$, which is achieved if and only if $\dot \beta \equiv 0$.

Next, assume that $\gamma_L(T) <x\leq l(t).$ Note that we must have $\beta(T)=\gamma_L(T)$. Otherwise, $\beta$ would intersect the (constant) optimal trajectory $s \mapsto \gamma_L(T)$ at some interior time $s<T$, which is impossible by the previously explained local uniqueness.  Thus, it is enough to show that $\beta$ is a straight line, that is, that $t^*=T$. But note that, by the strict convexity of $\gamma_L$, \[ \dot \beta(t)=\dot \beta(t^*)=\dot\gamma_L(t^*)< \dot \gamma_L(T)=\dot l(t),\] unless $t^*=T$. Thus, $\beta\leq l< \gamma_L$ on $[t,t^*]$,  unless $t^*=T$. However, by definition of $t^*$, $\beta(t^*)=\gamma_L(t^*)$, so $t^*=T$, as wanted.

Finally, assume that $l(t)<x\leq \gamma_L(t)$. We just need to show that $\beta$ is not a straight line. As in the previous case, we must have $\beta(T)=\gamma_L(T)$. But the linear function $l_1:[t,T] \to \R$ connecting $(x,t)$ and $(\gamma_L(T),T)$ has slope 
\[\dot l_1(t)=\frac{\gamma_L(T)-x}{T-t}<\frac{\gamma_L(T)-l(t)}{T-t}=\dot l(t)=\dot \gamma_L(T). \]
Thus, $l_1>\gamma_L$ near $s=T$, which means, by \eqref{beta zaxmsa}, that $\beta \neq l_1$. Hence, $\beta$ is not a straight line, as wanted.

\end{proof}
Having identified the optimal trajectories, we can now readily bound the space derivative of $u$ outside of $\{m>0\}$, in terms of the speed of the free boundary curves.
\begin{cor} 
\label{cor grad bd m=0} Under the assumptions of Theorem \ref{thm.intro1}, let $(u,m)$ be the solution to \eqref{mfg}--\eqref{tc}. Then, for every $(x,t)\in \R \times (0,T)$ such that $m(x,t)=0$, we have
\be \label{grad bd m=0}|u_x(x,t)|\leq \max(|\dot\gamma_L(t)|,|\dot \gamma_R(t)|).\ee
\end{cor}
\begin{proof}
By symmetry, we may assume that $x\leq \gamma_L(t)$. Let $\beta$ be the optimal trajectory for \eqref{urepform}. Note that, by \eqref{optimc}, $u_x(x,t)=-\dot \beta(t)$. The result now follows from the convexity of $\gamma_L$. Indeed, by Proposition \ref{prop trajec}, there are three possibilities:
\begin{enumerate}
    \item $\beta$ is constant, in which case $u_x(x,t)=0$.
    \item $\beta$ is a straight line, with the negative slope
    \[\dot \beta(t) =\frac{\gamma_L(T)-x}{T-t}\geq \frac{\gamma_L(T)-\gamma(t)}{T-t}>\dot \gamma_L(t).\]
    \item $\beta$ is a straight line on some interval $[t,t^*]$, with the negative slope
    \[\dot \beta(t)=\dot \beta(t^*)=\dot \gamma_L(t^*)\geq \dot \gamma_L(t).\]
\end{enumerate}
\end{proof}

To complete the estimates of this subsection, we employ a localized version of the standard gradient estimate (\cite[Thm. 3.9]{Munoz}, \cite[Prop. 3.10]{CMP}). This sort of time-localization was first done in \cite[Thm. 6.5]{Porretta}. However, our goal here is to leverage Proposition \ref{prop:osc u} and Corollary \ref{cor grad bd m=0} to obtain the sharp algebraic rate of decay as $t \to \infty$. 

\begin{prop}[Optimal rate for the gradient of $u$] \label{prop: grad rate} Under the assumptions of Theorem \ref{thm.intro1}, let $(u,m)$ be the solution to \eqref{mfg}--\eqref{tc}. For each $\delta>0$ and each $t_0 \in [\delta,T/2]$, there exists a positive constant $C=C(C_0,|a_0|,|b_0|,\theta, \theta^{-1},\delta^{-1})$ such that
\begin{equation}  \label{grad local bd}  \|u_t(\cdot,t_0)\|_{L^{\infty}(\R)}+\|u_x(\cdot,t_0)\|_{L^{\infty}(\R)}^2\leq Ct_0^{2(\al -1)}.
\end{equation}    
\end{prop}
\begin{proof}
Since the estimate is allowed to degenerate as $\delta \to 0$, we may assume, for simplicity, that $\delta=1$, so that $t_0\in [1,T]$. First note that, in view of the first equation of \eqref{mfg}, Proposition \ref{prop: smoothing effect}, and \eqref{alpha theta}, it suffices to show the estimate for $u_x$. Furthermore, in view of \eqref{gamL'} and Corollary \ref{cor grad bd m=0}, it suffices to bound $|u_x|$ in the support of $m$. We will use the fact that, $u$ is $C^{\infty}$ on $\{m>0\} \cap (\R \times (0,T))$ and satisfies the quasilinear elliptic equation
\be \label{quasilinear asxa} Q_u(u)=0, \text{ where } Q_u(v)=-v_{tt} + 2u_x v_{xt} -(u_x^2+\theta m^{\theta})v_{xx}=0. \ee
This fact is known, and is justified exactly as in the proof of Lemma \ref{lem:regu Dw in mu>0}. Let $t_0 \in [1,T/2]$, and write \[S:=S(m,t_0/2,2t_0)=\{(x,t): t_0/2 \leq t \leq t_0, \gamma_L(t) \leq x \leq \gamma_R(t)\}, \quad \omega=\underset{S}{\text{osc}} (u), \quad 
\tilde{u}=u-\min_{S} u ,\]
so that 
\be \label{utilamxoa} |\util| \leq \omega \ee
on $S$. Set
\[v(x,t)=\frac12 u_x^2 + \frac{\vep}{ t_0^{2\al}}\frac12 \util^2- \zeta_1(t) -\zeta_2(t), \quad \zeta_1(t)= \frac{A t_0^{2\al}}{(t-t_0/2)^2}, \quad \zeta_2(t)=\frac{A t_0^{2\al}}{(2t_0-t)^2},\]
where $\vep>0$ and $A>0$ are constants to be chosen later, independently of $t_0$.
We define the linearization of $Q_u$ as
\be \label{lineariz} L_u(v)=Q_u(v)-2(m^{\theta})_xv_x -\theta u_{xx} (-v_t+ u_x v_x).\ee
Let $(x,t)\in S$ be a maximum point of $v$ in $S$. Since $v(\cdot,2t_0)\equiv v(\cdot,t_0/2) \equiv -\infty$, we must have $t\in (t_0/2,2t_0)$. If $x=\gamma_L(t)$ or $x=\gamma_R(t)$, we may conclude immediately by \eqref{gamL'} and Corollary \ref{cor grad bd m=0}. Hence, we may assume that $\gamma_L(t) < x < \gamma_R(t)$, so that $(x,t)$ is an interior maximum. Then, in particular, $v_x(x,t)=0,$ so that, writing \be \label{k defi sxaval} k=\vep/t_0^{2\al},\ee
we have
\be \label{grad uxx asd} |u_{xx}|=\frac{\vep}{t_0^{2\al}}|\util| \leq k\omega. \ee
Differentiating the first equation in \eqref{mfg}, we see that \be \label{mthetax asxnj} (m^{\theta})_x=-u_{xt}+u_xu_{xx}.\ee 
Thus, through direct computation, we obtain from \eqref{quasilinear asxa}, \eqref{lineariz}, and \eqref{mthetax asxnj} that
\be \label{grad 1 asdaf} L_u\left(\frac12 u_x^2 \right)=-\theta m^{\theta} u_{xx}^2-(-u_{xt}+u_xu_{xx})^2  =-\theta m^{\theta} u_{xx}^2-((m^{\theta})_x)^2 \leq -((m^{\theta})_x)^2.\ee
Similarly, by \eqref{quasilinear asxa}, \eqref{utilamxoa}, \eqref{lineariz}, \eqref{grad uxx asd}, and \eqref{mthetax asxnj}, we obtain 
\begin{multline} \label{grad 13e0ksmx} L_u\left(k \frac12 \util^2 \right)= -k\theta m^{\theta} u_{x}^2 -k(-u_t +u_x^2)^2 -2k (m^{\theta})_xu_x\util  -k\theta u_{xx}(-u_t+ u_x^2)\util  \\ \leq  - k(-u_t+u_x^2)^2 + ((m^{\theta})_x)^2 +  k^2 u_x^2\omega ^2 +\frac{k}{4}(-u_t+u_x^2)^2 +2\theta^2 k u_{xx}^2 \omega^2 \\ \leq - \frac{3k}{4} (-u_t+u_x^2)^2 + ((m^{\theta})_x)^2 +   k^2 u_x^2\omega ^2+2\theta^2 k^3 \omega^4 \\ 
\leq  - \frac{3k}{4} (-u_t+u_x^2)^2 + ((m^{\theta})_x)^2 + \frac{k}{16}u_x^4 + 4 k^3 \omega^4 +2\theta^2 k^3 \omega^4. \end{multline}
Noting that $-u_t+u_x^2=m^{\theta}+\frac12 u_x^2 \geq \frac12 u_x^2 ,$ we obtain from \eqref{grad 13e0ksmx} that
\be \label{grad 2 sadf}
    L_u\left( \frac12 k \util^2\right) \leq  - \frac{k}{8}u_x^4 +((m^{\theta})_x)^2 +K_1 k^3 \omega^4,
\ee
 where $K_1=4+2\theta^2$. Next, we compute from \eqref{quasilinear asxa} and \eqref{lineariz} that
\be \label{zet1sacxas} L_u(-\zeta_1) = \frac{At_0^{2\al}}{(t-t_0/2)^2}\left( \frac{6}{(t-t_0/2)^2}+  \frac{2\theta u_{xx}}{t-t_0/2}\right)\leq \frac{At_0^{2\al}}{(t-t_0/2)^2}\left( \frac{7}{(t-t_0/2)^2}+  \theta^2 u_{xx}^2\right).
\ee
We let $K$ be the constant of Proposition \ref{prop:osc u}, so that
 \be \label{osclocrecadxsa}\omega\leq K t_0^{2\al-1}, \ee
 and require that $\vep < \frac{1}{2\theta K}.$ Recalling that $t\in (t_0/2,2t_0),$ we then obtain from \eqref{k defi sxaval},\eqref{grad uxx asd}, and \eqref{osclocrecadxsa} that
\[\theta^2u_{xx}^2\leq \theta^2 k^2 \omega^2 \leq \theta^2 k^2 K^2t_0^{4\al-2}=\frac{ \theta^2 \vep^2 K^2}{t_0^2} \leq \frac{1}{(2t_0)^2}\leq \frac{1}{(t-t_0/2)^2}.\]
Thus, \eqref{zet1sacxas} implies
\be \label{zet12sacx} L_u\left(-\zeta_1 \right) \leq \frac{8A t_0^{2\al}}{(t-t_0/2)^4}.\ee
For a fixed choice of $\vep,$ we require $A\geq 32/\vep$, which gives
\be \label{zet13sacx}\frac{8A t_0^{2\al}}{(t-t_0/2)^4} \leq \frac{ \vep A^2 t_0^{2\al}}{4(t-t_0/2)^4} = \frac{ \vep}{4t_0^{2\al}} \left( \frac{A t_0^{2\al}}{(t-t_0/2)^2}\right)^2= \frac{k}{4} \zeta_1^2. \ee
Combining \eqref{zet12sacx} and \eqref{zet13sacx}, we get
\be \label{zet1} L_u\left(-\zeta_1 \right) \leq \frac{k}{4} \zeta_1^2.  \ee 
In a similar way, we obtain
\be \label{zet2} L_u(-\zeta_2) \leq  \frac{k}{4} \zeta_2^2. \ee
Putting together \eqref{grad 1 asdaf}, \eqref{grad 2 sadf}, \eqref{zet1}, and \eqref{zet2}, and recalling that $(x,t)$ is an interior maximum point of $v$, we obtain
\[ 0\leq L_u(v) \leq -\frac{k}{8}u_x^4 + K_1 k^3 \omega^4 + \frac{k}{4}(\zeta_1^2+\zeta_2^2) \leq -\frac{k}{8}u_x^4 + K_1 k^3 \omega^4 + \frac{k}{4}\left(\zeta_1+\zeta_2\right)^2,\]
which implies that
\[ \left(\frac12 u_x^2\right)^2 \leq 2\max\left(2K_1k^2\omega^4,\frac{1}{2}(\zeta_1+\zeta_2)^2\right) . \]
Recalling \eqref{osclocrecadxsa} and \eqref{k defi sxaval}, as well as the definition of $v$, this yields
\[ \max_{S}v  \leq \sqrt{4K_1k^2\omega^4} + \frac{\vep}{t_0^{2\al}} \max_S  \util ^2 \leq  2\sqrt{K_1} \vep K^2 t_0^{2\al-2} + \vep K^2 t_0^{2\al-2} = (2\sqrt{K_1}+1)\vep K^2 t_0^{2\al-2}. \]
In particular, setting $t=t_0$, we obtain
\[\|u_x(\cdot,t_0)\|_{L^{\infty}((\gamma_L(t_0),\gamma_R(t_0))}^2 \leq 2(2 \sqrt{K_1}+1) \vep K t_0^{2\al-2} + 2(\zeta_1(t_0)+\zeta_2(t_0))\leq C t_0^{2(\al-1)}. \]
\end{proof}
\begin{rem} \label{rem: ux bd [0,1]} The previous proposition does not provide a uniform estimate on the gradient near $t=0$. Such bounds are available, but they typically depend on $T$ through the  oscillation of $u$ (see \cite[Prop. 3.10]{CMP}). In particular, the time localization arguments of the above proof readily yield the bound
\be  \|Du\|_{L^{\infty}(\R \times (0,1))} \leq C=C(C_0,|a_0|,|b_0|,\theta, \theta^{-1},\text{osc}_{L^{\infty}(\R \times (0,1))}(u)),\ee
but the quantity $\text{osc}_{L^{\infty}(\R \times (0,1))}(u)$ may, in principle, blow up as $T \to \infty$. See, however, Theorem \ref{thm: fin theta>2}.
\end{rem}
As before, we may now translate the results of this subsection into estimates for the continuous rescaling. We note that, when $\theta<2,$ \eqref{gamma x bd}, \eqref{mass cons}, and \eqref{tc} imply that $u(\cdot,T) \approx T^{2\al-1}$ is unbounded as $T \to \infty$. As one can readily verify, this implies that $u(\cdot,t)$ is also unbounded as $T\to \infty$ for any fixed $t$. Thus, in order to obtain $L^{\infty}$ bounds in this supercritical range of $\theta$, we renormalize $u$ by a vertical translation. Note also from \eqref{v linfty th leq 2} that, for the critical value $\theta=2$, the bounds on $v$ are still merely local (in fact, linear) in $\tau$, and this is consistent with \eqref{Uself}, due to the presence of the logarithmic term $z(t)$.
\begin{prop}[Lipschitz estimates for the continuous rescaling $v$] \label{prop: v Dv bd}Under the assumptions of Theorem \ref{thm.intro1}, let $(\tilde{u},m)$ be the solution to \eqref{mfg}--\eqref{tc}, and let
\be \label{u defi cases} u(x,t)= \begin{cases} \tilde{u}(x,t) & \theta >2 \\
\tilde{u}(x,t)- u(0,1), & \theta \leq 2.
    
\end{cases}
\ee
  Let $(v,\mu)$ be the continuous rescaling \eqref{v mu defi} corresponding to $(u,m)$, and let $M>0$. There exist positive constants $K_1=K_1(C_0,|a_0|,|b_0|,\theta, \theta^{-1})$  and $K_2=K_2(C_0,|a_0|,|b_0|,\theta, \theta^{-1},M)$ such that, for every $\tau \in [-M,\infty)$,
\be \label{v grad bd}\|v_{\eta}(\cdot,\tau)\|_{L^{\infty}(\R)}\leq K_1, \quad \|v_{\tau}(\cdot,\tau)\|_{L^{\infty}(-M,M)}\leq K_2,\ee
\be \label{v linfty bd th > 2}   \|v(\cdot,\tau)\|_{L^{\infty}(\R)}\leq K_1\left(1+\frac{\theta+2}{\theta-2}\right), \, \text{ for } \theta>2,\ee
and 
\be \label{v linfty th leq 2} 
\|v(\cdot,\tau)\|_{L^{\infty}(-M,M)} \leq \begin{cases} 
K_2 (M+\tau) & \theta =2, \\
K_2 \left(M+ \frac{2+\theta}{2-\theta} \left(1-e^{-\left(\frac{2-\theta}{2+\theta}\right)(\tau+M)}\right)\right) & \theta < 2.    
\end{cases}
\ee
\end{prop}
\begin{proof} Differentiating \eqref{v mu defi} with respect to $\eta$ and $\tau$, we have
\be \label{v eta comp} v_{\eta}(\eta,\tau)=t^{1-\al}u_x(x,t), \ee
\be \label{v tau comp} v_{\tau}(\eta,\tau)=(1-2\al)v(\eta,\tau) + \al v_{\eta}(\eta,\tau)\eta + t^{2-2\al}u_t(x,t)=\frac{\theta-2}{\theta+2}v(\eta,\tau) + \al v_{\eta}(\eta,\tau)\eta + t^{2-2\al}u_t(x,t).
\ee
The bound \eqref{v grad bd} on $v_{\eta}$ thus follows directly from \eqref{grad local bd}, taking 
\[\delta=e^{-M}.\] For the bound on $v_{\tau}$, at this stage we can only conclude for in the case that $\theta=2,$ in which case  \eqref{grad local bd} and \eqref{v tau comp} yield
\[ |v_{\tau}| \leq |v_{\eta}\eta| + t^{2-2\al}|u_t(x,t)| \leq \sqrt{C}M + C.  \]
Notice that, from \eqref{v tau comp}, in order to show \eqref{v grad bd} for the remaining values of $\theta$, it is in fact sufficient to prove \eqref{v linfty bd th > 2} and \eqref{v linfty th leq 2}, which is done next.

We first assume that $\theta>2$, so that $2\al-1 < 0$, and $u=\tilde{u}$. In view of \eqref{mfg}, \eqref{alpha theta}, \eqref{smoothing effect}, and \eqref{kap size}, $u$ satisfies, for some constant $K>0$,
\be \label{hj subsol acsamx} \begin{cases} -u_t + \frac12 u_x^2 \leq Kt^{2\al-2} & (x,t)\in \R \times (\delta,T), \\ 
u(x,T)\leq K T^{2\al-1} & x\in \R.    
\end{cases} 
\ee
The function \[ \zeta(t)=K T^{2\al-1}+\frac{K}{1-2\al}t^{2\al-1}\]
is a supersolution to \eqref{hj subsol acsamx}, so the comparison principle implies that
\be \label{u compar scoaca} u(x,t)\leq \zeta(t) \leq \left(1+\frac{1}{1-2\al}\right)Kt^{2\al-1}=\left(1+\frac{\theta+2}{\theta-2}\right)Kt^{2\al-1}.\ee
Recalling that $u=\tilde{u}\geq0$ due to \eqref{urepform} and \eqref{tc}, \eqref{u compar scoaca} yields \eqref{v linfty bd th > 2}.

Next, assume that $\theta \leq 2$. Then $2\al-1\geq 0$, and we see from \eqref{u defi cases} and \eqref{grad local bd} that, for $t\geq \delta$,
\be \label{u(0,t) kmasxoa}|u(0,t)| \leq \int_{\delta}^{t} Cs^{2\al-2}ds \leq \begin{cases} C \log(t/\delta) & \theta=2, \\ \frac{C}{2\al-1}(t^{2\al-1}-\delta^{2\al-1}) & \theta <2.
\end{cases}  \ee
On the other hand, the space derivative estimate of $\eqref{grad local bd}$ gives
\be \label{u space asoaxm} \|u(\cdot,t)-u(0,t)\|_{L^{\infty}(-Mt^{\al},Mt^{\al})} \leq 2M\sqrt{C}t^{2\al-1}.  
\ee
Thus, combining \eqref{u(0,t) kmasxoa} with \eqref{u space asoaxm}, and noting that $2\al-1=(2-\theta)/(2+\theta),$ we get
\be \label{expl u bd ascoac}\|u(\cdot,t)\|_{L^{\infty}(-Mt^{\al},Mt^{\al})}\leq \begin{cases} C\log(t)+MC+2M\sqrt{C} & \theta=2, \\
(2M \sqrt{C})t^{2\al-1}+C\frac{ 2+\theta}{2-\theta} (t^{2\al-1}-\delta^{2\al-1})
& \theta<2.    
\end{cases}
\ee
Finally, translating this inequality in terms of $v$ using \eqref{tau eta defi} and \eqref{v mu defi}, we obtain \eqref{v linfty th leq 2}.

\end{proof}

\subsection{Exact rates for energy estimates}
The next task will be to obtain the exact rates for certain integral bounds on the first derivatives of $m$. We use the displacement convexity inequality, Lemma \ref{lem:displ}.
\begin{prop}[Energy estimates with optimal rate]  \label{prop: energy rates} Under the assumptions of Theorem \ref{thm.intro1}, let $(u,m)$ be the solution to \eqref{mfg}--\eqref{tc}, and let $\vep \in (0,1)$. Then $ m^{\frac{\theta+\vep}{2}}\in H^1_{\emph{loc}}(\R \times (0,T))$. Furthermore, for every $\delta>0$, there exists a constant $C=C(C_0,|a_0|,|b_0|,\theta, \theta^{-1},\delta^{-1})$ such that, for each $t_0 \in (\delta,T/4)$,
\be \label{m energy bd} \int_{t_0/2}^{2t_0}\intr \left(\left((m^{\frac{\theta+\vep}{2}})_x\right)^2+\left(t_0^{1-\al}(m^{\frac{\theta+\vep}{2}})_t \right)^2\right) dxdt \leq \frac{C}{\vep (1-\vep)t_0^{1-\al(1-\vep)}}.\ee    
\end{prop}
\begin{proof}  In this proof, the value of the constant $C$ will be allowed to increase from line to line. We choose $\zeta \in C^{\infty}_c(0,T)$ such that $0\leq \zeta \leq 1$, $\zeta \equiv 1$ in $(t_0/2,2t_0)$, $\zeta \equiv 0$ outside of $(t_0/4,4t_0)$, and $|\zeta''| \leq C_2/t_0^2$, where $C_2$ is a universal constant. Then, taking $p=\vep$ in \eqref{displ subsol} yields

\be \label{energy e0dlq1}   \int_{t_0/2}^{2t_0}\intr (m^{\vep} u_{xx}^{2}+(m^{\frac{\theta+\vep}{2}})_{x}^{2}) dxdt \leq \frac{C}{\vep (1-\vep)t_0^2}\int_{t_0/4}^{4t_0} \intr m^{\vep} dxdt.  \ee
On the other hand, \eqref{smoothing effect}, \eqref{gamL}, and \eqref{grad local bd} yield
\[ \int_{t_0/4}^{4t_0} \intr m^{\vep} dxdt \leq Ct_0^{1 +\al - \al \vep }, \]
and, thus, \eqref{energy e0dlq1} becomes
\be \label{energy e0dlq2} \int_{t_0/2}^{2t_0}\intr (m^{\vep} u_{xx}^{2}+(m^{\frac{\theta+\vep}{2}})_{x}^{2}) dxdt \leq  \frac{C}{\vep (1-\vep)t_0^{1-\alpha(1-\vep)}}.\ee 
In order to prove \eqref{m energy bd}, it remains to estimate the time derivative of $m$. We proceed by bounding it in terms of the left hand side of \eqref{energy e0dlq2}. Crucially, we take advantage of Propositions \ref{prop: smoothing effect} and \ref{prop: grad rate} to ensure that this is done with the appropriate scaling. Furthermore, we recall that $(u,m)$ is smooth in the set $\{m>0\}$, which justifies the computations below.

By the second equation in \eqref{mfg}, we have, 
\[ m_t = mu_{xx}+m_x u_x. \]
Thus, we obtain from \eqref{smoothing effect}, \eqref{grad local bd}, and \eqref{alpha theta}, that, in $\R \times (t_0/2,2t_0)$,
\[ \left|(m^{\frac{\theta+\vep}{2}})_t\right| \leq \frac{\theta+\vep}{2} \left|m^{\frac{\theta+\vep}{2}} u_{xx} \right| + \left| (m^{\frac{\theta+\vep}{2}})_x u_x \right| \leq  Ct_0^{\al-1} (m^{\frac{\vep}{2}}|u_{xx}| +|(m^{\frac{\theta+\vep}{2}})_x|), \]
which means
\be \label{mt ener asodka}\left|(m^{\frac{\theta+\vep}{2}})_t\right|^2 \leq Ct_0^{2\al-2}\left(m^{\vep}u_{xx}^2+((m^{\frac{\theta+\vep}{2}})_x)^2 \right). \ee
The estimate \eqref{m energy bd} now follows from \eqref{energy e0dlq2} and \eqref{mt ener asodka}. As for the local integrability, while this qualitative property was so far only shown on $(0,T/2),$ it of course holds on any compact sub-interval of $(0,T)$, by varying the choice of test function $\zeta$ used in \eqref{displ subsol}.

\end{proof}
As usual, these scale-appropriate bounds allow us to obtain energy estimates for the continuous rescaling.
\begin{prop}[Energy estimates for the continuous rescaling] \label{prop: v mu energy}Under the assumptions of Theorem \ref{thm.intro1}, let $(u,m)$ be the solution to \eqref{mfg}--\eqref{tc}, let $(v,\mu)$ be the continuous rescaling given by \eqref{v mu defi}, let $\vep \in (0,1)$, and let $M>0$. Then $\mu^{\frac{\theta+\vep}{2}} \in H^1_{\emph{loc}}(\R \times (-\infty,\log(T)))$. Furthermore, there exists a constant $C=C(C_0,|a_0|,|b_0|,\theta, \theta^{-1},M)$ such that, for every $\tau_0 \in (-M,\log(T)-\log(4))$,
 \be  \label{v mu energy} \int_{\tau_0 -\log(2)}^{\tau_0+\log(2)} \int_{\R} |D(\mu^{\frac{\theta+\vep}{2}})|^2d\eta d\tau  \leq \frac{C}{\vep (1-\vep)}.\ee
 \end{prop}
 \begin{proof}
In view of \eqref{v mu defi}, we have
\be \mu_{\et}(\et,\ta)=t^{2\al}m_x(x,t)\ee
Hence, changing variables according to \eqref{tau eta defi}--\eqref{v mu defi}, recalling \eqref{alpha theta}, and using Proposition \ref{prop: energy rates}, one obtains
\begin{multline} \label{mu ener asocd} \int_{\tau_0-\log(2)}^{\tau_0-\log(2)} \intr ((\mu^{\frac{\theta+\vep}{2}})_{\et})^2 d\et d\ta = \int_{t_0/2}^{2t_0}\intr t^{1-\al(1-\vep)} ((m^{\frac{\theta+\vep}{2}})_x)^2dxdt\\
\leq C t_0^{1-\al(1-\vep)}\int_{t_0/2}^{2t_0}\intr  ((m^{\frac{\theta+\vep}{2}})_x)^2dxdt\leq \frac{C}{\vep(1-\vep)}.  \end{multline}
On the other hand, differentiating the equation $\mu(\eta,\ta)=t^{\al}m(x,t)$ with respect to $\ta$, we have
\be \mu_{\ta}=t^{\al+1}m_t(x,t)+\al \mu_{\et} \et + \al \mu. \ee
Hence,
\be  |(\mu^{\frac{\theta+\vep}{2}})_{\ta}|^2 \leq C(\mu^{\theta+\vep-2}t^{2\al+2}m_t(x,t)^2 +  \et^2 (\mu^{\frac{\theta+\vep}{2}})_{\et}^2+ \mu^{\theta+\vep}). \ee
The integral of the first term can be estimated by changing variables as in \eqref{mu ener asocd}, whereas the last two terms are treated directly by appealing to \eqref{mu ener asocd}, \eqref{mu suppt bd}, and \eqref{mu linfty}. 
\end{proof}
\begin{rem} \label{rem energy} The bound for $\mu$ in \eqref{v mu energy} is strictly stronger than the one required in \eqref{mu cont assumption}, since $\mu$ is bounded.    
\end{rem}
\subsection{Exact rates for the modulus of continuity of the density}
In this subsection, we explain how the intrinsic scaling arguments of \cite[Sec. 4.3.1]{CMP} can be adapted to our purposes. The goal here is to obtain a H\"older continuity estimate for $m$ which decays sufficiently fast with time, so as to guarantee uniform continuity of the continuous rescaling $\mu$.

We begin by discussing the underlying ideas behind these H\"older estimates. The starting point is to switch to Lagrangian coordinates and consider the function 
\be \label{p defi} p(x,t)=m(\gamma(x,t),t)^{\theta},\quad (x,t)\in [a_0,b_0]\times [0,T].\ee It was shown in \cite[Lem. 3.5]{CMP} that $p$ satisfies the divergence form, degenerate elliptic equation,
\begin{equation} \label{eq:DG p eq} -\left((\gamma_x)^{-1}p_x\right)_x  - \left(\gamma_x(\theta p)^{-1}p_t\right)_t= 0,\end{equation}
where $\gamma_x>0$ is the space derivative of the Lagrangian flow \eqref{flow defi}.
The main point now, which was exploited in \cite{CMP}, is that, thanks to the fact that the $\gamma_x$ is bounded above and below (for bounded times), we can quantify in a precise way how the equation degenerates near the free boundary: as $x$ approaches $\{a_0,b_0\}$, the time diffusion blows up like $p(x,t)^{-1}$ in the time direction, while the space diffusion remains bounded above and below. This quantification is, of course, intrinsic, since it depends on the value of the solution itself. However, if one considers appropriate intrinsic subdomains, H\"older continuity estimates may be obtained with the classical arguments of E. De Giorgi (\cite{degiorgi}), which are valid for uniformly elliptic equations in divergence form having bounded and measurable coefficients. The method of intrinsic scaling is due to E. DiBenedetto (\cite{DiB,DiB2}, see also \cite{Urbano}).

The key additional observation that we will utilize here is that the long time behavior of the degeneracy can also be quantified. Indeed, in view of \eqref{gamma x bd}, the diffusion in \eqref{eq:DG p eq} is of the order of $t^{-\alpha}$ in the space direction, and $t^{\alpha} p(x,t)^{-1}$ in the time direction. This time-dependence must be carefully tracked for the resulting continuity estimates to remain valid for large times, for the height--normalized density $\mu=t^{\alpha}m$, in the zoomed-out variables $(\eta,\tau)=(xt^{\al},\log(t))$.

The first result is the following intrinsic Harnack inequality (cf. \cite[Prop. 4.16]{CMP}).
\begin{prop}[Harnack inequality]
\label{prop:Harnack} 
Under the assumptions of Theorem \ref{thm.intro1}, let $(u,m)$ be the solution to \eqref{mfg}--\eqref{tc}, and let $p$ be defined by \eqref{p defi}. There exists a constant $C(C_0,|a_0|,|b_0|)$ such that the following holds. Let $(x_{0},t_0)\in(a_{0},b_{0})\times (0,T]$,
and let $\rho_1,\rho_2>0$ be such that $\rho_1 < \frac{1}{2}\emph{dist}(x_0,\{a_0,b_0\})$ and $\rho_2<\frac{1}{2}t_0$.
Then
\begin{equation} \label{eq:DG Harnack}
\sup_{(x_{0}-\rho,x_{0}+\rho)\times(t_0-\rho_2,t_0]}p\leq C\inf_{(x_{0}-\rho_1,x_{0}+\rho_2)\times(t_0-\rho_2,t_0]}p.
\end{equation}
\end{prop}
\begin{proof}
For simplicity, we normalize $a_{0}=0$, and by symmetry we may assume
that $x_{0}<\frac{b_{0}}{2}.$ As usual, the constant $C$ may increase at each step. Recalling \eqref{gamma x bd}, we have, for every $(x,t)\in  (x_0-\rho_1,x_0+\rho_1) \times  (t_0-\rho_2,t_0)$, 
\[ \label{gamx bd axapo} \frac{1}{C}(1+t^{\al}) \leq \gamma_x(x,t) \leq C (1+t^{\al}).\]
Recall also that, by \eqref{mass cons}, \be p(x,t)=m_0(x)^{\theta}/\gamma_x^{\theta}.\ee Thus, \eqref{mass cons} and \eqref{m0 bump} imply that
\[ \frac{1}{C} \frac{x}{(1+t^{\al})^{\theta }} \leq p(x,t) \leq C \frac{x}{(1+t^{\al})^{\theta } }. \]
This yields
\be \label{harnack upper asxas} p(x,t) \leq C \frac{x}{(1+t^{\al})^{ \theta}}\leq C \frac{(3x_0/2)}{(1+(t_0/2)^{\al})^{ \theta}}\leq C \frac{(3x_0/2)}{(1+(t_0/2)^{\al})^{ \theta}} \leq C 3\cdot 2^{\alpha \theta -1} \frac{x_0}{(1+t_0^{\al})^{ \theta}} \leq C^2 3 \cdot 2^{\al \theta -1} p(x_0,t_0). \ee
Similarly,
\be \label{harnack lower asxas} p(x,t) \geq \frac{1}{C} p(x_0,t_0).\ee
Since $(x,t)$ is arbitrary, we conclude from \eqref{harnack upper asxas} and \eqref{harnack lower asxas} that
\[ \sup_{(x_{0}-\rho,x_{0}+\rho)\times(t_0-\rho_2,t_0]} p \leq 3\cdot 2^{\al \theta}  C_2^4\inf_{(x_{0}-\rho,x_{0}+\rho)\times(t_0-\rho_2,t_0]} p. \]
\end{proof}

\begin{defn} \label{defi intrinsic} Given $(x_{0},t_{0})\in(a_{0},b_{0})\times(0,T)$ and $\rho>0$, we
define the intrinsic rectangle $R_{\rho}(x_{0},t_{0})$ of radius
$\rho$ centered at $(x_{0},t_{0})$ by
\[
R_{\rho}(x_{0},t_{0}):=(x_{0}- t_0^{-\al/2}\rho,x_{0}+ t_0^{-\al/2}\rho )\times(t_{0}-t_0^{\al/2}(\theta p(x_{0},t_{0}))^{-\frac{1}{2}}\rho,t_{0}+t_0^{\al/2}(\theta p(x_{0},t_{0}))^{-\frac{1}{2}}\rho).
\]    
\end{defn}
Following the proof of \cite[Cor. 4.20]{CMP}, one then obtains the following proposition through classical arguments of De Giorgi. 
\begin{prop}[Reduction of oscillation] 
\label{prop:reduction of osc}Under the assumptions of Theorem \ref{thm.intro1}, let $(u,m)$ be the solution to \eqref{mfg}--\eqref{tc}, and let $p$ be defined by \eqref{p defi}. Assume that $(x_0,t_0)\in (a_0,b_0)\times (\delta,T/2)$, and assume that the intrinsic rectangle $R_{\rho}(x_0,t_0)$ is such that
\be \label{eq:distance assumption} R_{8\rho}(x_0,t_0) \subset (a_0,b_0)\times (0,T).\ee
Then there exists a constant $0<\sigma<1$, depending only on $C_0,|a_0|,|b_0|, \delta^{-1}$, such that
\begin{equation} \label{eq:osc decr}
\text{\emph{osc}}_{R_{\rho}(x_{0},t_{0})}(p)\leq\sigma\, \text{\emph{osc}}_{R_{4\rho}(x_{0},t_{0})} (p).
\end{equation}
\end{prop}
Up to this point, the only substantial difference in our treatment of these continuity estimates with respect to \cite[Sec. 4.3.1]{CMP} lies in our definition of the intrinsic rectangles, which takes into account the long time scaling. In the last step, however, the difference is somewhat delicate, so we provide the proof in full detail.
\begin{prop}[Exact rates for the modulus of continuity] \label{prop:holder m}Let $\delta>0$. There exists $s\in (0,1)$, depending only on $C_0,|a_0|,|b_0|,\delta^{-1}$, such that the following holds. For every $\overline{t} \in [2\delta,T/4]$, and every $(x_0,t_0)$, $(x_1,t_1) \in (a_0,b_0) \times [\overline{t}/2,2\overline{t}],$
 \be \label{p conti lag} |p(x_1,t_1)-p(x_0,t_0)|\leq C t_0^{-\al \theta} (|x_1-x_0|^{s}+ |(t_1-t_0)t_0^{-1}|^s),
 \ee
 where $C=C(C_0,|a_0|,|b_0|,\delta^{-1}).$ Furthermore, for every $(x_0,t_0),(x_1,t_1) \in S(m,\overline{t}/2,2\overline{t}),$
\be \label{p conti} |m^{\theta}(x_1,t_1)-m^{\theta}(x_0,t_0)| \leq Ct_0^{-\alpha \theta} (|(x_1-x_0)t_0^{-\al}|^{s}+|(t_1-t_0)t_0^{-1}|^{s}).
\ee    
\end{prop}
\begin{proof}
 Let $(x_{0},t_{0}),(x_1,t_1)\in(a_{0},b_{0})\times[\overline{t}/2,2\overline{t}]$, where $x_1<x_0$. As usual, throughout this proof, the value of the constant $C>0$ may increase from line to line, and with no loss of generality we take $\delta=1$. We will consider intrinsic rectangles centered at $(x_0,t_0)$, so we abbreviate $R_{\rho}:=R_{\rho}(x_0,t_0)$. As in
the proof of Proposition \ref{prop:Harnack}, we may assume that
$a_{0}=0$ and $x_{0}<\frac{1}{2}b_0$, and write, for each $(x,t)\in R_{4 \rho},$
\begin{equation} \label{eq:DG A}
\frac{1}{C}\frac{x}{t^{\al \theta}}\leq p(x,t)\leq C \frac{x}{t^{\al \theta}}.
\end{equation}
We let $\vep \in (0,1)$ be a small constant to be chosen below, and set
\be \label{rho0 defi}
\rho_{0}=\vep t_0^{\al/2}x_0.
\ee 
In order to apply Proposition \ref{prop:reduction of osc}, we will determine a choice of $\vep$ so that \eqref{eq:distance assumption} holds for $R_{8\rho_0}$. Assume then that $(x,t)\in R_{8\rho_0}$. Recalling Definition \ref{defi intrinsic}, if we require $\vep \leq 1/8$, then
\be \label{DG contain x cascd}|x-x_0|<t_0^{-\al/2}8\rho_0=8\vep x_0< x_0,\ee
 On the other hand, recalling that $0<\al<1$ and $t_0 \in [1,T/2]$, \eqref{eq:DG A} implies, for some $C_2>0$
\be \label{DG contain t cascd}|t-t_0| < t_0^{\al/2} (\theta p(x_0,t_0))^{-1/2}8\rho_0 \leq C_2 t_0^{3\al/2-1}x_0^{-1/2}\rho_0= C_2t_0^{2(\al-1)}\vep x_0^{1/2} t_0 \leq C_2|b_0-a_0|^{\frac12} \vep  t_0:=K \vep t_0. \ee
Thus, choosing $\vep = \min(1/8,1/K)$, we obtain \eqref{eq:distance assumption} from \eqref{DG contain x cascd} and \eqref{DG contain t cascd}. Furthermore, we see from the same computations that
\be \label{DG 4rho cont} R_{4\rho} (x_0,t_0) \subset (x_0/2,3x_0/2) \times (t_0/2,3t_0/2).\ee
Setting now
\be \label{DG a defi asc}
a=\max(t_0^{\al/2}|x_1-x_{0}|,t_0^{-\al/2}\sqrt{\theta p(x_{0},t_{0})}|t_1-t_{0}|),
\ee 
we distinguish two alternative cases.

\textbf{Case 1.} $(x_1,t_1)\in R_{4\rho_{0}}$. Equivalently,
we have $a<4\rho_{0}$. Let $n\geq0$ be the unique integer such that
\[
\frac{1}{4^{n}}\rho_{0}\leq a<\frac{1}{4^{n-1}}\rho_{0}.
\]
Iterating Proposition \ref{prop:reduction of osc}, we see that, in view of \eqref{eq:DG A}, \eqref{rho0 defi}, and \eqref{DG 4rho cont}, 
\begin{multline*}
\text{osc}_{R_{4^{-(n-1)}\rho_{0}}}(p)\leq\sigma^{n}\text{osc}_{R_{4\rho_{0}}}(p)\leq \sigma^{n}Ct_0^{-\al \theta} (x_{0}+4\rho_0 t_0^{-\al/2})\\=\sigma^{n}Ct_0^{-\al \theta} (\rho_0 \vep^{-1} t_0^{-\al/2}+4\rho_0 t_0^{-\al/2})\leq Ct_0^{-\al \theta}.
\end{multline*}
Moreover, by increasing the value of $\sigma$ if necessary, we may assume that $\sigma>\frac{1}{4},$  so that $s=-\log\sigma(\log4)^{-1}$ satisfies $0<s<1$. Thus, observing that $(x_1,t_1)\in R_{4^{-(n-1)}\rho_0}$ and $n\geq-\log\left(\rho_{0}^{-1}a\right)(\log4)^{-1}$,
 we have
\be \label{dg penu asxsd}
|p(x_1,t_1)-p(x_{0},t_{0})|\leq  C t_0^{-\al \theta} t_0^{-\al/2}(\rho_{0}^{-1}a)^{s}\rho_{0}=C t_0^{-\al \theta}  t_0^{-\al/2}\rho_{0}^{1-s}a^s.
\ee
We estimate $a$ as follows. Note that, since $\al \theta = 2-2\al,$ \eqref{eq:DG A} yields
\[t_0^{-\al/2}\sqrt{\theta p(x_0,t_0)}|t_1-t_0|\leq (k_1 \theta x_0)^{1/2}t_0^{-\al/2-\al \theta/2}|t_1-t_0|\leq Ct_0^{\al/2}|(t_1-t_0)t_0^{-1}|.\]
Therefore, \eqref{DG a defi asc} gives
\be  a\leq Ct_0^{\al/2}\max(|x_1-x_{0}|,|(t_1-t_0)t_0^{-1}|),  
\ee
so that
\be \label{a bd zmcak} a^s \leq C t_0^{\al s/2} (|x_1-x_0|^s+|(t_1-t_0)t_0^{-1}|^s).\ee
Thus, we obtain from \eqref{rho0 defi}, \eqref{dg penu asxsd} and \eqref{a bd zmcak} that \eqref{p conti lag} holds.

\textbf{Case 2.} $(x_1,t_1)\notin R_{4\rho_{0}}.$ Then $a\geq4\rho_{0}$. Since $x_1<x_{0}$, appealing to the lower bound on $\gamma_x=\left(f(m_0(x))p(x,t)^{-1}\right)^{\frac{1}{\theta}}$ given by \eqref{gamma x bd}, recalling \eqref{m0 bump}, 
\begin{multline} \label{mucont dqwxojd}
|p(x_1,t_1)-p(x_{0},t_{0})|  \leq t_0^{-\al \theta}C(f(m_{0})(x_1)+f(m_{0})(x_{0}))\leq  Ct_0^{-\al \theta} x_0 = C t_0^{-\al \theta} \vep^{-1} t_0^{-\al/2}\rho_{0}\\
  \leq  C t_0^{-\al \theta}t_0^{-\al/2}(a/4) \leq  Ct_0^{-\al \theta}(|x_{1}-x_{0}|+|(t_{1}-t_{0})t_0^{-1}|).
\end{multline}
We finally point out that, since $0<x_1<x_0<b_0$ and $t_0,t_1\in (\overline{t}/2,2\overline{t})$,
\be \label{DG plt smpcs} |x_1-x_0|\leq C|x_1-x_0|^s , \,\,\,\, \text{ and } \,\,\,\,| (t_1-t_0)t_0^{-1}|\leq C| (t_1-t_0)t_0^{-1}|^s. \ee Therefore, \eqref{p conti lag} follows from \eqref{mucont dqwxojd} and \eqref{DG plt smpcs}.

We now show \eqref{p conti}. Let $(x_0,t_0),(x_1,t_1) \in  S(m,\overline{t}/2,2\overline{t})$, with $t_0 \leq t_1$. Recalling that the support of $m$ is expanding with time, we may write
$x_0=\gamma(a_0,t_1),\, x_1=\gamma(a_1,t_1)$. Then  \eqref{p conti lag} and \eqref{gamma x bd} yields
\begin{multline}    
\label{mucont111casd} |m(x_1,t_1)-m(x_0,t_1)|\leq C t_0^{-\al \theta} |a_1-a_0|^{\beta} \\ 
\leq C t_0^{-\al \theta} |\|(\gamma_x)^{-1}\|_{L^{\infty}((a_0,b_0)\times (\overline{t}/2,2\overline{t}))}(x_1-x_0)|^{\beta} \leq C t_0^{-\al \theta}  |t_0^{-\al}(x_1-x_0)|^{\beta}. \end{multline} 
Similarly, writing $x_0=\gamma(b_0,t_0)$, we get
\be \label{mucont222casd} |m(x_0,t_1)-m(x_0,t_0)| \leq Ct_0^{-\al \theta} (|b_0-a_0|^s+|t_0^{-1}(t_1-t_0)|^s).\ee
We conclude from \eqref{mucont111casd} and \eqref{mucont222casd} by observing that, in view of \eqref{flow defi}, \eqref{gamma x bd}, and \eqref{grad local bd},
\[|b_0-a_0| \leq  \|\gamma_t\|_{L^{\infty}((a_0,b_0)\times (\overline{t}/2,2\overline{t}))} \|(\gamma_x)^{-1}\|_{L^{\infty}((a_0,b_0)\times (\overline{t}/2,2\overline{t}))}|t_1-t_0| \leq C t_0^{\al-1} t_0^{-\al}|t_1-t_0| =C|(t_1-t_0)t_0^{-1}|.
\]
\end{proof}   
We are now in a position to prove the last remaining estimate for the continuous rescaling required by Proposition \ref{prop: convergence result}, namely the uniform continuity of $\mu$.
\begin{prop} \label{prop: mu holder cont}Under the assumptions of Theorem \ref{thm.intro1}, let $(u,m)$ be the solution to \eqref{mfg}--\eqref{tc}, let $\mu$ be the continuous rescaling \eqref{v mu defi}, and let $M>0$. Let $s\in(0,1)$ be the constant from Proposition \ref{prop:holder m}, for $\delta=e^{-M}/2$. There exists a constant $C=C(C_0,|a_0|,|b_0|,M)$ such that, for every  $(\eta_1,\tau_1),(\eta_0,\tau_0)\in \R \times [-M,\log(T)-\log(4)]$,
\be |\mu^{\theta}(\et_1,\ta_1)-\mu^{\theta}(\et_0,\ta_0)| \leq C \left( |\et_1-\et_0|^{\beta}+\left|\ta_1-\ta_0\right|^{\beta}\right). \ee
\end{prop}
\begin{proof}
%With no loss of generality, we assume that $\tau_0 \leq \tau_1$. 
Since $\mu$ is bounded, and its support remains bounded as $\tau \to \infty$ (see Proposition \ref{prop: mu bd}), we may  assume that 
\be \label{dg ettaubd foeqkc}  |\eta_1| + |\eta_0| \leq C, \quad|\tau_1 - \tau_0| \leq \log(2). \ee Writing, for $i\in \{0,1\},$ $t_i=e^{\tau_i}$ and $x_i=\eta_i t_i^{\al}$, we then have \be t_0\in [2,T/4],  \quad t_0/2 \leq t_1 \leq 2t_0. \ee
Thus, allowing the constant $C>0$ to increase at each step,  \eqref{p conti} and \eqref{dg ettaubd foeqkc} yield
\begin{multline} \label{eq: mucont ask1} |t_0^{\al \theta} m^{\theta}(x_1,t_1)-t_0^{\al \theta} m^{\theta}(x_0,t_0)|\leq C(|(x_1-x_0)t_0^{-\al}|^{s}+|(t_1-t_0)t_0^{-1}|^{s})\\
=C(|\eta_1 e^{\tau_1-\tau_0}-\eta_0|^s+|e^{\tau_1-\tau_0}-1|^s)\leq C(|\eta_1-\eta_0|^s+|\tau_1-\tau_0|^s).  
\end{multline}
On the other hand, by \eqref{smoothing effect} and \eqref{dg ettaubd foeqkc},
\be \label{eq: mucont ask2}  |(t_1^{\al \theta}-t_0^{\al \theta})m^{\theta}(x_1,t_1)| \leq C |(t_1^{\al \theta}-t_0^{\al \theta})|t_1^{-\al \theta} = C|e^{\al(\tau_1-\tau_0)}-1| \leq C|\tau_1 - \tau_0|\leq C |\tau_1 - \tau_0|^{s}. \ee
Hence, we conclude from \eqref{eq: mucont ask1} and \eqref{eq: mucont ask2} that
\be |\mu^{\theta}(\eta_1,\tau_1)-\mu^{\theta}(\eta_0,\tau_0)|= |t_1^{\al \theta}m^{\theta}(x_1,t_1)-t_0^{\al \theta}m^{\theta}(x_0,t_0)|\leq C(|\eta_1-\eta_0|^s+|\tau_1-\tau_0|^s).\ee
\end{proof}
\section{Convergence results for the finite horizon problem} \label{sec: main}
We prove in this section, for solutions to \eqref{mfg}--\eqref{tc}, the main results of the paper. In particular, the main goal of this section is the following claim, which corresponds to Theorem \ref{thm: fin hor}.
\begin{thm} \label{thm: main} Assume that $m_0$ satisfies \eqref{m0 assum} and \eqref{m0 C1,1}, and that \eqref{kap size} holds. Let $(u^T,m^T)$ be the solution to \eqref{mfg}--\eqref{tc}, and let $\int_{\R}m_0=a>0$. Then, for every $p\in [1,\infty],$
\be \label{Sec5 conv result 1} \lim_{t \to \infty} \limsup_{T\to \infty} t^{\al(1-\frac1p)}\| m^T(\cdot,t)- \mathcal{M}_a(\cdot,t) \|_{L^{p}(\R)} =0,
\ee
\be \label{Sec5 conv result 2} \lim_{t \to \infty} \limsup_{T\to \infty} t^{2-\al(1+\frac1p)}\left\| m^T(\cdot,t)\left|u^T_x(\cdot,t)- \mathcal{U}_x(\cdot,t)\right|^2 \right\|_{L^{p}(\R)} =0,
\ee
\be \label{Sec5 conv result 3} \lim_{t \to \infty}  t^{2-\al(1+\frac1p)}\left\| m(\cdot,t)\left|u_t(\cdot,t)- (\mathcal{U}_a)_t(\cdot,t)\right|\right\|_{L^{p}(\R)} =0. \ee
\end{thm}
This theorem will be proved separately for each of the three relevant ranges of the parameter $\theta$. Namely, the cases $\theta \in (0,2)$, $\theta\in (2,\infty)$, and $\theta=2$ are, respectively, proved in Proposition \ref{prop: main th<2}, Corollary \ref{cor:main th>2}, and Corollary \ref{cor:main th=2}. As will be explained below, the three proofs are fundamentally different, but they all rely on the following result.
\begin{prop} [Existence of solutions as subsequential limits] \label{prop: existence subseq} Let $T_n\in [1,\infty)$ be a sequence of terminal times such that $T_n \to \infty$. Under the assumptions of Theorem \ref{thm.intro1},  let $(\util^{T_n},m^{T_n})$ be the solution to \eqref{mfg}--\eqref{tc}, and let $u^{T_n}$ be defined according to \eqref{u defi cases}. Up to extracting a subsequence, there exists a solution $(u,m)$ to \eqref{mfginfhor} on $\R \times (0,\infty)$ such that, for every $K>0$, $(u^{T_n},m^{T_n})$ converges to $(u,m)$ in $C^1([-K,K] \times [1/K,K]) \times C(\R \times [1/K,K])$. Any such limit $(u,m)$ satisfies, for any $t_0>0$, the assumptions of Proposition \ref{prop: convergence result}, and \eqref{IHH conv result 1}--\eqref{IHH conv result 3} hold. 
\end{prop}
\begin{proof} We write $(u_n,m_n):=(u^{T_n},m^{T_n})$, and let $(v_n,\mu_n)$ be the corresponding continuous rescaling \eqref{v mu defi} of $(u_n,m_n)$. As a result of Propositions \ref{prop: mu bd}, \ref{prop: v Dv bd}, \ref{prop: v mu energy}, and \ref{prop: mu holder cont}, and recalling Remark \ref{rem energy}, we have, for any constant $M>0$,
\be \label{mun bd} \| \mu_n \|_{L^{\infty}\left(\R \times (-\infty, \log(T))\right)}\leq K_1, \ee
\be \label{mun spt bd}\text{supp}(\mu_n(\cdot,\ta)) \subset [-R_1,R_1], \quad  \ta \in [-M,\infty),\ee
\be \label{Dvn + mun cont bd}\|Dv_n\|_{L^{\infty}([-R_1,R_1]\times [0,\infty))}+\|\mu_n^{\theta} \|_{C^{\beta}\left(\R \times \left[-M, \log(T)-\log(4)\right]\right)} \leq K_2, \ee
\be \label{mu_n energy bd} \|\mu_n^{\theta}\|_{H^1(\R \times (\tau_1-\log(2),\tau_1+\log(2))} \leq K_2, \quad \tau_1 \in (-M,\infty), \ee
where $K_1=K_1(C_0,|a_0|,|b_0|)$, $K_2=K_2(C_0,|a_0|,|b_0|,M)$,  $R_1=R_1(C_0,|a_0|,|b_0|,R_a)$, $R_1 \geq R_a$. Furthermore, by \eqref{v linfty bd th > 2} and \eqref{v linfty th leq 2}, $v_n$ is bounded on compact subsets of $\R \times \R.$ It follows from the Arzel\`a-Ascoli theorem that, up to extracting a subsequence, $(v_n,\mu_n)$ converges locally uniformly to certain continuous functions $v:\R \times \R \to \R $, $\mu: \R \times \R \to [0,\infty].$ Defining the function $(u,m)$ in terms of $(v,\mu)$ according to \eqref{v mu defi IH}, the lodqll6 uniform convergence of $(u_n,m_n)$ to $(u,m)$ follows, and $(u,m)\in W^{1,\infty}_{\text{loc}}(\R \times (0,\infty)) \times C(\R \times (0,\infty))$. 

We now explain why $u\in C^{1}(\R \times (0,\infty))$, and  $u_n  \to u$ in $C^1_{\text{loc}}$. Note that, by \eqref{p conti}, the function $m^{\theta}_n$ has a uniform modulus of H\"older continuity on every set $\R \times [t_0,t_1]$, where $0<t_0<t_1<T_n/4$. Recall also that $u_n$ is a $C^1$ solution of the first equation in \eqref{mfg}, and is uniformly Lipschitz on $\R \times [t_0,t_1]$. By the interior regularity theory for $C^1$ solutions of HJ equations, it follows that $Du_n$ has a uniform modulus of H\"older continuity on $\R \times [t_0,t_1]$ (see \cite[Thm 4.23]{CMP}. The original argument dates back to \cite{CaSo}). The $C^1_{\text{loc}}$ convergence of $u_n$ hence follows once more by the Arzel\`a-Ascoli theorem. 

The fact that $v \in C^2(\{\mu>0\})$ follows in the same way from Remark \ref{rem: Dv regu finite hor}. Furthermore, by the uniform estimates \eqref{mun bd}, \eqref{mun spt bd}, \eqref{Dvn + mun cont bd}, and \eqref{mu_n energy bd}, we see that $(v,\mu)$ satisfies \eqref{mu cont assumption} and \eqref{Dv bd assumption}. Passing to the limit in \eqref{mfg}, it directly follows that $(u,m)$ is a solution of the first equation in the classical sense, and to the second equation in the distributional sense. Finally, let $\gamma_L^n, \gamma_R^n$ be the free boundary curves of $(u_n,m_n)$. By Proposition \ref{prop: free bd rates}, the functions $(\gamma_L^n,\gamma_R^n)$ are uniformly bounded on $C^{1,1}((0,M))$ for every $M>0$, for $n$ sufficiently large. Thus, they converge uniformly to a pair of $C^{1,1}_{\text{loc}}(\left[0,\infty\right))$ curves $\gamma_L, \gamma_R$, which are uniquely characterized by passing to the limit as $n\to \infty$ in \eqref{freebd ch}.

In summary, $(v,\mu)$ satisfies, for any $t_0>0$, all the assumptions of Proposition \ref{prop: convergence result}, and, thus, its conclusion, \eqref{IHH conv result 1}--\eqref{IHH conv result 3}.
\end{proof}

\subsection{The supercritical range $\theta \in (0,2)$}
The main difficulty in showing the convergence result for the range $\theta<2$ will be the lack of a uniqueness result for \eqref{mfginfhor}, even among solutions that satisfy the natural growth estimates from Section \ref{sec: estimates}. For this reason, we are unable to characterize the subsequential limits obtained earlier in Proposition \ref{prop: existence subseq}. However, we resolve this issue by exploiting an advantageous feature of the range $\theta \in (0,2)$, namely the exponential rates of convergence given by Proposition \ref{prop: exponential theta<2}. Indeed, we prove below that, even if the subsequential limit $(u,m)$ for the family $(u^T,m^T)$ were not unique, the convergence result of Proposition \ref{prop: convergence result} would still hold uniformly across all such limit solutions $(u,m)$. 
\begin{prop} \label{prop: main th<2} The statement of Theorem \ref{thm: main} holds if $\theta<2$.    
\end{prop}
\begin{proof}Since the conclusion is invariant under vertical translations of $u^T,$ we may redefine $u^T$ according to \eqref{u defi cases}. By the same argument as in the proof of Lemma \ref{lem:reduc to cont resc}, the result is equivalent to showing that, for every $p\in [1,\infty]$,
\be \label{capomcqdc} \lim_{\ta \to \infty}\limsup_{T\to \infty} \| \mu^T(\cdot,\ta)- M_a(\cdot) \|_{L^{p}(\R)} =0, \,\,\,\,\,\text{ and } \,\,\,\,\,\lim_{\ta \to \infty}\limsup_{T\to \infty} \| \mu^T |w_{\eta}^T|^2(\cdot,\ta)\|_{L^{p}(\R)} = 0. \ee 
Furthermore, in view of \eqref{mu suppt bd} and \eqref{v grad bd}, it is enough to prove \eqref{capomcqdc} for $p=\infty$.

Let $\vep>0$ be arbitrary. By Proposition \ref{prop: mu holder cont}, if $T \geq 4t$, there exists a uniform modulus of continuity $\omega_{1} : \left[0,\infty\right) \to \left[0,\infty\right)$ for the functions $\mu^T(\cdot,\ta)$, for $\ta \in \left[0,\infty\right)$. Since the function $M_a$ is uniformly continuous in $\R$, we may also assume that $\omega_1$ is a modulus for $M_a$. In addition, from Lemma \ref{lem:regu Dw in mu>0} and Remark \ref{rem: Dv regu finite hor}, we infer that there exists a uniform modulus of continuity $\omega_{2,\vep}:\left[0,\infty\right) \to \left[0,\infty\right)$ for the functions $w_{\eta}(\cdot,\ta)$ on the set $\{\mu(\cdot,\ta)\geq \vep /2 \}$. We choose $\delta>0$ such that \be \omega_{1}(\delta)+\omega_{2,\vep}(\delta)<\vep/2. \ee 
If $\mu(\et_1,\ta)<\vep,$ then
\be \label{th<2 sxpaok}\mu^T |w^T_{\et}|^2(\et_1,\ta) \leq \vep\|w^T_{\et}\|^2_{L^{\infty}(\{\mu>0\})} =O(\vep).  \ee
Otherwise, $\mu(\et_1,\ta)\geq \vep$, and we have, for any $T\geq 4t,$ by \eqref{mu suppt bd}, and \eqref{v grad bd},
\begin{align} \mu^T |w^T_{\et}|^2(\et_1,\ta)\leq &  2\|\mu^T\|_{\infty}\|w^T_{\et}\|_{L^{\infty}(\{\mu^T>0\})} \text{osc}_{|\eta-\eta_1|<\delta}(w_{\et}(\cdot,\ta))+(2\delta)^{-1}\int_{|\eta-\eta_1|<\delta} \mu^T |w^T_{\et}|^2(\cdot,\ta)\\ \label{th<2 sxpaok2}
\leq & O(1) \omega_{2,\vep}(\delta)+ (2\delta)^{-1}\intr \mu^T |w^T_{\et}|^2(\cdot,\ta)\leq O(\vep)+(2\delta)^{-1}\intr \mu^T |w^T_{\et}|^2(\cdot,\ta).\end{align}
On the other hand, by Proposition \ref{prop: existence subseq} and the exponential bounds \eqref{exp conv w} for the infinite horizon system, we have
\be \label{th<2 sxpaok3}\limsup_{T \to \infty}\intr \mu^T |w^T_{\et}|^2(\cdot,\ta)=O(1)e^{-k \ta} .\ee
We infer from \eqref{th<2 sxpaok}, \eqref{th<2 sxpaok2}, and \eqref{th<2 sxpaok3} that
\be \label{th<2 sxpaok4} \limsup_{T\to \infty}  \|\mu^T |w_{\et}^T|^2(\cdot,\ta) \|_{L^{\infty}(\R)} \leq O(\vep) + O(\delta^{-1})e^{-k \ta}.\ee
Through similar arguments, and appealing to \eqref{exp conv mu}, we also obtain
\be \label{th<2 sxpaok5} \limsup_{T\to \infty}  \|h^T(\cdot,\ta)\|_{L^{\infty}(\R)} \leq O(\vep) + O(\delta^{-1})e^{-k \ta}, \ee
where
\be h^T(\et,\ta)= F(\mu^T(\et,\ta))-F(M_a(\et))- \left(R_a -\frac{\al(1-\al)}{2}\et^2 \right)(\mu^T(\et,\ta)-M_a(\et)).\ee
Letting $\ta \to \infty$ and then $\vep \to 0$ in \eqref{th<2 sxpaok4}, we obtain the second equality in \eqref{capomcqdc}. For the first equality, we argue as follows. We recall, from \eqref{quant pos dcoaia}, that
\be h^T(\et,\ta) = \begin{cases} F(\mu^T(\et,\ta))-F(M_a(\et))- F'(M_a)(\mu^T(\et,\ta)-M_a(\et)) & M_a(\et)>0, \\
F(\mu^T(\et,\ta))+\left(\frac{\al(1-\al)}{2}\et^2-R_a \right)\mu^T(\et,\ta) & M_a(\et)=0.
\end{cases}\ee
In particular, $h^T(\et,\ta)\geq g(\mu^T(\et,\ta),M_a(\et))\geq 0,$ where
\be g(x,y)=F(x)-F(y)-F'(y)(x-y), \quad (x,y)\in \left[0,\infty\right)\times \left[0,\infty \right).\ee
Note that $g$ satisfies $g(x,y)=0$ if and only if $x=y$ (by the convexity of $F$).  Letting $\ta \to \infty$ and then $\vep \to 0$ in \eqref{th<2 sxpaok5}, we get
\be \lim_{\ta \to \infty} \limsup_{T\to \infty}  \|g(\mu^T(\cdot,\ta),M_a(\cdot))\|_{L^{\infty}(\R)}=0.\ee
It therefore follows from the boundedness of $\mu^T$ and $M_a$, and the continuity of $g$, that
\be \lim_{\ta \to \infty}\limsup_{T\to \infty}   \| \mu^T(\cdot,\ta)-M_a(\cdot)\|_{L^{\infty}(\R)}=0.\ee
\end{proof}

\subsection{The subcritical range $\theta \in (2,\infty) $}

In the previous subsection, we treated the case $\theta \in (0,2)$ by exploiting quantitative results about the rate of convergence for the infinite horizon problem. When $\theta \in \left(2,\infty\right)$, no such rates are available, so we will rely instead on well-posedness of \eqref{mfginfhor}. Indeed, we will show that, for this range, there exists exactly one solution to \eqref{mfginfhor} that vanishes as $t \to \infty$. In particular, this will imply that the  subsequential limits of Proposition \ref{prop: existence subseq} are all equal, which is sufficient to conclude.
\begin{thm} \label{thm: fin theta>2} Assume that $m_0$ satisfies \eqref{m0 assum} and \eqref{m0 C1,1}, that  \eqref{kap size} holds, and that $\theta > 2$. There exists a unique solution $(u,m)\in W^{1,\infty}(\R \times (0,\infty))\times C(\R \times [0,\infty))$ to \eqref{mfginfhor}, in the sense of Definition \ref{def: sol}, that satisfies 
\be \label{u(infty)=0 th>2} \lim_{t \to \infty} \|u(\cdot,t)\|_{L^{\infty}(\R)}=0. \ee Furthermore, $(u,m)$ is also characterized uniquely as the limit
\be \label{limit char th>2} (u,m)=\lim_{T \to \infty} (u^T,m^T) \,\,\,\,\text{ in } \,\,\, C^1_{\emph{loc}}(\R \times (0,\infty))\times C_{\emph{loc}}(\R \times \left[0,\infty\right)),\ee
where $(u^T,m^T)$ is the solution to \eqref{mfg}--\eqref{tc},
and $(u,m)$ satisfies, for every $p\in [1,\infty]$,
\be \label{conv result 1 theta>2} \lim_{t \to \infty}  t^{\al(1-\frac1p)}\left\| m(\cdot,t)- \mathcal{M}_a(\cdot,t) \right\|_{L^{p}(\R)} =0,
\ee
\be \label{conv result 2 theta>2} \lim_{t \to \infty}  t^{2-\al(1+\frac1p)}\left\| m(\cdot,t)\left|u_x(\cdot,t)- (\mathcal{U}_a)_x(\cdot,t)\right|^2 \right\|_{L^{p}(\R)} =0,  \ee
\be \label{conv result 3 theta>2} \lim_{t \to \infty}  t^{2-\al(1+\frac1p)}\left\| m(\cdot,t)\left|u_t(\cdot,t)- (\mathcal{U}_a)_t(\cdot,t)\right|\right\|_{L^{p}(\R)} =0. \ee
where $a=\int_{\R}m_0>0$.\end{thm}
\begin{proof} We first show existence, and begin by explaining why, for this subcritical range of $\theta$, $(u^T,m^T)$ satisfy certain uniform estimates up to $t=0$ (in addition to the uniform estimates away from $t=0$ that were used in Proposition \ref{prop: existence subseq}, for arbitrary $\theta$). We recall that, since $\theta>2$, $2\al-1=(2-\theta)/(2+\theta)$ satisfies
\be \label{2al-1 neg askpax} 2\al-1<0. \ee
Let $(v^T,\mu^T)$ be the continuous rescaling \eqref{v mu defi} for $(u^T,\mu^T)$. By Proposition \ref{prop: v Dv bd}, we have
$\|v^T(\cdot,\ta) \|_{L^{\infty}(\R))}=O(1)$ for $\ta \geq 0$, which is to say
\be \label{UTBDSDCM} \|u^T(\cdot,t) \|_{L^{\infty}(\R)} = O(t^{2\al-1}), \quad t \geq 1. \ee
We infer from \eqref{2al-1 neg askpax} that, in particular, $\|u^T(\cdot,t) \|_{L^{\infty}(\R)} \to 0$ as $t \to \infty$, uniformly in $T$. The functions $m^T$ are uniformly bounded, by Proposition \ref{prop: smoothing effect}. Moreover, \eqref{UTBDSDCM} implies, in particular, that $u(\cdot,1)$ is bounded. Thus, as in \eqref{u compar scoaca}, we conclude by the comparison principle that $u$ is bounded on $\R \times [0,1]$. In conjunction with \eqref{UTBDSDCM}, this yields
\be \label{u^T bd [0,1] th>2} \|u^T\|_{L^{\infty}(\R \times [0,T])}=O(1).\,\, \ee
In turn, by Remark \ref{rem: ux bd [0,1]}, as well as \eqref{grad local bd}, we conclude that
\be \|u^T\|_{W^{1,\infty}(\R \times [0,T])}=O(1).\,\,\ee
On the other hand, it follows from \cite[Prop. 3.13, Prop. 3.15]{CMP} that the modulus of continuity for $m^T$ near $t=0$ depends only on $\|u^T_x\|_{L^{\infty}(\R \times (0,T) )}$ and $\|m_0^{\theta}\|_{W^{1,\infty}(\R)}$. Hence, $m^T$ is equicontinuous up to $t=0$ (the equicontinuity away from $t=0$ was shown already in Proposition \ref{prop: existence subseq}).

Let $(u,m)$ be any subsequential limit of $(u^T,m^T)$, given by Proposition \ref{prop: existence subseq}. Note that in particular, said Proposition implies that $(u,m)$ satisfies \eqref{Sec5 conv result 1}--\eqref{Sec5 conv result 2}. By the above discussion, uniform convergence holds up to $t=0$, and  $(u,m)\in W^{1,\infty}(\R \times (0,\infty))\times C(\R \times [0,\infty)).$ By \eqref{UTBDSDCM}, $\lim_{t \to \infty} \|u(\cdot,T)\|_{L^{\infty}(\R)}=0$. 

We now show the uniqueness. Note that this will automatically imply that there is only one possible subsequential limit, and that convergence holds for the full family $(u^T,m^T)$. Let $(\util,\mtil)\in W^{1,\infty}(\R \times (0,\infty))\times C(\R \times [0,\infty))$ be another solution to \eqref{mfg} satisfying $\lim_{t \to \infty} \|\util(\cdot,t)\|_{L^{\infty}(\R)}=0$. The first step is the standard Lasry-Lions proof of uniqueness for MFG, which yields, for $t\in (0,\infty),$
\be \intr(m(\cdot,t)-\mtil(\cdot,t))(u(\cdot,t)-\util(\cdot,t)) + \int_{0}^t \intr \left(\frac{1}{2}(m+\mtil)|u_x-\util_x|^2 + (m^{\theta}-\mtil^{\theta})(m-\mtil)\right)=0. \ee
Letting $t \to \infty$, and using the decay of $u$ and $\util$, we conclude that $\mtil=m$. Since we want to show global uniqueness of $u$, even in the set $\{m=0\}$, we finish the proof by an optimal control argument.  Since we showed that $m_1=m$, $\util$ is a Lipschitz viscosity solution of 
\be -\util_t + \frac12 \util_x^2=m^{\theta}, \,\,\, \text{ in } \,\,\, \R \times (0,\infty),\ee
so it must satisfy the dynamic programming principle \eqref{dpp}. Namely, for every $(x,t)\in \R \times [0,\infty),$ and every $t_1>t$,
\be \label{dpp sacag} \util(x,t)=\inf_{\beta \in H^1((t,t_1)), \,\beta(t)=x} \quad\int_t^{t_1} \left( \frac{1}{2} |\Dot \beta(s)|^2+m(\beta(s),s)^{\theta} \right) ds + \util(\beta(t_1),t_1).
 \ee
The idea is now to use \eqref{u(infty)=0 th>2} and \eqref{2al-1 neg askpax} to let $t_1 \to \infty$ in this formula.
On one hand, \eqref{dpp sacag} implies that, for any $\beta \in H^1((t,\infty))$,
\be \util(x,t) \leq \int_t^{t_1} \left( \frac{1}{2} |\Dot \beta(s)|^2+m(\beta(s),s)^{\theta} \right)ds + \util(\beta(t_1),t_1), \ee
so that, letting $t_1 \to \infty$,
\be \label{dpp <= inf akpo} \util(x,t) \leq \inf_{\beta \in H^1((t,t_1)), \,\beta(t)=x}  \int_t^{\infty} \left( \frac{1}{2} |\Dot \beta(s)|^2+m(\beta(s),s)^{\theta} \right) ds. \ee
On the other hand, in view of \eqref{alpha theta} and \eqref{smoothing effect}, we have $m^{\theta}(x,t)=O(t^{-\al \theta})=O(t^{2\al-2}).$ Thus, choosing a minimizer $\beta \in H^1((t,t_1))$ for \eqref{dpp sacag}, and extending it continuously to be constant in $\left[t_1,\infty\right)$ to get a competitor for the right hand side of \eqref{dpp <= inf akpo}, we obtain 
\begin{multline} \label{dpp >= inf akpo} \util(x,t)=\int_t^{t_1} \left( \frac{1}{2} |\Dot \beta(s)|^2+m(\beta(s),s)^{\theta} \right) ds + \util(\beta(t_1),t_1) \\ 
= \int_{t}^{\infty}\left( \frac{1}{2} |\Dot \beta(s)|^2+m(\beta(s),s)^{\theta} \right) ds  - O(1) \int_{t_1}^{\infty}s^{2\al-2}ds + \util(\beta(t_1),t_1)\\ \geq \inf_{\beta \in H^1((t,t_1)), \,\beta(t)=x}  \int_t^{\infty} \left( \frac{1}{2} |\Dot \beta(s)|^2+m(\beta(s),s)^{\theta} \right) ds -O(t_1^{2\al-1}) + o(1) \end{multline}
as $t_1\to \infty$. 
We conclude from \eqref{dpp <= inf akpo} and \eqref{dpp >= inf akpo} that
\be \util(x,t) = \inf_{\beta \in H^1((t,t_1)), \,\beta(t)=x}  \int_t^{\infty} \left( \frac{1}{2} |\Dot \beta(s)|^2+m(\beta(s),s)^{\theta} \right) ds,\ee
which uniquely characterizes the value function in terms of $m$, and, in particular, implies that $\util=u$.
\end{proof}
\begin{cor}\label{cor:main th>2} The statement of Theorem \ref{thm: main} holds when $\theta<2$.    
\end{cor}
\begin{proof}By \eqref{mu suppt bd}, for any fixed $t$, the support of $m^T(\cdot,t)$ is compact and remains bounded as $T\to \infty$. Thus, the result follows from \eqref{limit char th>2}, \eqref{conv result 1 theta>2}, and \eqref{conv result 2 theta>2}.
\end{proof}
\subsection{The critical case $\theta=2$}
When $\theta=2$, the quantitative results of Proposition \ref{prop: exponential theta<2} break down. The global uniqueness results from the previous section fail as well, because the limiting solution $u$ may, in principle, be unbounded both as $t\to 0^+$ and as $t \to \infty$ (see \eqref{Uself} and \eqref{expl u bd ascoac}). 

Our strategy here will be to insist on uniquely characterizing the subsequential limits of $(u^T,m^T)$, at least in the region $\{m>0\}$, through a partial well-posedness result. The unbounded behavior of $u$ as $t\to \infty$ will be addressed by exploiting the unique scaling of the case $\theta=2$, which is precisely the correct one to leverage the convergence results of Section \ref{sec: convergence IH} in the uniqueness argument. 
The lack of regularity up to $t=0$, on the other hand, will be addressed by weakening the definition of solution, as follows (cf. \cite{Carda1,CMS,OPS}). We denote by $\cP_1(\R)$ the space of probability measures on $\R$ which have a finite first order moment, equipped with the Monge-Kantorovich distance $d_1$. 
\begin{defn} \label{def:weak sol} We say that $(u,m) \in W^{1,\infty}_{\text{loc}}(\R \times (0,\infty))\times C(\R \times (0,\infty))$ is a weak solution to \eqref{mfginfhor} if $m\in C([0,T],\cP_1(\R))$, the first equation holds in the viscosity sense, the second equation holds in the distributional sense, and $m(\cdot,0)=m_0$ holds in the $d_1$ sense.    
\end{defn}
We can now show the main convergence result for the critical case.

\begin{thm} \label{thm: fin theta=2} Assume that $m_0$ satisfies \eqref{m0 assum} and \eqref{m0 C1,1}, that \eqref{kap size} holds, and that $\theta = 2$. There exists a weak solution $(u,m)\in W^{1,\infty}_{\emph{loc}}(\R \times (0,\infty))\times C(\R \times (0,\infty))$ to \eqref{mfginfhor} such that, as $T\to \infty$,
\be \label{mT lim th=2}m^T \to  m \,\, \text{ uniformly on compact subsets of } \,\, \R \times \left[0,\infty \right), \ee 
\be \label{DuT lim th=2} Du^T \to Du \,\, \text{ uniformly on compact subsets of }\,\,\,\{(x,t)\in \R \times (0,T) : x\in \emph{supp}(m(\cdot,t))\},\ee
where $(u^T,m^T)$ is the solution to \eqref{mfg}--\eqref{tc}. The function $(u,m)$ satisfies, for every $p\in [1,\infty]$, and $a=\int_{\R}m_0>0$, \eqref{conv result 1 theta>2}--\eqref{conv result 3 theta>2}.

Furthermore, $(u,m)$ is the unique weak solution to \eqref{mfginfhor} satisfying the natural growth assumptions of Proposition \ref{prop: convergence result}, in the following sense: $m$ is unique, and $u$ is unique up to a constant in the set $\{m>0\}$.
 \end{thm}
\begin{proof} As in the proof of Proposition \ref{prop: main th<2}, since the conclusion is invariant under vertical translations of $u^T$, we may redefine $u^T$ according to \eqref{u defi cases}.  Let $(u,m)$ be a subsequential limit for $(u^T,m^T)$ as given by Proposition \ref{prop: existence subseq}. The fact that $(u,m)$ satisfies \eqref{conv result 1 theta>2}--\eqref{conv result 3 theta>2} follows from Proposition \ref{prop: existence subseq}. Note that, for each fixed $t>0$, the family of measures $\{m^T(\cdot,t)\}$ is tight, since, by \eqref{mu suppt bd}, their supports are all contained in a compact set (which depends on $t$). Therefore, to show that $(u,m)$ is a weak solution to \eqref{mfginfhor}, it is then sufficient to show that $\{m^T\}$ is equicontinuous as $T\to \infty$ in the Monge-Kantorovich distance $d_1$. For this purpose, we use the standard fact (see \cite[Lem. 2.3]{CaLaLiPo}) that 
\be  \label{H structure} \frac{d}{dt} \intr \left(\frac12 m^T(u^T_x)^2(x,t)- \frac{(m^T)^{\theta+1}(x,t)}{\theta+1}\right)dx=0.\ee
 Formally, \eqref{H structure} is obtained by multiplying the two equations of \eqref{mfg}--\eqref{tc} by $u_t$ and $m_t$, respectively, adding the two results, and then integrating in space. In the present setting, these formal computations can be readily justified exactly as those of Lemma \ref{lem: E basic}, with the necessary integrability being given by Proposition \ref{prop: energy rates}, so we omit the details. 

Since $m^T$ and $(u^T)_x(\cdot,1)$ are uniformly bounded (see \eqref{smoothing effect} and \eqref{grad local bd}), it follows from \eqref{H structure} that \be \intr \frac12 m u_x(\cdot,t)^2dx=O(1). \quad t\in [0,T].\ee 
The point of this estimate is that it does not degenerate near $t=0$ like, for example, \eqref{grad local bd} does. Thus, from the second equation of \eqref{mfg}, for every $0\leq t_1<t_2 \leq T$, and every $\vfi \in \text{Lip}_1(\R),$ 
\begin{multline}    
\intr \vfi(x) (m^T(x,t_2)-m^T(x,t_1))dx = \intr m^T u^T_x \vfi_x \leq \int_{t_1}^{t_2}\intr \frac12 (m^T (u^T_x)^2 + m^T)dxdt
\\ = \int_{t_1}^{t_2}(O(1)+a/2)dt=O(1)(t_2-t_1), \end{multline} 
so that \be d_1(m^T(\cdot,t_2),m^T(\cdot,t_1))= O(1)(t_2-t_1), \ee
as wanted. This shows the existence of a weak solution with the required properties, and that every subsequential limit of $(u^T,m^T)$ gives such a solution.

It remains to show uniqueness, as well as the limits \eqref{mT lim th=2} and \eqref{DuT lim th=2}. 
Let $(u,m)$ and $(\util,\mtil) \in  W^{1,\infty}_{\text{loc}}(\R \times (0,\infty))\times C(\R \times (0,\infty))$ be two weak solutions to \eqref{mfgi} that satisfy, for some $t_0>0$, the assumptions of Proposition \ref{prop: convergence result}. As in the proof of Theorem \ref{thm: fin theta>2}, we intend to adapt the Lasry-Lions uniqueness proof. The two main difficulties are: first, $u$ is not assumed to have any regularity up to $t=0$, and second, $u$ does not decay as $t\to \infty$.

To treat the lack of regularity at $t=0$, we may directly follow an argument developed originally in \cite{OPS}, that applies for a very general class of weak solutions. The main point is the following: the one-sided bound $u_t=\frac12{u_x^2}-m^{\theta} \geq -m^{\theta} \geq -C$ guarantees the existence of a (possibly infinite) measurable function $u(\cdot,0^+)$ such that $\lim_{t \to 0^+}u(x,t)=u(x,0^+)\in \R \cup \{-\infty\}$. It is then not difficult to show that $u(\cdot,0^+)\in L^1(dm_0)$, and this modest regularity can be leveraged to make the Lasry-Lions computation work up to $t=0$  (we refer to \cite[Thm. 5.2]{CMP} for the details). Namely, one obtains, for $t>0$,
\begin{multline} \label{kapsda th2} \intr(m(\cdot,t)-\mtil(\cdot,t))(u(\cdot,t)-\util(\cdot,t)) + \int_{0}^t \intr \left(\frac{1}{2}(m+\mtil)|u_x-\util_x|^2 + (m^{\theta}-\mtil^{\theta})(m-\mtil)\right)=0 . \end{multline}
Finally, we address the matter of letting $t\to \infty$ in \eqref{kapsda th2}. For this purpose, we take advantage of the asymptotic results of Section \ref{sec: convergence IH}, which (only in the case $\theta=2$) have the correct scaling to be useful in the uniqueness proof.
Let $(w,\mu)$ and $(\tilde{w},\tilde{\mu})$ be the corresponding continuous rescalings for $(u,m)$ and $(\util, \mtil)$, defined according to \eqref{v mu defi IH} and \eqref{w defi}. We claim that the function 
\be f(\ta)= \intr w(\eta,\ta) (\mu(\eta,\ta)-M_a(\et))d\et \ee
satisfies $ \lim_{\ta \to \infty} f(\ta)=0. $
Indeed, by Lemma \ref{lem:lasrylionstheta=2}, $f$ is bounded and non-increasing, so it has a limit $f(\infty)$. We argue that $f(\infty)=0$. Indeed, if $R$ is as in Proposition \ref{prop: convergence result}, then
\be |f(\ta)|=\left|\intr \left(w(\et,\ta)- (2R)^{-1}\int_{-R}^R w(y,\ta)dy\right)(\mu-M_a)d\et\right|\leq \|w_{\et}\|_{L^{\infty}(-R,R)}\|\mu(\cdot,\ta)-M_a(\cdot)\|_{L^{1}(\R)},\ee
and the right hand side tends to $0$ as $\ta \to \infty$ by \eqref{IHH conv result 1} and Lemma \ref{lem:reduc to cont resc}. The same arguments apply to $(\tilde{w},\tilde{\mu}),$ so we conclude that 
\be \label{odpkqw th2} \lim_{\ta \to \infty} \intr w(\eta,\ta) (\mu(\eta,\ta)-M_a(\et))d\et =\lim_{\ta \to \infty} \intr \tilde{w}(\eta,\ta) (\tilde{\mu}(\eta,\ta)-M_a(\et))d\et=0. \ee
The key point now is that, in the present critical case $\theta=2$, the time scaling factor of \eqref{odpkqw th2}, when rewritten as a statement in $(u,m)$, is $t^{2\al-1}=t^0=1$.  Namely, \eqref{v mu defi} yields
\be \intr w(\et,\ta) (\mu(\eta,\ta)-M_a(\et))d\et=t^{2\al-1}\intr u(x,t)(m(x,t)-\mathcal{M}_a(x,t)))dx, \ee
so that, since, $2\al-1=(2-\theta)/(2+\theta)=0$, \eqref{odpkqw th2} becomes
\be \label{csmap th2} \lim_{t \to \infty} \intr u(x,t) (m(x,t)-\mathcal{M}_a(x,t)))dx =\lim_{\vep \to \infty} \intr \util(x,t) (\mtil (x,t)-\mathcal{M}_a(x,t)))dx=0. \ee
Hence, writing
\be \intr(m(\cdot,t)-\mtil(\cdot,t))(u(\cdot,t)-\util(\cdot,t))\\
= \intr(m(\cdot,t)-\mathcal{M}_a(\cdot,t))u(\cdot,t)- \intr(\mtil(\cdot,t)-\mathcal{M}_a(\cdot,t))\util(\cdot,t)),\ee
we may use \eqref{csmap th2} to let $t \to \infty$ in \eqref{kapsda th2}, obtaining
\be \label{LL askdx th2}  \int_{0}^{\infty} \intr \left(\frac{1}{2}(m+\mtil)|u_x-\util_x|^2 + (m^{\theta}-\mtil^{\theta})(m-\mtil)\right)=0, \ee 
so that $\mtil=m$, and $u_x=\util_x$ in $\{m>0\}$. The first equation of \eqref{mfginfhor} then implies that $u_t=\util_t$ in $\{m>0\}$. We showed earlier that the subsequential limits of $(u^T,m^T)$ are weak solutions, so, since $m$ is unique and $Du$ is unique on $\{m>0\}$, \eqref{mT lim th=2} and \eqref{DuT lim th=2} hold. In particular, by \eqref{freebd ch}, the support of $m$ is connected, so $u$ and $\util$ differ by a constant on $\{m>0\}$.
\end{proof}
\begin{cor}\label{cor:main th=2} The statement of Theorem \ref{thm: main} holds when $\theta=2$.    
\end{cor}
\begin{proof}By \eqref{mu suppt bd}, for any fixed $t$, the support of $m^T(\cdot,t)$ is compact and remains bounded as $T\to \infty$. Thus, the result follows from \eqref{mT lim th=2}, \eqref{DuT lim th=2}, and \eqref{conv result 1 theta>2}--\eqref{conv result 3 theta>2}.
\end{proof}

\section{The planning problem} \label{sec:plan}
In this final section, we explain the necessary modifications to the preceding arguments in order to prove the main results for the case of the planning problem \eqref{mfg}--\eqref{pp}. The changes needed are relatively small. A minor point is that the estimates of Theorem \ref{thm.intro2} now involve the function $\mathscr{d}$ defined by \eqref{scrd defi}, which means that some of the estimates that held up to the terminal time now degenerate near $t=T$. This is not an issue, however, since our results concern the intermediate asymptotic behavior in which $t << T$. Besides this, there are two main difficulties.

The first difficulty is the fact that the free boundary curves are no longer known to be monotone. This issue is readily addressed in Lemma \ref{lem: plan early monotone} below, by showing that monotonicity still holds in the interval $[0,\sigma T],$ where $\sigma>0$ is a small constant.

The second difficulty is that, while the function $u$ in \eqref{mfg}--\eqref{tc} is globally unique, $u$ in \eqref{mfg}--\eqref{pp} is only unique (up to a constant) in the set $\{m>0\}$. For this reason, the optimal trajectories outside the support of $m$ cannot be characterized in this case as they were in Proposition \ref{prop trajec}. Instead, we must argue that, given any solution $(\tilde{u},m)$ there exists a sufficiently well-behaved solution $(u,m)$ that agrees with $\tilde{u}$ in the support of $m$. This will be explained in Proposition \ref{prop: plan estimates} below. 
\begin{lem} \label{lem: plan early monotone}  Assume that $m_0$ satisfies \eqref{m0 assum} and \eqref{m0 C1,1}, and $m_T$ satisfies \eqref{mT assum} and $\eqref{mT bump}$. Let $(u,m)$ be a solution to \eqref{mfg}--\eqref{pp}, and let $\gamma_L$, $\gamma_R$ be its free boundary curves. Then there exists \be K=K(C_0,C_1,|a_0|,|b_0|,|a_1|,|b_1|)\ee  such that \eqref{smoothing effect}, \eqref{gamL}, \eqref{gamL'} and \eqref{gamL''} hold for $t\in [0,T/2]$. Furthermore, there exists $\sigma\in (0,1/4)$, with $\sigma^{-1}=\sigma^{-1}(C_0,C_1,|a_0|,|b_0|,|a_1|,|b_1|)$, such that, for $T>\max(|a_0|^{1/\al},|b_0|^{1/\al},|a_1|^{1/\al},|b_1|^{1/\al},\sigma^{-1}),$ we have

 \be \label{plan monot near 0} \gamma_L'(t)<0 \text{ and } \gamma_R'(t)>0, \quad t\in [0,\sigma T],\ee
 \be  \label{plan monot near T}\gamma_L'(t)>0 \text{ and } \gamma_R'(t)<0, \quad t\in [(1-\sigma) T,T]. \ee
\end{lem}
\begin{proof} 
As usual, it is enough to show the statements for $\gamma_L$. The proof of \eqref{smoothing effect}, \eqref{gamL}, \eqref{gamL'} and \eqref{gamL''} follows exactly as in Propositions \ref{prop: smoothing effect} and \ref{prop: free bd rates}, noting that $\mathscr{d}(t)=t$ for $t \in [0,T/2]$, and that the monotonicity of $\gamma_L$ and $\gamma_R$ was not used in these proofs. In fact, by a symmetric argument, one sees that the following holds as well:
\be \label{gam'' plan aoisjc} |\gamma_L(t)|\leq C(1+(T-t)^{\al}), \quad \Ddot \gamma_L(t) \geq \frac{1}{K(1+(T-t)^{2-\al})}, \quad t\in [T/2,T], \ee
Assume that $\sigma T\geq 1,$ and, by contradiction, assume that $\dot \gamma_L(\sigma T) \geq 0$. Since $\gamma_L$ is convex, we must then have $\dot \gamma_L(t) \geq 0$ for $t \in [\sigma T,T]$. Thus, by \eqref{gamL}, allowing the constant $C$ to increase at each step, we have

\be \label{plan early asodao} \dot \gamma_L(T/2) = \int_{\sigma T}^{T/2}\Ddot \gamma_L(t)\geq \frac{1}{C}\int_{\sigma T}^{T/2} t^{\al-2}dt  = \frac{1}{C(1-\al)}T^{\al-1}((\sigma^{-1})^{1-\al}-2^{1-\al})\ee
Since $\gamma_L$ is convex and non-decreasing on $[\sigma T,T]$, and recalling that $\al \in (0,1)$, it then follows from \eqref{gamL} and \eqref{plan early asodao} that
\begin{multline}    
\gamma_L(T) \geq \gamma_L(\sigma T)+ (T/2)\dot \gamma_L(T/2) \geq -C(\sigma T)^{\al}+ \frac{1}{2C(1-\al)}T^{\al}\left(\left(\sigma^{-1}\right)^{1-\al}-2^{1-\al}\right)\\
=T^{\al}\left(\frac{1}{2C(1-\al)}\left(\left(\sigma^{-1}\right)^{1-\al}-2^{1-\al}\right)-C\sigma^{\al}\right)\geq T^{\al}\end{multline} 
if $\sigma$ is chosen sufficiently close to $0$, independently of $T$. But $\gamma_L(T)=a_1$ is a fixed constant, so this is impossible as soon as $T^{\al}> a_1$. Hence, for such $T$, $\dot \gamma_L(\sigma T)> 0$, as wanted. This proves \eqref{plan monot near 0}. Similarly,
\eqref{plan monot near T} then follows by a symmetric argument, using \eqref{gam'' plan aoisjc} instead of \eqref{gamL} and \eqref{gamL''}, and the fact that $\gamma_L(0)=a_0$ is also a fixed constant.   
\end{proof}
In the next proposition, we state the a priori estimates for the planning problem which correspond to those of Section \ref{sec: estimates}.
\begin{prop}\label{prop: plan estimates} Assume that $m_0$ satisfies \eqref{m0 assum} and \eqref{m0 C1,1}, and $m_T$ satisfies \eqref{mT assum} and $\eqref{mT bump}$. For any solution $(\tilde{u},m)$ to \eqref{mfg}--\eqref{pp}, there exists a solution $(u,m)$ such that $u$ agrees with $\util$ on $\{m>0\}$, and the following holds. For any $\delta>0$ and any $\vep \in (0,1)$, there exist constants $K>0$, $K_1>0$, $\sigma \in (0,1/4)$,  with $ \sigma^{-1}$ $ =\sigma^{-1}(C_0,C_1,|a_0|,|b_0|,|a_1|,|b_1|),$
$K$$=K(C_0,C_1,|a_0|,|b_0|,|a_1|,|b_1|)$, $K_1$$=(C_0,C_1,|a_0|,|b_0|,|a_1|,|b_1|,\theta,\theta^{-1},\delta^{-1})$, such that, for $T\geq K,$ the following holds:
\be \label{plan m gamL bds}\text{\eqref{smoothing effect}, \eqref{gamL}, \eqref{gamL'}, and \eqref{gamL''} hold for $t\in (1,T/2),$} \ee 
\be \label{oscu local bd plan}\underset{S\left(m,\frac{t_0}{2},2t_0\right)}{\emph{osc}} (u)\leq  
  K_1 t_0^{2\al-1}, \quad t_0\in [\delta,T/4],\ee
  \be \label{plan Du comp spt} Du \equiv 0 \text{ outside of $[-K(1+T^{\al}),K(1+T^{\al})].$}\ee
\be \label{plan grad energy bds}\text{\eqref{grad local bd}, and \eqref{m energy bd} hold for $t_0 \in [0, \sigma T]$, with constant $K_1.$} \ee 
Furthermore, there exists $s\in (0,1)$, depending only on $C_0,C_1,|a_0|,|b_0|,|a_1|, |b_1|,\delta^{-1}$, such that, for every $\overline{t} \in [2\delta,\sigma T]$, and every $(x_0,t_0),(x_1,t_1) \in S(m,\overline{t}/2,2\overline{t}),$
\be \label{p conti plan} |m^{\theta}(x_1,t_1)-m^{\theta}(x_0,t_0)| \leq Ct_0^{-\alpha \theta} (|(x_1-x_0)t_0^{-\al}|^{s}+|(t_1-t_0)t_0^{-1}|^{s}).
\ee    
\end{prop}
\begin{proof} We begin by noting that \eqref{plan m gamL bds} was already shown in Lemma \ref{lem: plan early monotone}. The oscillation bound \eqref{oscu local bd plan} follows exactly as in Proposition \ref{prop:osc u}, with the only difference being the shorter range $t_0\in [0,T/4]$, which corresponds to the shorter range $[0,T/2]$ in \eqref{plan m gamL bds}.

The proof of the gradient rates \eqref{grad local bd} given in Section \ref{subsec: grad bds} had two independent components: the first was estimating $|u_x|$ outside the support of $m$ by characterizing the optimal trajectories (Proposition \ref{prop trajec} and Corollary \ref{cor grad bd m=0}), and the second was estimating $|Du|$ inside $\{m>0\}$ by using the elliptic equation satisfied there by $u$ (Proposition \ref{prop: grad rate}). The latter component of the proof may be repeated with no change, since the terminal condition was not used. Again, the only difference is in the shorter range $t\in [0,\sigma T]$, to ensure that \eqref{plan m gamL bds} and \eqref{oscu local bd plan} are valid (in fact, for this part of the proof inside the support, taking $\sigma<1/4$ would suffice). Thus, to show \eqref{grad local bd} for $t\in [\delta,\sigma T]$, we only need to prove that, for $(x,t)\in (0,\sigma T)$ such that $m(x,t)=0$, we have
\be \label{grad out local dsaoij} |u_x(x,t)|\leq \max(\dot \gamma_L(t),\dot \gamma_R(t)).\ee
This is the step where the planning problem differs the most from \eqref{mfg}--\eqref{tc}. Indeed, we may no longer follow the approach of characterizing the trajectories uniquely, since $u$ may not be uniquely defined outside of $\{m>0\}$. Instead, we explain how to construct a $C^1$ solution that satisfies \eqref{grad out local dsaoij}, by redefining $u$, if necessary, outside of the support of $m$ (while preserving the fact that $(u,m)$ is a global solution to \eqref{mfg}--\eqref{pp}). By symmetry, it is sufficient to explain how to construct the values of $u(x,t)$ in the region  $\Omega=\{(x,t)\in \R \times [0,T]:x\leq \gamma_L(t)\}$.
We require $\sigma$ to be less than the corresponding constant $\sigma$ of Lemma \ref{lem: plan early monotone}. By the strict convexity of $\gamma_L$, there exists a unique time $t^*\in [\sigma T, (1-\sigma)T]$ such that $\gamma_L'(t^*)=0$. Indeed, $t^*$ is the global minimum point of $\gamma_L$. The region $\Omega$ is then subdivided into four subregions with non-overlapping interiors:
\be \Omega_1=\{0\leq t\leq t^*, \, \gamma_L(t^*)\leq x \leq \gamma_L(t)\}, \quad \Omega_2=\{0 \leq t\leq t^*, \,\, x\leq \gamma_L(t^*)\}, \ee
\be \Omega_3= \{t^*\leq t \leq T, \, \gamma_L(t^*)\leq x \leq \gamma_L(t)\},\quad  \Omega_4=\{t^*\leq t \leq T, \,\, x\leq \gamma_L(t^*)\}. \ee
We extend $u$ piece by piece from $\text{supp}(m)$ to $\text{supp}(m)\cup \Omega$, starting with $\Omega_1$. Consider the family $\{l_t\}$ of line segments \be l_t(s)=(\gamma_L(t)+(s-t)\gamma_L'(t),s), \quad s\in [0,t].\ee
Then, by the convexity and monotonicity of $\gamma_L$, it follows that $\Omega_1$ is the disjoint union of these line segments. Recalling that $\gamma_L'(t)=-u_x(\gamma_L(t),t),$ it follows that, through the method of characteristics, one may define the values of $u$ on $\Omega_1$ along the line segments $l_t(\cdot)$ in terms of the values of $u(\gamma_L(t),t)$ and $u_x(\gamma_L(t))$. Note that the regions $\Omega_1$ and $\Omega_2$ overlap on the vertical segment $\{(\gamma_L(t^*),t):t\in [0,t^*]\}$, where $u_x\equiv u_t \equiv \gamma_L'(t^*)=0$. Hence, extending $u$ continuously to be constant on $\Omega_2$ produces a $C^1$ solution to the HJ equation on $\Omega_1 \cup \Omega_2$. Performing the same construction on $\Omega_3$, and then extending $u$ to be constant on $\Omega_4$, we get a $C^1$ extension of $u$ to all of $\text{supp}(m) \cup \Omega$. The key feature of this solution is that, by construction, at any $(x,t)\in \Omega_1$, we have $u_x(x,t)=u_x(\gamma_L(t_1),t_1)=-\gamma_L'(t_1)$ for some $t_1\in [t,t^*]$ and, therefore, using the convexity and monotonicity of $\gamma_L$ on $[0,t^*]$,
\be |u_x(x,t)|\leq |\gamma_L'(t_1)|\leq |\gamma_L'(t)|.\ee
This shows that $u$ satisfies \eqref{grad out local dsaoij}, and hence \eqref{grad local bd} holds globally. Furthermore, \eqref{plan Du comp spt} holds by construction, due to \eqref{gamL}. For the remaining estimates, the proofs proceed with no change, other than shortening the ranges to $[0,\sigma T]$, and halving the value of $\sigma$ whenever necessary. Namely, \eqref{m energy bd} follows with the same proof as Proposition \ref{prop: energy rates}, and \eqref{p conti} follows as in the proof of Proposition \ref{prop:holder m}.
\end{proof}
We now explain how to obtain all of the convergence results for the planning problem, starting with the subcritical case.
\begin{thm} \label{thm: fin theta>2 plan} Assume that $m_0$ satisfies \eqref{m0 assum} and \eqref{m0 C1,1},  $m_T$ satisfies \eqref{mT assum} and $\eqref{mT bump}$, and that $\theta > 2$. Let $(u,m)\in W^{1,\infty}(\R \times (0,\infty))\times C(\R \times [0,\infty))$ be the unique solution to \eqref{mfginfhor}, in the sense of Definition \ref{def: sol}, that satisfies 
\be \label{plan u(infty)=0 th>2} \lim_{t \to \infty} \|u(\cdot,t)\|_{L^{\infty}(\R)}=0. \ee
Let $(u^T,m^T)$ be the solution to \eqref{mfg}--\eqref{pp} that satisfies $\intr u^T(\cdot,T/2)m^T(\cdot,T/2)=0$.
Then, up to an adequate choice of $u^T$ on its region of non-uniqueness $\{m^T=0\}$, we have
\be \label{limit char th>2 plan} \lim_{T \to \infty} (u^T,m^T)=(u,m) \,\,\,\,\text{ in } \,\,\, C^1_{\emph{loc}}(\R \times (0,\infty))\times C_{\emph{loc}}(\R \times \left[0,\infty\right)).\ee

\end{thm}
\begin{proof} If needed, we redefine $u^T$ to be the solution given by Proposition \ref{prop: plan estimates}. The estimates for the continuous rescaling $(v^T,\mu^T)$ follow from Proposition \ref{prop: plan estimates} in the same way as in Propositions \ref{prop: mu bd}, \ref{prop: v Dv bd}, \ref{prop: v mu energy}, and \ref{prop: mu holder cont}, with only the following minor changes in the $L^{\infty}$ bound for $v^T$ of Proposition \ref{prop: v Dv bd}. Since $\intr u^T(\cdot,T/2)m^T(\cdot,T/2)=0$, there exists $x_0 \in \text{supp}(m^T(\cdot,T/2))$ such that $u(x_0,T/2)=0$. Then by iterating \eqref{oscu local bd plan} a finite number of times, $u^T(x,\sigma T)=O(T^{2\al-1})$ for all $x\in \text{supp}(m^T(\cdot,\sigma T))$. By the space derivative estimate given by \eqref{plan grad energy bds}, we in turn get \be \|u(\cdot,\sigma T)\|_{L^{\infty}(-K(1+T^{\al}),K(1+T^{\al}))}=O(T^{2\al-1}). \ee 
But then \eqref{plan Du comp spt} implies that \be \|u(\cdot,\sigma T)\|_{L^{\infty}(\R)}=O(T^{2\al-1}). \ee
We can now use the comparison principle on $\R \times [1,\sigma T]$ as in \eqref{u compar scoaca}, and then on $\R \times [0,1]$ as in \eqref{u^T bd [0,1] th>2}, to conclude that
\be \|u^T(\cdot,t)\|_{L^{\infty}(\R)}=O(t^{2\al-1}) \,\, \text{ for } t\in [1, \sigma T], \,\,\,\text{ and } \,\,\, \|u^T\|_{L^{\infty}(\R \times (0,\sigma T))}=O(1).  \ee
Having the estimates for the continuous rescaling, the proof of Theorem \ref{thm: fin theta>2} now goes through with no changes.    
\end{proof}
Next, we discuss the critical case.
\begin{thm} \label{thm: th=2 plan}  Assume that $m_0$ satisfies \eqref{m0 assum} and \eqref{m0 C1,1}, $m_T$ satisfies \eqref{mT assum} and $\eqref{mT bump}$, and that $\theta = 2$. Let $(u,m)$ is the unique weak solution to \eqref{mfginfhor} satisfying the natural growth assumptions of Proposition \ref{prop: convergence result}, which is guaranteed by Theorem \ref{thm: fin theta=2}. Let $(u^T,m^T)$ be a solution to \eqref{mfg}--\eqref{pp}. Then, as $T \to \infty$, \eqref{mT lim th=2}--\eqref{DuT lim th=2} hold. \end{thm}
\begin{proof}
We may assume that $T$ is large enough that the conclusion of Lemma \ref{lem: plan early monotone} holds. Furthermore, the statement is invariant under vertical translations of $u^T$, and under replacement of $u^T$ by any solution that agrees with $u^T$ on the set $\{m^T>0\}$. Hence, we may first redefine $u^T$ outside the support of $m^T$, if necessary, so that the estimates of Proposition \ref{prop: plan estimates} hold for the pair $(u^T,m^T)$. Afterward, we may replace $u^T$ according to \eqref{u defi cases}. Note that this is just translation by a constant, so that $(u^T,m^T)$ is still a solution to \eqref{mfg}--\eqref{pp} satisfying the conclusion of Proposition \ref{prop: plan estimates}. The estimates for the continuous rescaling $(v^T,\mu^T)$, including the $L^{\infty}$ bounds, now follow with no change from \ref{prop: plan estimates}, and we may conclude by following the proof of Theorem \ref{thm: fin theta=2}.
\end{proof}
Finally, the following statement is the general convergence result, valid for any $\theta>0$, and corresponds to Theorem \ref{thm: fin hor}.
\begin{thm} \label{thm: main plan}Assume that $m_0$ satisfies \eqref{m0 assum} and \eqref{m0 C1,1}, and that  $m_T$ satisfies \eqref{mT assum} and $\eqref{mT bump}$. Let $(u^T,m^T)$ be the solution to \eqref{mfg}--\eqref{tc}, and let $\int_{\R}m_0=a>0$. Then, for every $p\in [1,\infty],$ \eqref{conv result 1}--\eqref{conv result 3} hold.
\end{thm}
 \begin{proof} The claim for $\theta>2$ follows from Theorems \ref{thm: fin theta>2} and \ref{thm: fin theta>2 plan}. For $\theta=2$, it follows from Theorems \ref{thm: fin theta=2} and \ref{thm: th=2 plan}. Finally, for $\theta \in (0,2)$, we redefine $u^T$ exactly as in the proof of Theorem \ref{thm: th=2 plan}, so that the proof of Proposition \ref{prop: main th<2} goes through with no change.    \end{proof}
 \subsection*{Acknowledgments} The author would like to thank A. Porretta for valuable comments and suggestions which helped improve the manuscript. The author was partially supported by P. E. Souganidis' NSF grant DMS-1900599, ONR grant N000141712095, and AFOSR grant FA9550-18-1-0494.


\begin{thebibliography}{1}
\bibitem{Bar} G.I. Barenblatt, {\it On some unsteady motions of a liquid or a gas in a porous medium}, Prikl. Mat. Mekh. 16 (1952), 67-78.

\bibitem{CaSo} P. Cannarsa and H. M. Soner, {\it Generalized one-sided estimates for solutions of Hamilton-Jacobi equations and applications}. Nonlinear Analysis: Theory, Methods \& Applications, 13(3) (1989), 305-323.
\bibitem{Carda1}P. Cardaliaguet,  {\it Weak solutions for
first order mean field games with local coupling}, in Analysis and geometry in control theory and its applications, Springer 2015, pp.~111--158.
\bibitem{CaLaLiPo} P. Cardaliaguet, J.-M. Lasry, and P.-L. Lions, A. Porretta. {\it Long time average of mean field games}. Networks
\& Heterogeneous Media, 7(2):279, 2012.
\bibitem{CMS}P. Cardaliaguet, A. R. M\'esz\'aros, F. Santambrogio, {\it First order mean field games with density constraints: pressure equals price}, SIAM J. Control Optim., 54 (2016),  2672-2709.
\bibitem{CMP} P. Cardaliaguet, S. Munoz, A. Porretta, Free boundary regularity and support propagation in mean field games and optimal transport, arXiv:2308.00314, 2023.
\bibitem{CaPo2019} P. Cardaliaguet, A. Porretta. {\it Long time behavior of the master equation in mean field game theory.} Analysis \& PDE, 12(6):1397–1453, 2019.
\bibitem{CiPo} M. Cirant and A. Porretta. {\it Long time behaviour and turnpike solutions in mildly non-monotone mean field games.} ESAIM: Control Optim. Calc. Var., 27, 2021.
\bibitem{degiorgi}E. De Giorgi, {\it Sulla differenziabilit\`a e l’analiticit\`a delle estremali degli integrali multipli regolari.} Mem. Accad. Sci. Torino. Cl. Sci. Fis. Mat. Nat. (3), 3:25--43, 1957.
\bibitem{DiB}E. DiBenedetto, {\it Degenerate parabolic equations}, Springer-Verlag, 1993.

\bibitem{DiB2}E. DiBenedetto, U. Gianazza, Vincenzo Vespri, {\it   Harnack estimates for quasi-linear degenerate parabolic differential equations}, Acta Math. 200 (2008), 181--209.
\bibitem{GaVaBook}V.A. Galatniokov, J.L V{\'a}zquez, {\it A Stability Technique for Evolution Partial Differential Equations. A Dynamical Systems Approach}, PNLDE 56, Birkh\"auser Verlag. 
\bibitem{GilbargTrudinger}D. Gilbarg, N.S. Trudinger, {\it Elliptic partial
differential equations of second order, Springer}, Berlin, 2001, pp.
120-130.
\bibitem{GoSe} D. Gomes, T. Seneci, {\it Displacement convexity for first-order mean-field games}, Minimax Theory Appl. 3 (2018), 261--284.
\bibitem{Hu} G. Huisken, {\it Asymptotic behavior for singularities of the mean curvature flow,} J. Differential Geom. 31 (1990), no. 1, 285–299.
\bibitem{LaSa}H. Lavenant, F. Santambrogio, {\it Optimal density evolution with congestion: $L^\infty$  bounds via flow interchange techniques and applications to variational Mean Field Games}, Commun. Partial. Differ. 43 (2018), 1761-1802.
\bibitem{OPS} C. Orrieri, A. Porretta, G. Savar\'e, {\it A variational approach to the mean field planning problem}, J. Funct. Anal. 277 (2019),  1868-1957.
\bibitem{L-college} P.-L. Lions, {\it Cours at Coll\` ege de France (2009-2010)}, 
https://www.college-de-france.fr/site/pierre-louis-lions/course-2009-2010.html.
\bibitem{L-HJ book} P.-L. Lions, {\it Generalized solutions of Hamilton--Jacobi equations,} Research Notes in Mathematics 69, Pitman, London, 1982. 
\bibitem{McCann} R. J. McCann, {\it A convexity principle for interacting gases},  Adv. Math.  128 (1997),  153-179.
\bibitem{MimikosMunoz} N. Mimikos-Stamatopoulos, S. Munoz, {\it Regularity
and long time behavior of one-dimensional first-order mean field games
and the planning problem}, SIAM J. Math. Anal. 56 (2024), no. 1, 43--78, 2024.
\bibitem{Munoz}S. Munoz, {\it Classical and weak solutions to local first-order
mean field games through elliptic regularity}, Ann. Inst. H. Poincaré
Anal. Non Linéaire 39 (2022), no. 1, 1--39.
\bibitem{Porretta} A. Porretta, {\it Regularizing effects of the entropy functional in optimal transport and planning problems}, J. Funct. Anal. 284 (2023), no. 3.
 \bibitem{Urbano} J. M. Urbano: {\it The method of intrinsic scaling}, Lecture Notes in Mathematics 1930, Springer-Verlag, 2008.
\bibitem{vazquez2007porous}
J. L. V{\'a}zquez, 
\textit{The Porous Medium Equation. Mathematical Theory}, 
Oxford Mathematical Monographs. 
The Clarendon Press, Oxford University Press, Oxford, 2007.
\end{thebibliography}
\end{document}